\documentclass[11pt]{imsart}

\usepackage{subfigure}
\usepackage[font=small,labelfont=bf]{caption}
\usepackage{caption}
\usepackage[letterpaper]{geometry}
\usepackage{amsmath, amsthm, amsfonts, amsbsy, thmtools, amssymb}

\usepackage{thm-restate}
\usepackage{hyperref}
\usepackage{cleveref}
\usepackage[mathscr]{euscript}
\usepackage{tikz}
\usepackage{enumitem}
\usepackage{stmaryrd}
\usepackage{comment}
\usetikzlibrary{arrows}

\makeatletter
\def\namedlabel#1#2{\begingroup
	#2%
	\def\@currentlabel{#2}%
	\phantomsection\label{#1}\endgroup
}
\makeatother

\numberwithin{equation}{section}

\DeclareMathOperator{\Id}{Id}
\DeclareMathOperator{\Real}{Re}		
\DeclareMathOperator{\Imaginary}{Im}

\DeclareMathOperator{\Conf}{Conf}

\def\ASEP{\boldsymbol{\mathcal{A}}}
\def\ASEPB{\boldsymbol{\mathcal{B}}}
\def\ASEPF{\mathcal{F}}
\def\PP{\mathbb{P}}
\def\EE{\mathbb{E}}
\def\Z{\mathbb{Z}}

\def\R{\mathbb{R}}
\def\h{\mathfrak{h}}
\def\bfrho{\boldsymbol{\rho}}
\def\bfeta{\boldsymbol{\eta}}
\def\bfzeta{\boldsymbol{\zeta}}
\def\bfalpha{\boldsymbol{\alpha}}
\def\bfxi{\boldsymbol{\xi}}
\def\bfX{\boldsymbol{X}}
\def\bfY{\boldsymbol{Y}}
\def\bfZ{\boldsymbol{Z}}
\def\bfU{\boldsymbol{U}}
\def\bfK{\boldsymbol{K}}
\def\bfM{\boldsymbol{M}}
\def\llbrace{\{\!\!\{}
\def\rrbrace{\}\!\!\}}

\newtheorem{thm}{Theorem}[section]
\newtheorem{prop}[thm]{Proposition}
\newtheorem{lem}[thm]{Lemma}
\newtheorem{cor}[thm]{Corollary}

\theoremstyle{remark}
\newtheorem{rem}[thm]{Remark}
\theoremstyle{definition}
\newtheorem{definition}[thm]{Definition}

\def\note#1{\textup{\textsf{\color{blue}(#1)}}}

\begin{document}

\begin{frontmatter}

\title{The ASEP speed process}
\runtitle{The ASEP speed process}
\runauthor{Aggarwal, Corwin, Ghosal}

\begin{aug}
\author{\fnms{Amol}  \snm{Aggarwal}\textsuperscript{1}}
 \author{\fnms{Ivan}  \snm{Corwin}\textsuperscript{2}}
 \author{\fnms{Promit} \snm{Ghosal}\textsuperscript{3}}
\address{\textsuperscript{1}Columbia, Department of Mathematics, 2990 Broadway, NY, NY 10027, \\ Clay Mathematics Institute 70 Main St \#300, Peterborough, NH 03458 \\ Institute of Advanced Study, School of Mathematics, 1 Einstein Drive, Princeton, NJ 08540,  amolaggarwal@math.columbia.edu}
\address{\textsuperscript{2}Columbia, Department of Mathematics, 2990 Broadway, NY, NY 10027,  corwin@math.columbia.edu}
\address{\textsuperscript{3}MIT, Department of Mathematics, 182 Memorial Drive, Cambridge, MA 02139,  promit@mit.edu}




\end{aug}

\begin{abstract}
For ASEP with step initial data and a second class particle started at the origin we prove that as time goes to infinity the second  class particle almost surely achieves a velocity that is uniformly distributed on $[-1,1]$. This positively resolves Conjecture 1.9 and 1.10 of \cite{SP} and allows us to construct the ASEP speed process.
\end{abstract}




\end{frontmatter}

\setcounter{tocdepth}{1}
\tableofcontents
\section{Introduction}\label{sec:intro}

Consider ASEP started in step initial data with one second class particle at the origin (see Figure \ref{fig:Traj}). Specifically, at time $t = 0$, each site $j \leq -1$ is occupied with a first class particle, the site $j = 0$ is occupied by a second class particle, and all sites $j > 0$ are initially unoccupied and (for the definition of the dynamics which follows) will be considered infinite class. First and second class particles have left jump rate $L$ and right jump rate $R$ where we assume that $R>L\geq 0$ and $R-L=1$. Jumps are subject to the rule that when a class $k$ particle tries to jump into a site with a class $k'$ particle, the particles switch places if and only if $k<k'$ (otherwise, they stay put). We denote this process by $\ASEP_t=(\bfeta_t,\bfX_t)$ where $\bfeta_t\in \{0,1\}^{\Z}$ are the occupation variables for the first class particles and $\bfX_t$ is the location of the second class particle (we require that $\bfeta_t(\bfX_t)=0$ so there is no first class particle at the site of the second class particle). Initially, $\bfeta_0(j)=\mathbf{1}_{j<0}$ and $\bfX_0=0$. 
%

Our main result, which is the  positive resolution of \cite[Conjecture 1.9]{SP}, shows that in large $t$, the trajectory of $\bfX(t)$ is almost surely linear with slope uniform on $[-1, 1]$. In other words, the second class particle chooses a random direction in the rarefaction fan uniformly and then proceeds asymptotically in that direction (see Figure \ref{fig:Traj}).

\begin{figure}
	\begin{center}
	\includegraphics[width=4in]{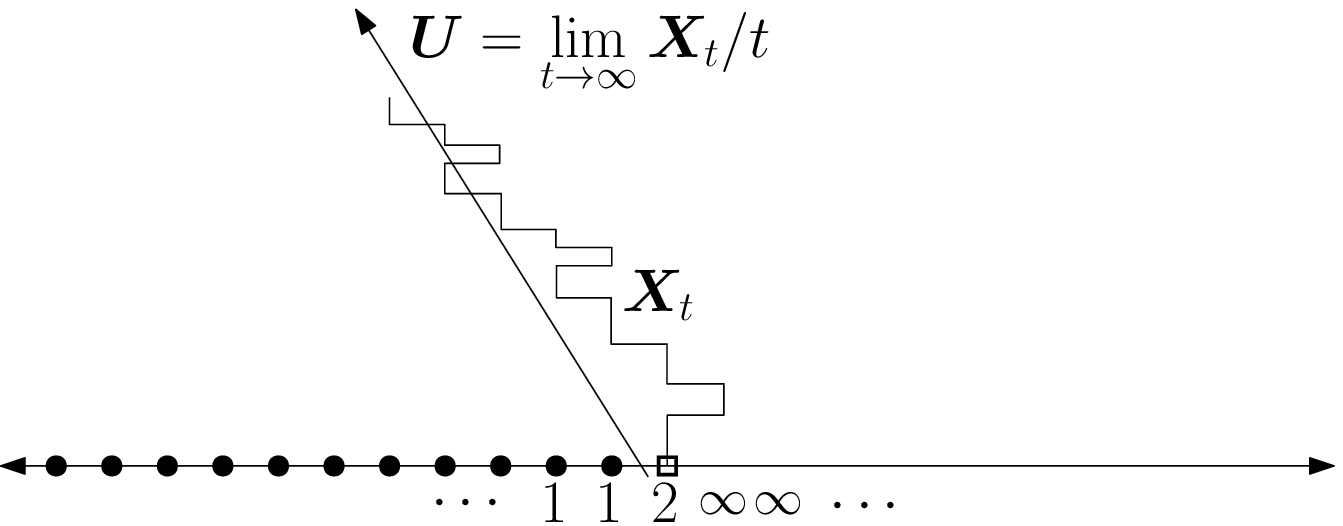}
	\end{center}
	\caption{Illustration of \Cref{xtlimitU}.}
	\label{fig:Traj}
\end{figure}

\begin{thm}[Conjecture 1.9 of \cite{SP}]
	\label{xtlimitU}
	The limit velocity $\bfU:=\lim\limits_{t \rightarrow \infty}\bfX_t/t$ of the second class particle $\bfX_t$ in $\ASEP_t$ exists almost surely and its law is uniform on $[-1, 1]$.
\end{thm}

The distributional limit of $\bfX_t/t$ (which we recall below) was known to be uniform for $L=0$ from \cite{SCP}, see equation (1.5). That was generalized to all $L$ in \cite[Theorem 2.1]{FGM09}. A  different proof of the distributional limit was given in \cite[Theorem 1.1]{GSZ}, based on color-position symmetries for multispecies ASEP discovered in \cite{BorWhe} and \cite{BorBuf}.

\begin{prop}[\cite{SCP,FGM09,GSZ}]
\label{xt1}
For any $\rho \in [0, 1]$,
\begin{equation*}
\lim_{t \rightarrow \infty} \PP \big[ \bfX_t/t \le 1 - 2\rho \big] = \rho.
\end{equation*}
\end{prop}

Thus, the proof of Theorem \ref{xtlimit} reduces to  the following almost sure limit for $X_t/t$.

\begin{thm}
	\label{xtlimit}
	The limit $\bfU:=\lim\limits_{t \rightarrow \infty}\bfX_t/t$ exists almost surely.
\end{thm}

\Cref{xtlimitU} implies well-definedness of the ASEP speed process, confirming \cite[Conjecture 1.10]{SP}.
Consider multispecies ASEP where initially at $n\in \Z$, we start with a class $n$ particle. Let the particles evolve as indicated above: each particle independently attempts to jump left and right with rates $L$ and $R$; those attempted jumps are achieved only if the destination is occupied with a higher class (hence lower priority) particle. For each $n\in \Z$, the class $n$ particle sees an initial condition which is equivalent to a translation of the initial condition considered in \Cref{xtlimitU}. Thus \Cref{xtlimitU} applies for each particle, namely if we let $\bfX_t(n)$ denote the location of the particle that started in position $n\in \Z$ at time $t\geq 0$, we have $\big(\bfX_t(n)-n\big)/t$ converges almost surely to random variable $U(n)$ with distribution uniform on $[-1,1]$. Taking a union over all particles implies that this holds simultaneously for all particles. Let $\mu^{{\rm ASEP}}$ denote the joint law of all $\big(U(n)\big)_{n\in \Z}$.

\begin{cor}[Conjecture 1.10 of \cite{SP}]
The ASEP speed process measure $\mu^{{\rm ASEP}}$ is well defined and translation invariant with each $U(n)$ uniform on $[-1,1]$.
\end{cor}

Having constructed this measure it is natural to investigate properties of it such as the joint distributions of various $U(n)$. We will not pursue this here, but we mention that \cite{SDMA} establishes various results in this direction (for instance, related to the properties of ``convoys'' of second class particles that move at the same limiting velocity) and \cite{GSZ} probes the distribution of $\min\big(U(1),\ldots, U(n)\big)$ as a function of $n$.

\smallskip

In the remainder of this introduction we will discuss how our results fit with respect to previous work, and then describe the heuristics and proof ideas. The proof that we provide combines probabilistic ideas (i.e., couplings) with integrable tools (i.e., effective hydrodynamic bounds). The interplay of these two techniques allows us to prove a result that we do not know how to attain with either separately.

Second class particles have been extensively studied with varying perspectives and purposes. When such a particle is started at a shock, it tracks out a microscopic version of the evolution of the shock \cite{Fer92,MSL}; when it is started in stationary initial data, it follows the characteristic velocity \cite{FF94, Rez91} and displays super-diffusive scaling around that related to the KPZ two-point distribution \cite{PS02,FS06, CTPS, Agg18,QV07,BS10}.

For step (sometimes called anti-shock) initial data, there is an entire rarefaction fan in the hydrodynamic equation and thus a continuum of characteristics velocities \cite{THMC}. The behavior of a second class particle started in such initial data (as we consider here) was first taken up in \cite{SCP} in the case $L=0$. As noted above \Cref{xt1}, they showed the asymptotic uniformity of the location of the second class particle in the rarefaction fan. They also proved that for any $0<s<t$ fixed, $\lim_{\varepsilon\to 0} \big( \frac{\bfX_{s/\varepsilon}}{s/\varepsilon}-\frac{\bfX_{t/\varepsilon}}{t/\varepsilon}\big) =0$ in probability.

This convergence was strengthened a decade later in \cite{MG05}, which proved the almost sure limit for the velocity of a second class particle (i.e., the $L=0$ case of \Cref{xtlimitU}); alternative proofs for the same result appeared in \cite{FP05,PTCI}. The starting point for \cite{MG05} is the coupling between $L=0$ TASEP and exponential last passage percolation (LPP). The almost sure limit relied on Sepp\"al\"ainen's microscopic variational formula for TASEP \cite{Sep99} along with some LPP concentration results. This relation to LPP is valuable and relates the second class particle to the competition interface \cite{FP05}. TASEP gaps relate to a totally asymmetric zero range process, leading to an understanding of second class particles for that model \cite{ABGM19,Gon14}.

When $L>0$, the LPP variational formula no longer holds. Thus, a new set of ideas is needed to establish  \Cref{xtlimitU}.  We will outline these below. The proof of \Cref{xtlimitU} is given in Section \ref{Linear}, relying on all of the results developed in this paper.

\smallskip
\noindent \emph{Understanding the results in terms of hydrodynamics.}
The uniformity of $\bfX_t/t$ on $[-1,1]$ is a microscopic manifestation of an observation about the hydrodynamic limit of ASEP.
Recall that the evolution of the density $\rho$ of particles on macroscopic time and space scales in ASEP is governed by the weak entropy solution to the inviscid Burgers equation
$$\partial_t \rho(t,x) = \partial_x\big(\rho(t,x)(1-\rho(t,x))\big).$$
In particular, as $\varepsilon\to 0$, the density field for the occupation process at time $t/\varepsilon$ in location $x/\varepsilon$ should converge in a weak  sense to the solution of this PDE (provided the initial data converges likewise).
If we start with step initial data $\rho(0,x)=\mathbf{1}_{x\leq 0}$ versus shifted step-initial data $\rho(0,x)=\mathbf{1}_{x\leq -\delta}$, the difference of the solutions at time $t$ is a function that is essentially uniform with value $\delta/(2t)$ between $-t$ and $t$. By the basic coupling of ASEP (see Section \ref{sec:couplings}), the shift in initial data can be interpreted as the addition of many second class particles to the left of the origin and the behavior of the hydrodynamic limit suggests the uniform distribution of the velocity of those particles.
The proof of the uniform distribution in \cite{SCP} uses the fact that ASEP reaches some form of local equilibrium. This means that if the local density is $\rho$, then the local distribution of particles should be given by Bernoulli product measure with parameter $\rho$. These measures are stationary for ASEP.

Assuming this local equilibrium behavior, we can start to understand why the second class particle maintains its velocity. Based on the hydrodynamic theory for step initial data, if $\bfX_t/t= 1-2\rho$ for some $\rho\in (0,1)$ then the density around $\bfX_t$ will be roughly $\rho$ and assuming local equilibrium, the occupation variables for first class particles around $\bfX_t$ will be close to i.i.d. Bernoulli with parameter $\rho$. In this equilibrium situation, $\bfX_t$ jumps left at rate $R\rho$ if position $\bfX_t-1$ is occupied by a first class particle and rate $L(1-\rho)$ if $\bfX_t-1$ has a hole; similarly $\bfX_t$ jumps rate at rate $L\rho$ if position $\bfX_t+1$ is occupied by a first class particle and rate $R(1-\rho)$ if $\bfX_t+1$ has a hole. Thus the expected instantaneous velocity of $\bfX_t$ is $(R-L)(1-2\rho) =1-2\rho$ and so in expectation $\bfX_t$ continues to move along the characteristic velocity $1-2\rho$. This is not the same as showing an almost sure limiting velocity. For infinite i.i.d. Bernoulli $\rho$ initial data, \cite{Fer92} showed exactly the latter.

\smallskip\noindent \emph{Proof sketch \Cref{xtlimit} when $L=0$ (TASEP).}
Though we are interested in the $L>0$ case, it is useful to first focus on $L=0$.  The proof  we describe here is different than  \cite{MG05} and does not rely on LPP. It also extends (using two additional ingredients) to $L>0$. We start by explaining an overly optimistic approach to the proof and then explain how it can be modified to produce an actual proof.

\begin{figure}
	\begin{center}
	\includegraphics[width=5.5in]{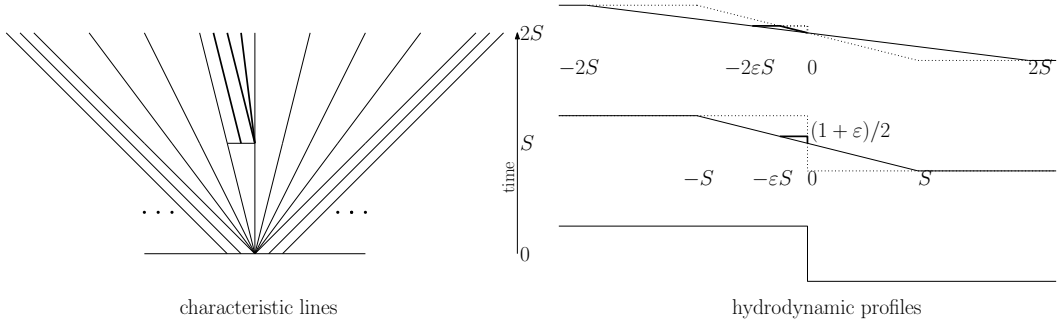}
	\end{center}
	\caption{Left: The linear characteristic lines used to solve the inviscid Burgers equation from step initial data. At time $S$ the density is perturbed in the interval $(-\varepsilon S,0)$ to match that of the left endpoint of the interval. The subsequent characteristics show how this perturbation evolves in time via the inviscid Burger equation. Right: The densities corresponding to the characteristics on the left. At time $S$  the profile (thin line) is augmented with the bold line to have density $(1+\varepsilon)/2$ on the interval $(-\varepsilon S,0)$. The time $2S$ profile is then shown (with dotted lines transcribing the time $S$ profile).}
	\label{fig:charhydro}
\end{figure}

For step initial data TASEP, at a large time $S$, we expect the density of particles will be approximated by the solution to the Burgers equation which linearly interpolates  between density one to the left of $-S$ and density zero to the right of $S$ (see the  rarefaction fan at the intermediate time in \Cref{fig:charhydro}). Assume for the moment that the occupation variables at time $S$ are independent Bernoulli with parameters given by this hydrodynamic profile, and also assume that $\bfX_S=0$ so it lies along a zero velocity characteristic. (If $\bfX_S$ were along another characteristic, we would need to work in a moving reference frame.)

Under these assumptions, we can couple our time $S$ system to another TASEP where the Bernoulli parameter profile is augmented to the left of the origin (i.e., the location of $\bfX_S$) as in \Cref{fig:charhydro}. Under the basic coupling, this corresponds to adding extra second class particles to the left of $\bfX_S$ to create the augmented profile. Importantly, these additional second class particles remain to the left of $\bfX_t$ at all times $t>S$. This fails when $L>0$.

Using the above observation, we see that in order to lower-bound the motion of $\bfX_t$ for $t>S$, it suffices to control the locations of the extra second class particles. While it is hard to control individual particles, we know how to control lots of them by use of hydrodynamic limit theory. Consider adding in enough second class particles so as to make a macroscopic change in the density profile. For example, on the interval $(-\varepsilon S,0)$ we can change the density to equal $(1+\varepsilon)/2$, as depicted on the right of \Cref{fig:charhydro}. At time $2S$ (top of \Cref{fig:charhydro}) this perturbation will evolve as to only perturb the density on the interval $(-2\varepsilon S,0)$. This suggests that with high probability, of the $O(S)$ added second class particles, all but $o(S)$ of them will be  to the right of $-2\varepsilon S$ and hence $\bfX_{2S}$ will be to the right of $-2\varepsilon S$ as well. Since $\varepsilon$ was arbitrary this suggests that $\bfX_t$ should maintain a velocity at least 0 (and by particle-hole symmetry, the opposite should follow too).

There are a number of issues above. The perturbation should really be on a spatial interval of size $o(S)$. This is because the above argument permits the velocity to drop by $\varepsilon$ on the time increment $S$ to $2S$, and if we repeat on doubling time intervals ($2S$ to $4S$, etc) the net drop may compound to become unbounded. This can be remedied by perturbing instead on an interval like $(-S^{1-\gamma},0)$ for some small $\gamma>0$. Assuming our hydrodynamic results extend to this scale, we should be able to bound the total drop in $\bfX_t$ at times of the form $S_n=2^{n}S$ for $n=0,1,\ldots$. However, at intermediate times $\bfX_t$ could wander in a manner that would prevent the  velocity from having a limit. To remedy this, we instead consider a sequence of times that grows like $S_n= S e^{\sqrt {n}}$ (in fact, by choosing $S_{n+1} = S_n+ S_n/\log S_n$). By a Poisson bound (from the basic coupling) the intermediate wandering of $\bfX_t$ does not change the velocity much compared to the $S_n$ times.

Besides these modifications, there is still the issue of justifying the simplistic assumptions we made based on hydrodynamic theory considerations. This is done by making use of \emph{effective} versions of hydrodynamic limit results that quantify with exponential decay how close the actual number of particles is to the hydrodynamic limit profile on spatial and fluctuation scales that are $o(S)$. For example, for step initial data if we look at the number of particles at time $S$ in an interval $[X,Y]$ with $-S<X<Y<S$, we expect that it will be approximately $S$ times the integral from $X/S$ to $Y/S$ of the hydrodynamic profile function $(1-z)/2$. An effective hydrodynamic concentration inequality  would say that for some $\alpha\in (0,1)$ the probability that the deviation of this number of particles around what we expect it to be will exceed $s S^{\alpha}$ is bounded above by $c^{-1} e^{-c s}$ for some $c>0$. (The optimal $\alpha$ should be $1/3$ and the decay should actually be faster than $e^{-cs}$ for any $c>0$, though we do not need or pursue this.) We also make use of similar bounds for other types of initial data such as the perturbed one, though these can be deduced from bounds for the class of step-Bernoulli initial data via coupling arguments. We use the exponential decay in these bounds when taking union bounds to control the hydrodynamic comparison at each $S_n$.

The step initial data effective hydrodynamic result is present in the literature. We quote \cite[Theorem 13.2]{LPDL} and \cite[Proposition 4.1 and Proposition 4.2]{ASFTLPM} (see \Cref{l0estimate} below) for this result. In fact, \cite{LPDL} essentially relies on \cite{CTPS} which uses Fredholm determinantant asymptotics as well as Widom's trick to establish the lower and upper tail bounds respectively. In general for determinantal models like TASEP, one tail often follows directly from showing decay of the kernel of the Fredholm determinant while the other is typically more complicated to demonstrate and requires tools like Widom's trick or Riemann-Hilbert problems \cite{BDMMZ}.

\smallskip\noindent \emph{Proving  \Cref{xtlimit} when $L>0$ (ASEP).}
It is easy to see (e.g. considering a two-particle system) that the presence of additional second class particles to the left of $\bfX_t$ may effect its motion and hence the simple coupling used above for TASEP fails. In its place, we make use of a more sophisticated coupling that was introduced in \cite[Section 4]{MSL} (see \Cref{prop:Rez} below). It says that for $t>S$, $\bfX_t$ can be stochastically lower bounded by the motion of a random second class particle unifomly chosen among those added to the left of $\bfX_t$ at time $t=S$. This enables us to implement for ASEP a similar sort of hydrodynamic argument as  given above for TASEP.

In addition to the above coupling, we also need to develop effective hydrodynamic concentration inequalities for ASEP. Due to reduction and coupling arguments, it suffices for us to demonstrate these in the case of step Bernoulli initial data. Distributional limit theorems for step initial data ASEP go back to \cite{AASIC} and for step Bernoulli initial data to \cite{ASC} where the one-point distribution of the height function (which captures the integrated occupation variables) was analyzed directly.

In \cite{BCS} it was realized that the ASEP height function $q$-Laplace transform admits a simpler form as a Fredholm determinant. The $q$-Laplace transform asymptotically captures the tails of the probability distribution. Our effective hydrodynamic results require both upper and lower tail control. As is typical in such formulas, one tail (typically called the upper tail) is readily accessible from the Fredholm determinant formula via decay of the kernel therein (see also \cite{DZ21} which derives the corresponding large deviation principle for this tail). The other (lower) tail requires a different type of argument. As mentioned earlier, in determinantal models, this is sometimes achieved via Widom's trick or Riemann-Hilbert problems, and in related random matrix theory contexts, other tools like electrostatic variational problems or tridiagonal matrices can be used for such bounds.

The first instance of a positive temperature model for which the lower tail was bounded in a manner adapted to KPZ scaling was the KPZ equation. This was achieved in \cite{CG} using a remarkable rewriting in \cite{BG} of the KPZ Laplace transform Fredholm determinant formula proved in \cite{ACQ}. Through this formula the Laplace transform for the KPZ equation was matched to a certain multiplicative functional of the determinantal Airy point process. From this, \cite{CG} derived tail bounds by controlling the behavior of the Airy points (something achievable through existing techniques).

There is a similar identity from \cite{AEPDPP} which relates the $q$-Laplace transform for ASEP to the expectation of a multiplicative functional of a certain discrete Laguerre determinantal point process (see also \cite{Bor18} which proves a more general result higher in the hierarchy of stochastic vertex models). From this identity is should be possible to extract fairly tight lower tail bounds. However, we do not need to use the full strength of this identity. In fact, the behavior of this multiplicative functional can be upper bounded by the behavior of its lowest particle, which ends up being equal to the TASEP height function. Thus, through this identity we can deduce the ASEP tail from existing knowledge of that of TASEP.

\smallskip
\noindent \emph{Outline.} Section \ref{sec:couplings} contains the definition of the basic coupling as well as key consequences such as attractivity (\Cref{xizeta1}), finite speed of propagation (\Cref{xizetaequal}) and monotonicity (\Cref{xizeta2}). We also recall as \Cref{prop:Rez} the coupling from \cite{MSL}, the proof of which is provided in Section \ref{sec:Rez} for completeness. Section \ref{sec:mde} contains our effective hydrodynamic concentration estimates that mainly stem from \Cref{hetaxi} -- these include \Cref{hetaxi2} and \Cref{hetalinear}. \Cref{hetaxi} is proved in Appendix \ref{sec:modDevproof} and \ref{RightKernel}.

Section \ref{Linear} contains the proof of our main result,  \Cref{xtlimit} (which combined with \Cref{xt1} implies  \Cref{xtlimitU} immediately). \Cref{xti} gives the main technical result that controls the motion of the second class particle between two times. This result translates into \Cref{cor:almostthere} and then into \Cref{xtlimit}. Section \ref{couple} proves \Cref{xti} by setting up a coupling as outlined in the proof sketch above and then showing (as \Cref{zti}) that most of the additional second class particles move at a speed close to that of the characteristic.
Section \ref{LimitProcess} proves \Cref{zti} by utilizing the effective hydrodynamic concentration estimates from  Section \ref{sec:mde}.

\smallskip
\noindent \emph{Notation.}
We fix $R > L \ge 0$ with $R - L = 1$. Unless specified otherwise we assume all constants and parameters are real valued, with the exception of indices which are obviously integer valued. When we introduce constants (the value of which may change despite using the same symbol), we will generally specify upon which parameters they depend by writing $c=c(\cdots)$ with the dependence inside the parentheses. We do not attempt to track constants through the paper or optimize our estimates (e.g. in concentration inequalities) beyond what is needed to reach our main result. We will typically use the sanserif font $\mathsf{E}$ for events and write $\mathsf{E}^c$ for complement of $\mathsf{E}$ and $\mathbf{1}_{\mathsf{E}}$ for indicator function which is $1$ on the event $\mathsf{E}$ and $0$ otherwise. We typically use $\eta,\zeta,\xi$ to denote elements of $\{0,1\}^{\Z}$, i.e., occupation variables. We will use bold-faced letters such as $\bfeta,\bfX$ to denote random variables. For real $x\leq y$ define $\llbracket x,y\rrbracket := \big[\lfloor x\rfloor, \lceil y\rceil\big]\cap \mathbb{Z}$; if $x>y$ define $\llbracket x,y\rrbracket =\varnothing$, the empty set.

\smallskip
\noindent \emph{Acknowledgements.}
We thank Gidi Amir, Omer Angel, James B. Martin and Peter Nejjar for helpful comments.
A.A. was partially supported by a Clay Research Fellowship and gratefully acknowledges support from the Institute for Advanced Study. I.C was partially supported by the NSF through grants DMS:1937254, DMS:1811143, DMS:1664650, as well as through a Packard Fellowship in Science and Engineering, a Simons Fellowship, a Miller Visiting Professorship from the Miller Institute for Basic Research in Science, and a W.M. Keck Foundation Science and Engineering Grant. A.A, I.C. and P.G. also wish to acknowledge the NSF grant DMS:1928930 which supported their participation in a fall 2021 semester program at MSRI in Berkeley, California, as well as the CRM in Montreal, Canada where this work was initiated in the 2019 conference on ``Faces of Integrability''. 

\section{Couplings}\label{sec:couplings}

The (single class) ASEP can be described as a Markov process on occupation variables or ordered particle location variables. The {\it occupation process}  $\bfeta_t = \big(\bfeta_t (j)\big)_{j \in \mathbb{Z}}\!\in\! \{0,1\}^\mathbb{Z}$ has infinitesimal generator $\mathcal{L}$ which acts on local functions $f(\eta)$ as
$$
\mathcal{L} f(\eta) = \sum_{j\in \mathbb{Z}} \big(R\cdot \eta(j)(1-\eta(j+1)) + L\cdot \eta(j+1)(1-\eta(j))\big) \big(f(\eta^{j,j+1})-f(\eta)\big)
$$
where $\eta^{j,j+1}$ switches the value of $\eta(j)$ and $\eta(j+1)$ (so $\eta^{j,j+1}(i)=\eta(i)$ for $i\neq j,j+1$, $\eta^{j,j+1}(j)=\eta(j+1)$ and  $\eta^{j,j+1}(j+1)=\eta(j)$). In words, particles jump left and rate according to independent exponential clocks of rates $L$ and $R$, provided that the destination site is unoccupied. The sites $j$ where $\bfeta_t(j)=1$ are said to be occupied by particles, and otherwise (when $\bfeta_t(j)=0$) by holes. As mentioned previously, we will always assume that $R>L\geq 0$ so that there is a net drift to the right.

\begin{rem}
\label{etaeta}
Observe that the ASEP is preserved under interchanging particles and holes, and by reversing all jump directions. Stated alternatively, suppose that $\bfeta_t$ is an ASEP with left jump rate $L$ and right jump rate $R$; then, the process $\check{\bfeta}_t$ defined by setting $\check{\bfeta}_t (j) = 1 - \bfeta_t (-j)$ for all $j\in \Z$ is also an ASEP with left jump rate $L$ and right jump rate $R$. This is sometimes referred to as \emph{particle-hole symmetry}.
\end{rem}

The {\it basic coupling} provides a single probability space upon which the evolution for all initial data for ASEP can simultaneously be defined (see \cite[VIII.2]{Liggett85}). Moreover, that coupling enjoys the properties of being {\it attractive} and {\it monotone} (these are recorded below), and hence allows us to define second (and more general) class particles. This construction is easily seen to match with the dynamics explained in the introduction.

The basic coupling comes from the {\it graphical} construction of ASEP which we now recall (see also Figure \ref{fig:graphical}). To every site $j\in\mathbb{Z}$ we associate two Poisson point processes on $[0,\infty)$, one which has rate $L$ and one which has rate $R$. Call the rate $L$ process the {\it left arrows} and the rate $R$ process the {\it right arrows}. All of these (between sites and at the same site) will be independent. Above every site $j\in \mathbb{Z}$ we draw a vertical line representing time and draw left and right arrows out of $j$ at heights corresponding to the points in the left and right arrow point processes just defined. For any initial data $\bfeta_0$, we define the time evolution $\bfeta_t$ in the following manner. Particles initially occupy sites $j$ where $\eta_0(j)=1$ and remain in place until they encounter an arrow out of their site. At that time, they follow the arrow, provided that the destination site is unoccupied; otherwise, they remain in their site until the next arrow. The basic coupling can also be defined directly in terms of the generator of dynamics on multiple choices of initial data -- see Section \ref{sec:Rez} for such generators.

\begin{figure}
	\begin{center}
	\includegraphics[width=3in]{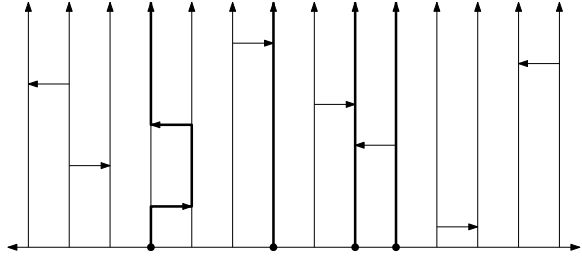}
	\end{center}
	\caption{The graphical construction of ASEP. Arrows are given by Poisson point processes and particles follow them provided the destination is unoccupied.}
	\label{fig:graphical}
\end{figure}

\begin{lem}[Attractivity]
	\label{xizeta1}
	Let $\bfeta_t$ and $\bfzeta_t$ denote two versions of ASEP with the same jump rates and with initial data such that $\bfeta_0 (j) \leq \bfzeta_0 (j)$ for each $j \in \mathbb{Z}$. Then, under the basic coupling, almost surely $\bfeta_t (j) \le \bfzeta_t (j)$ for all $j \in \mathbb{Z}$ and $t \ge 0$.	
\end{lem}

Attractivity allows us to define the first and second class particle process  $(\bfeta_t,\bfalpha_t)$ by the relation $\bfzeta_t = \bfeta_t + \bfalpha_t$ (see Figure \ref{fig:coupling}). By attractivity, $\bfalpha_t\in \{0,1\}^{\mathbb{Z}}$, and hence can be thought of as occupation variables for second class particles. We write $\PP^{\bfeta_0,\bfalpha_0}$ for the probability measure associated to the $(\bfeta_t,\bfalpha_t)$ process with initial data $(\bfeta_0,\bfalpha_0)$.
When there is a single second class particle (our particular interest), i.e., $\sum_{i \in \mathbb{Z}} \bfalpha_0(i)=1$, we denote its location at time $t$ by $\bfX_t$ (so that $\bfalpha_t(\bfX_t)=1$ and $\bfalpha_t(j)=0$ for all other $j$) and write  $\PP^{\bfeta_0,\bfX_0}$ for the probability measure associated to the $(\bfeta_t,\bfX_t)$ process with initial data $(\bfeta_0,\bfX_0)$.

\begin{rem}
\label{rem:secondclassduality}
The particle-hole symmetry noted in \Cref{etaeta} extends to two-species ASEP. In particle if we reverse all jump directions and swap first class particles and holes, and keep second class particles as is, then the two-species ASEP is preserved. Stated alternatively, suppose that
$(\bfeta_t,\bfalpha_t)$ records the first and second class particle occupation variables, then $\check{\bfeta}_t(j) =  1- \bfeta_t(-j)$ and $\check{\bfalpha}_t(j) = \bfalpha_t(-j)$ for all $j\in \Z$ is also a two-species ASEP with left jump rate $L$ and right jump rate $R$.
\end{rem}

For $x\in \mathbb{Z}$, $\bfeta_0\in \{0,1\}^{\mathbb{Z}}$ and $N\in \mathbb{Z}_{\geq 1}$, let $A^{\leq }(x,\bfeta_0,N)$ denote the set of $\bfalpha_0\in \{0,1\}^\mathbb{Z}$ such that $\sum_{j\in \mathbb{Z}} \bfalpha_0(j)=N$, $\bfeta_0+\bfalpha_0\in \{0,1\}^{\mathbb{Z}}$, $\bfalpha_0(x)=1$, and $\bfalpha_0(w)= 1$ only if $w\leq x$ (note that the ``only if'' is not ``if and only if'').
In words, this means that we start with $N$ second class particles relative to the first class particles at $\bfeta_0$, with the rightmost
one at site $x$.
Associate to $\bfalpha_0\in A^{\leq}(x,\bfeta_0,N)$ its ordered particle vector $\bfZ_0=(\bfZ_0(1)>\cdots>\bfZ_0(N))$ so that $\bfalpha(w)=1$ if and only if $w\in \{\bfZ(1),\ldots, \bfZ(N)\}$. Let $\bfZ_t=(\bfZ_t(1)>\cdots>\bfZ_t(N))$ be the ordered locations at time $t$ of $\bfalpha_t$.

The following result can be extracted from \cite[Section 4]{MSL} (we provide a proof of it in Appendix \ref{sec:Rez} for completeness). It says that to control the location of a single second class particle, we can introduce several second class particles to the left  and control the location of a typical (uniformly chosen) one of those (see the caption of Figure \ref{fig:coupling}).

\begin{figure}
	\begin{center}
	\includegraphics[width=3in]{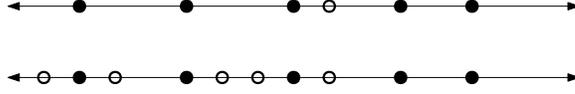}
	\end{center}
	\caption{Top: ASEP with first class particles (black bullets) and one second class particle (open disk). Bottom: ASEP with four additional second class particles added to the left of the top figure's second class particle. \Cref{prop:Rez} shows that we can couple the two versions of ASEP so the top second class particle stays to the right of a uniformly randomly chosen particle among the second class particles in the bottom figure.}
	\label{fig:coupling}
\end{figure}

\begin{prop}\label{prop:Rez}
	For any $y\in \mathbb{Z}$,  $\bfX_0\in \mathbb{Z}$ and $\bfeta_0\in \{0,1\}^{\mathbb{Z}}$ with $\bfeta_0(\bfX_0)=0$, and for any $N\in \mathbb{Z}_{\geq 1}$ and $\bfalpha_0\in A^{\leq}(\bfX_0,\bfeta_0,N)$,
	\begin{equation}\label{eq:Rezleq}
		\PP^{\bfeta_0,\bfX_0}[\bfX_t\leq y] \leq \frac{1}{N} \sum_{j=1}^{N} \PP^{\bfeta_0,\bfalpha_0}[\bfZ_t(j)\leq y].
	\end{equation}
\end{prop}
Another consequence of the graphical construction is ASEP's finite speed of propagation.
\begin{lem}
	\label{xizetaequal}
	Let $U \le V$, $T \ge 0$, and $\bfxi$ and $\bfzeta$ be two versions of ASEP (each with left and right jump rates $L$ and $R$, respectively). If $\bfxi_0 (j) = \bfzeta_0 (j)$ for each $j \in \llbracket U, V\rrbracket$, then under the basic coupling we have that $\bfxi_t (j) = \bfzeta_t (j)$ for each $j \in \llbracket U + 4RT, V - 4RT\rrbracket$ and $t \in [0, T]$, off of an event of probability at most $4 e^{-T/3}$.
\end{lem}
\begin{proof}
	This follows from large deviation bounds on the sum of exponential random variables which control how particles from outside an interval can effect the behavior far inside it.
\end{proof}

The final general result we give from coupling is {\it montonicity}. It deals with the integrated occupation variables, i.e., sometimes called the  height function or current.
Let $\bfxi_t$ denote ASEP and identify the {\it ordered particle locations} by $\cdots < \bfY_t(1)< \bfY_t(0) < \bfY_t(-1) < \cdots$ where the indexing is such that initially $\bfY_0 (0) \leq 0 < \bfY_{0}(-1)$ (subsequently, the $\bfY_t(j)$ track these indexed particles as they jump). For any $x \in \mathbb{Z}$, we define
\begin{flalign}
	\label{jtx}
	\h_t (x; \bfxi) = \displaystyle\sum_{i\in \mathbb{Z}} \big( \textbf{1}_{\bfY_0(i) \le 0} \textbf{1}_{\bfY_t(i) > x} - \textbf{1}_{\bfY_0(i) > 0} \textbf{1}_{\bfY_t(i) \le x} \big)
\end{flalign}
and extend $\h_t (x; \bfxi)$ to a continuous function in $x$ by linear interpolation. For $x,y\in \Z$,
\begin{flalign}\label{eq:heightdiff}
\h_t ([x,y]; \bfxi):= \h_t (x; \bfxi)-\h_t (y; \bfxi) = \sum_{i=x+1}^{y} \bfxi_t(i)
\end{flalign}
from which it is clear that for $j\in \mathbb{Z}$,
\begin{equation}\label{eqhdiff}
\bfxi_t(j) = \h_t (j-1;\bfxi)-\h_t (j;\bfxi).
\end{equation}
In particular, if $t=0$ we will use the short-hand $\h(x; \bfxi)=\h_0 (x; \bfxi)$ and have that
\begin{equation}\label{eqhsum}
\h(x; \bfxi) =\h_0(x; \bfxi) = \begin{cases} -\displaystyle\sum_{i=1}^{x} \bfxi_0(i)&\textrm{if } x\geq 1, \\ 0&\textrm{if } x=0,\\ \displaystyle\sum_{i=x+1}^{0} \bfxi_0(i)&\textrm{if } x\leq -1.\end{cases}
\end{equation}

At most one of the two summands on the right side of \eqref{jtx} is nonzero. Observe that $\h_t (x;\bfxi)$ has the following combinatorial interpretation: Color all particles initially to the right of $0$ red, and all particles initially at or to the left of $0$ blue. Then, $\h_t (x;\bfxi)$ denotes the number of red particles at or to the left of $x$ at time $t$ subtracted from the number of blue particles to the right of $x$ at time $t$.

The following  shows that if we start with two height functions that are coupled so that they are either ordered pointwise (up to a vertical shift by some $H$) or close to each other (within $K$), then this property persists under the basic coupling. In the first statement, the shift by $H$ may be necessary since our height functions are zeroed out to satisfy $\h_0 (0;\bfxi)=0$; observe that the second statement of the below lemma follows from the first.

\begin{lem}[Monotonicity]
	\label{xizeta2}
	Let $\bfxi_t$ and $\bfzeta_t$ be two ASEPs with the same jump rates.

	\begin{enumerate}[leftmargin=*]
		\item If for some $H\in \mathbb{Z}$ we have $\h_0 (x; \bfxi)+H \ge \h_0 (x; \bfzeta)$ for each $x \in \mathbb{Z}$, then under the basic coupling we almost surely have $\h_0 (x; \bfxi)+H \ge \h_0 (x; \bfzeta)$ for all $x \in \mathbb{Z}$ and $t \ge 0$.
		\item If for some $K \in \mathbb{Z}$ we have $\big| \h_0 (x; \bfxi) - \h_0 (x; \bfzeta) \big| \le K$ for each $x \in \mathbb{Z}$, then under the basic coupling we almost surely have $\big| \h_t (x; \bfxi) - \h_t (x; \bfzeta) \big| \le K$ for all $x \in \mathbb{Z}$ and $t \ge 0$.
	\end{enumerate}

\end{lem}

\begin{figure}[t]
	\begin{center}
	\includegraphics[width=2.5in]{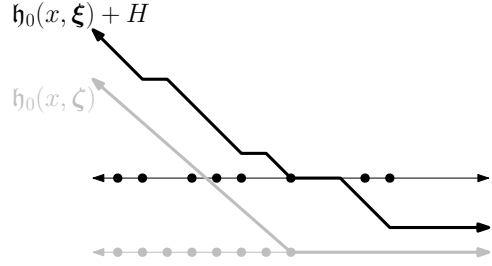}
	\end{center}
	\caption{Two height functions are depicted. The grey one is determined by the values of $\bfzeta_0$ while the black one is determined by the values of $\bfxi_0$. If the later is shifted by $H$ it point-wise exceeds the former. Provided this occurs at time 0, Lemma \ref{xizeta2} shows that this property persists for all time.}
	\label{fig:monotone}
\end{figure}

\section{Some effective hydrodynamics concentration estimates}
\label{sec:mde}


This section establishes uniform estimates that upper bound the maximal deviations that ASEP height functions can have from their hydrodynamic limits. The key to establishing these concentration bounds is an understanding of the fluctuations under the stationary measure (which just boils down to bounds on sums of i.i.d. Bernoulli random variables) and under step-Bernoulli initial data. This later result is contained in \Cref{hetaxi} and proved later in Section \ref{sec:modDevproof}. These are  put together using  attractivity of the basic coupling.

We begin with the following definition describing random particle configurations distributed according to a product measure. Such configurations will often serve as initial data for the versions of ASEP we consider. Throughout, all versions of ASEP will have the same left jump rate $L$ and right jump rate $R$, for $R > L \ge 0$ with $R - L = 1$.

\begin{definition}
	
\label{distributedinitial}

Fix a finite interval $I = \llbracket A, B\rrbracket$ with integer endpoints $A< B$, as well as a function $\varphi : \mathbb{R} \rightarrow [0, 1]$. We say that a particle configuration $\bfeta = \big( \bfeta (x) \big)$ is \emph{$\varphi$-distributed} on $I$ if its coordinates $\big \{ \bfeta (x) \}$ are all mutually independent and
\begin{flalign*}
	\PP \big[ \bfeta (A + x) = 1 \big] = \varphi \bigg( \displaystyle\frac{x}{B - A} \bigg), \qquad \text{for each $x \in \mathbb{Z}$}.
\end{flalign*}
We say that $\bfeta$ is $\varphi$-distributed on $\mathbb{Z}$ if its coordinates $\bfeta (x)$ are  mutually independent and
\begin{flalign*}
	\PP \big[ \bfeta (x) = 1 \big] = \varphi (x), \qquad \text{for each $x \in \mathbb{Z}$}.
\end{flalign*}
These two notations are somewhat at odds since the former (involving finite $I$) involves rescaling while the latter does not. We hope the reader will excuse us for this.
\end{definition}

When using \Cref{distributedinitial}, we will often (although not always, for instance, see the formulation of the lemma below) take $I = \llbracket-K, K\rrbracket$ for some integer $K \ge 1$ and $\varphi$ to be some piecewise linear function which takes value zero outside the interval $[0, 1]$. This will guarantee that $\bfeta$ only has particles on $\llbracket -K, K\rrbracket$.

The following is a concentration inequality for $\varphi$-distributed particle configurations.

\begin{lem}
\label{distributionconcentration}

Adopt the notation of \Cref{distributedinitial} and assume that $I=\mathbb{Z}$. For any $s \in \mathbb{R}_{\ge 1}$ and $X, Y \in \mathbb{Z}$, we have
\begin{flalign}
	\label{hx1}
	\PP \bigg[ \Big| \h (X; \bfeta) - \h (Y; \bfeta) - \displaystyle\sum_{j = X}^{Y} \varphi(j)   \Big| \ge s |Y - X|^{1/2} \bigg] \le 2 e^{-s^2}.
\end{flalign}
Now consider the case where  $I = \llbracket A, B\rrbracket$ is finite and $\varphi(x)\equiv 0$ for all $x\notin[0, 1]$. Then,
\begin{flalign}
	\label{hxy2}
	\PP \bigg[ \displaystyle\max_{\substack{X, Y \in \mathbb{Z} \\ X \le Y}} \Big| \h (X; \bfeta) - \h (Y; \bfeta) - \displaystyle\sum_{j = X}^{Y} \varphi \Big( \displaystyle\frac{j - A}{B - A} \Big)   \Big| \ge s (B - A)^{1/2} \bigg] \le 2 (B-A+1)^2 e^{-s^2 }.
\end{flalign}
\end{lem}

\begin{proof}
Observe that \eqref{hx1} followed immediately from Hoeffding's inequality and the fact that $\bfeta$ is $\varphi$-distributed.
Next, assume that $I = \llbracket A, B\rrbracket$ is a finite interval and that $\varphi$ is supported on $[0, 1]$. Using the fact that for $X, Y \in I$, we have $|Y - X| \le B - A$, Hoeffding's inequality and a union bound yields
\begin{equation*}
		\PP \bigg[ \displaystyle\max_{\substack{X, Y \in \llbracket A,B\rrbracket \\ X \le Y}} \Big| \h (X; \bfeta) - \h  (Y; \bfeta) - \displaystyle\sum_{j = X}^{Y} \varphi \Big( \displaystyle\frac{j - A}{B - A} \Big)   \Big| \ge s (B - A)^{1/2} \bigg] \le 2 (B - A + 1)^2 e^{-s^2}.
	\end{equation*}
The bound \eqref{hxy2} follows from combining the above with the fact that since $\varphi$ is supported on $[0, 1]$ we have for $X < A$ and $Y > B$ that
$
\h (A; \bfeta) - \h (X; \bfeta) = 0 = \h (B; \bfeta) - \h (Y; \bfeta).
$
\end{proof}

We now specify two choices we will commonly take for $\varphi$ from \Cref{distributedinitial}.

\begin{definition}
	
\label{lambdarhofunctions}

Fix real numbers $0 \le \lambda \le \rho \le 1$. Define the piecewise constant function $\Xi^{(\rho; \lambda)} : \mathbb{R} \rightarrow [\rho,\lambda]$ and the piecewise linear function $\Upsilon^{(\rho; \lambda)} : \mathbb{R} \rightarrow \mathbb{R}$ by setting
\begin{equation*}
\Xi^{(\rho; \lambda)} (z)= \begin{cases} \rho&\textrm{if } z\leq 0,\\\lambda&\textrm{if } z>0,\end{cases}\qquad\qquad
\Upsilon^{(\rho; \lambda)} (z)= \begin{cases} \rho&\textrm{if } z\leq 1-2\rho,\\  (1-z)/2&\textrm{if } 1-2\rho \leq z\leq 1-2\lambda,\\ \lambda&\textrm{if } z\geq 0.\end{cases}
\end{equation*}
\end{definition}

We say that an ASEP $\bfeta_t$ has \emph{$(\rho; \lambda)$-Bernoulli initial data} if $\bfeta_0$ is $\Xi^{(\rho; \lambda)}$-distributed on $\mathbb{Z}$. Observe in particular that $(1; 0)$-Bernoulli initial data is equivalent to step initial data, and that $(\rho; \rho)$-Bernoulli initial data is stationary for the ASEP; we call the latter \emph{$\rho$-stationary initial data}. The $\Upsilon^{(\rho; \lambda)}$-distributed initial data is meant to model the profile that one gets after running $\Xi^{(\rho; \lambda)}$-distributed initial data for a long time (with a linear interpolating rarefaction fan from density $\rho$ to density $\lambda$). The assumption $\lambda \leq \rho$ ensures that the hydrodynamic limit does not have shocks.

The following is a key concentration estimate $(\rho; 0)$-step Bernoulli initial data ASEP. This estimate is not optimal, either in the error bound $T^{2/3}$ or in the probability decay $e^{-cs}$. (In the case of step initial data, we believe that the $T^{1/3}$ scale is optimal, but the decay is not.) Note that for our purposes, it is sufficient that we have a bound of the form $T^{\alpha}$ for some $\alpha<1$. A proof of this result is given in Section \ref{sec:modDevproof}.
\begin{prop}
\label{hetaxi}
For any $\varepsilon > 0$, there exists  $c = c(\varepsilon) > 0$ such that the following holds. Let $\rho \in [\varepsilon, 1]$ and $\boldsymbol{\eta}$ be$(\rho; 0)$-Bernoulli initial data ASEP. For any $T > 1$ and $s \in [0, T]$,
\begin{flalign}
\label{hetaxilambda0}
 \displaystyle\max_{\substack{|X/T| \le 1 - \varepsilon \\ |Y/T| \le 1 - \varepsilon}}\PP \Bigg[ \bigg| \h_T(\llbracket X,Y\rrbracket; \bfeta) - T \displaystyle\int\limits_{X/T}^{Y/T} \Upsilon^{(\rho; 0)} (z) dz \bigg| \ge s T^{2/3} \Bigg] \le c^{-1} T e^{-c s}.
\end{flalign}
For step initial data (when $\rho=1$) \eqref{hetaxilambda0} holds with the term $sT^{2/3}$ replaced by $sT^{1/3}$.
The constants $c=c(\varepsilon)$ can be chosen so as to weakly decrease as $\varepsilon$ decreases to 0.
\end{prop}

From \Cref{hetaxi} and monotonicity, we deduce the following corollary showing that \eqref{hetaxilambda0} also holds under $(\rho; \lambda)$-Bernoulli initial data for any $0 \le \lambda \le \rho \le 1$.

\begin{cor}	\label{hetaxi2}
For any $\varepsilon\in(0,1)$, there exists $c = c(\varepsilon) > 0$ such that the following holds. For any $\lambda \in [0, 1 - \varepsilon]$ and $\rho \in [\varepsilon, 1]$ with $\lambda \le \rho$, let $\bfeta_t$ denote $(\rho; \lambda)$-Bernoulli initial data ASEP. Then, for any $T>1$ and $s \in [0, T]$,
	\begin{flalign}
		\label{hetaxilambdarho}
		\PP \Bigg[ \displaystyle\max_{\substack{|X/T| \le 1 - \varepsilon \\ |Y/T| \le 1 - \varepsilon}} \bigg|  \h_T(\llbracket X,Y\rrbracket; \bfeta)  - T \displaystyle\int\limits_{X/T}^{Y/T} \Upsilon^{(\rho; \lambda)} (z) dz \bigg| \ge s T^{2/3} \Bigg] \le c^{-1} T^3 e^{-c s}.
	\end{flalign}
The constants $c=c(\varepsilon)$ can be chosen so as to weakly decrease as $\varepsilon$ decreases to 0.
\end{cor}

\begin{proof}
By the particle-hole symmetry in \Cref{etaeta} along with \Cref{hetaxi} applied to $(1 - \lambda,0)$-Bernoulli initial data ASEP,  \eqref{hetaxilambdarho} holds if $(\rho; \lambda) = (1; \lambda)$.
	
Now consider the case where $\lambda = \rho$. Then $\bfeta$ is stationary in time, and so $\bfeta_T$ is also $\Xi^{(\rho; \rho)}$-distributed on $\mathbb{Z}$. Hence, the $\varphi = \Xi^{(\rho; \rho)}$ case of \eqref{hx1} together with a union bound over all integer $X, Y \in \llbracket -T, T\rrbracket$ yields
	\begin{flalign*}
		\PP \Bigg[ \displaystyle\max_{|X|, |Y| < T} \bigg|  \h_T(\llbracket X,Y\rrbracket; \bfeta) - \rho (Y - X) \bigg| \ge s T^{1/2} \Bigg] \le 50 T^2 e^{-s^2/2},
	\end{flalign*}
($50$ is not tight, but sufficiently large), which verifies \eqref{hetaxilambdarho}.
	
Now suppose that $(\rho; \lambda)$ is arbitrary satisfying $\lambda \in [0, 1  - \varepsilon]$ and $\rho \in [\varepsilon, 1]$, with $\lambda \le \rho$. By a union bound, to show \eqref{hetaxilambdarho}  it suffices that we show that there exists $c = c(\varepsilon) > 0$ such that for any integers $X\leq Y$ with  $|X/T|, |Y/T| \le (1 - \varepsilon)$,
	\begin{flalign}
		\label{probability1rholambda}
		\begin{aligned}
		& \PP \bigg[ \h_T(\llbracket X,Y\rrbracket; \bfeta) \ge T \displaystyle\int\limits_{X/T}^{Y/T} \Upsilon^{(\rho; \lambda)} (z) dz + s T^{2/3} \bigg] \le c^{-1} T e^{-c s}, \\
		& \PP \bigg[\h_T(\llbracket X,Y\rrbracket; \bfeta) \le  T \displaystyle\int\limits_{X/T}^{Y/T} \Upsilon^{(\rho; \lambda)} (z) dz - s T^{2/3} \bigg] \le c^{-1} T e^{-c s}.
		\end{aligned}
	\end{flalign}
We only establish the first bound in \eqref{probability1rholambda}, as the proof of the latter is entirely analogous.
	
To that end, let $\bfxi_t$ and $\bfzeta_t$ denote two ASEPs started with $(1; \lambda)$-Bernoulli initial data and $\rho$-stationary initial data, respectively. Since $\lambda \le \rho \le 1$, we may couple the Bernoulli initial data $\bfeta_0, \bfxi_0$ and $\bfzeta_0$ on the same probability space so that $\bfeta_0 (x) \le \min \big\{ \bfxi_0 (x), \bfzeta_0 (x) \big\}$, for each $x \in \mathbb{Z}$, almost surely. (This is a microscopic form of the ordering illustrated in \Cref{fig:Corollary35}.) The basic coupling used in \Cref{xizeta1} implies the existence of a coupling between $\bfeta_t, \bfxi_t$ and $\bfzeta_t$ such that $\bfeta_t (x) \le \min \big\{ \bfxi_t (x), \bfzeta_t (x) \big\}$ holds for each $x \in \mathbb{Z}$ and $t \in \mathbb{R}_{\ge 0}$, almost surely. In particular, using \eqref{eq:heightdiff} we almost surely have that
	\begin{flalign*}
		\h_T(\llbracket X,Y\rrbracket; \bfeta) =\displaystyle\sum_{j = X+1}^{Y} \bfeta_T (j) & \le \min \Bigg\{ \displaystyle\sum_{j = X+1}^Y \bfxi_T (j), \displaystyle\sum_{j = X+1}^Y \bfzeta_T (j) \Bigg\} \\
		& = \min \big\{ \h_T (X; \bfxi) - \h_T (Y; \bfxi), \h_T (X; \bfzeta)  - \h_T (Y; \bfzeta) \big\}.
	\end{flalign*}

\begin{figure}[t]
	\begin{center}
	\includegraphics[width=5in]{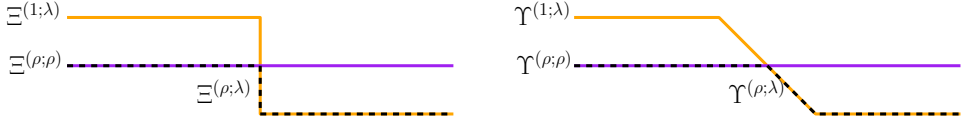}
	\end{center}
	\caption{The initial data $\Xi^{(\rho,\lambda)}$ (on the left) can be bounded above by the minimum of $\Xi^{(1,\lambda)}$ and $\Xi^{(\rho,\rho)}$, and likewise $\Upsilon^{(\rho,\lambda)}$ (on the right) can be bounded above by the minimum of $\Upsilon^{(1,\lambda)}$ and $\Upsilon^{(\rho,\rho)}$.}
	\label{fig:Corollary35}
\end{figure}

By \Cref{lambdarhofunctions} we have (see \Cref{fig:Corollary35}) that $\Upsilon^{(\rho; \lambda)} (z) = \min \big\{ \Upsilon^{(1; \lambda)} (z), \Upsilon^{(\rho; \rho)} (z) \big\}$. Therefore, to establish the first bound in \eqref{probability1rholambda}, it suffices to show that
	\begin{flalign}
	\label{xizetaprobability1rholambda}
	\begin{aligned}
		& \PP \bigg[ \h_T(\llbracket X,Y\rrbracket; \bfxi) \ge T \displaystyle\int\limits_{X/T}^{Y/T} \Upsilon^{(1; \lambda)} (z) dz + s T^{2/3} \bigg] \le c^{-1} T e^{-c s}, \\
		& \PP \bigg[\h_T(\llbracket X,Y\rrbracket; \bfzeta) \ge  T \displaystyle\int\limits_{X/T}^{Y/T} \Upsilon^{(\rho; \rho)} (z) dz + s T^{2/3} \bigg] \le c^{-1} T e^{-c s}.
	\end{aligned}
\end{flalign}	
Since the first and second estimates in \eqref{xizetaprobability1rholambda} follow from the already established $(\rho; \lambda) = (1; \lambda)$ and $(\rho; \lambda) = (\rho; \rho)$ cases of the corollary, we deduce the first inequality in \eqref{probability1rholambda}. The second inequality in \eqref{probability1rholambda} follows similarly as above by lower bounding by $(\rho,0)$-Bernoulli and $\lambda$-stationary initial data. This completes the proof of  \eqref{hetaxilambdarho} and hence the corollary.
\end{proof}

The rest of this section establishes effective hydrodynamic concentration inequalities for ASEP with initial data given by specific piecewise linear functions (though the methods apply more generally) defined below and illustrated in \Cref{fig:Prop41withoutpsi}.

\begin{figure}[t]
	\begin{center}
	\includegraphics[width=4in]{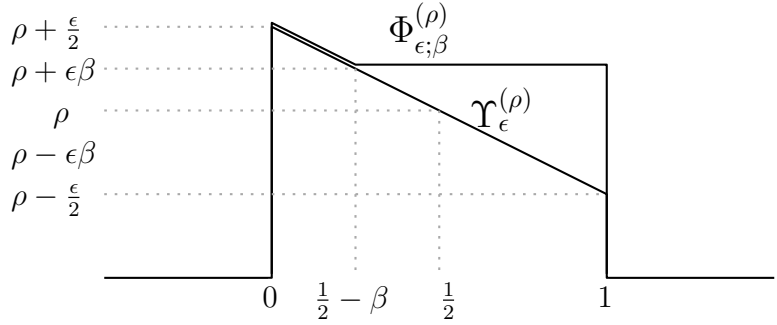}
	\end{center}
	\caption{$\Phi_{\varepsilon; \beta}^{(\rho)}$ and $\Upsilon_{\varepsilon}^{(\rho)}$ (with slight vertical shifts to make it easier to distinguish).}
	\label{fig:Prop41withoutpsi}
\end{figure}

\begin{definition}\label{linearconstant}
Fix any $\varepsilon \in \big( 0, \frac{1}{2} \big)$ and $\rho \in [\varepsilon, 1 - \varepsilon]$. Define $\Upsilon_{\varepsilon}^{(\rho)} : \mathbb{R} \rightarrow [0, 1]$ by
\begin{equation}
\label{functionlinear}
	\Upsilon_{\varepsilon}^{(\rho)} (z) = \begin{cases}\rho + \varepsilon (\frac{1}{2} - z) & \textrm{if }z \in [0, 1],\\ 0 &\textrm{if } z\notin[0,1].\end{cases}
\end{equation}
The function $\Upsilon_{\varepsilon}^{(\rho)}$ is a suitable translation and scaling of the function $\Upsilon^{(\rho; \lambda)}$ from \Cref{lambdarhofunctions}, where we additionally set it to $0$ outside of the interval $[0, 1]$. The function $\Upsilon_{\varepsilon}^{(\rho)}$ is linear on its non-zero support. It will also be useful to consider versions of this function that (continuously) transition from being linear to constant. To that end, for any $\varepsilon, \beta \in \big( 0, \frac{1}{2} \big)$ and $\rho \in [\varepsilon, 1 - \varepsilon]$, define  $\Phi_{\varepsilon; \beta}^{(\rho)}: \mathbb{R} \rightarrow [0, 1]$ by
\begin{equation*}
\Phi_{\varepsilon; \beta}^{(\rho)} (z) =\begin{cases}  \rho + \varepsilon (\frac{1}{2} - z )& \textrm{if }z \in [ 0, \frac{1}{2} - \beta ],\\
\rho + \varepsilon \beta & \textrm{if }z \in [\frac{1}{2} - \beta, 1 ],\\0&\textrm{if }z\notin[0,1].\end{cases}
\end{equation*}
\end{definition}

The following proposition provides effective hydrodynamic concentration estimates for the ASEP under either $\Upsilon_{\varepsilon}^{(\rho)}$-distributed or  $\Phi_{\varepsilon; \beta}^{(\rho)}$-distributed initial data.

\begin{prop}
	\label{hetalinear}
For any fixed $\delta \in \big( 0, \frac{1}{16R} \big)$, there exists $c = c(\delta) > 0$ such that the following holds. For any $S, T \in \mathbb{R}_{\ge 1}$ with $S \ge \delta^{-2} T$, $\beta \in \big( 0, \frac{1}{4} \big)$, $\varepsilon \in \big( 4 \delta, \frac{1}{2} \big)$, $\rho \in [\varepsilon, 1 - \varepsilon]$, and $\kappa \in [15, T]$:
\begin{enumerate}[leftmargin=*]
\item\label{hetaxilambdarholinear} ASEP $\bfeta_t$ with $\Upsilon_{\varepsilon}^{(\rho)}$-distributed initial data on the interval $\llbracket -\varepsilon S,\varepsilon S\rrbracket$ satisfies
\begin{flalign*}
		\begin{aligned}
	\PP \bigg[ \displaystyle\max_{\substack{|X/S| \le \varepsilon / 4 \\ |Y/S| \le \varepsilon / 4}} \Big|\h_T(\llbracket X,Y\rrbracket; \bfeta) - T\!\! \displaystyle\int\limits_{X/T}^{Y/T}  \Big( \rho + \displaystyle\frac{(1 - 2 \rho - z) T}{2 (S + T)} \Big) dz \Big| \ge \kappa S^{2/3} \bigg]  \le c^{-1} S^3 e^{-c \kappa};
	\end{aligned}
\end{flalign*}

\item\label{hetaxilambdarholinear2}  ASEP $\bfeta_t$ with $\Phi_{\varepsilon; \beta}^{(\rho)}$-distributed initial data on the interval  $\llbracket -\varepsilon S,\varepsilon S\rrbracket$ satisfies
\begin{flalign*}
	\begin{aligned}
	\!\!\!\!\!\!\!\!\PP \bigg[ \displaystyle\max_{\substack{|X/S| \le \varepsilon / 4 \\ |Y/S| \le \varepsilon / 4}} \Big| \h_T(\llbracket X,Y\rrbracket; \bfeta)- T\!\!\!  \displaystyle\int\limits_{X/T}^{Y/T}\!\!\! \max \Big\{ \rho + \displaystyle\frac{(1 - 2 \rho - z) T}{2 (S + T)}, \rho + \varepsilon \beta \Big\} dz \Big|\! \ge \kappa S^{2/3} \bigg]\!  \le c^{-1} S^3 e^{-c \kappa}.
	\end{aligned}
\end{flalign*}
%
\end{enumerate}
\end{prop}
\begin{proof}
	
	The proofs of \Cref{hetalinear}  \eqref{hetaxilambdarholinear} and \eqref{hetaxilambdarholinear2} are very similar, so we only detail that of \eqref{hetaxilambdarholinear}. The idea will be to compare $\bfeta_t$ on the time interval $t\in [0,T]$ to another version of ASEP  $\bfzeta_t$ that corresponds to step initial data ASEP, with all particles outside a specific window destroyed at time $S$ and then run for time $t\in [S,S+T]$. The window is chosen so the step initial data hydrodynamic limit replicates the profile for the initial data of $\bfeta_0$.

To this end, let $\bfxi_t$ denote ASEP under step initial data (to establish \eqref{hetaxilambdarholinear2} we would instead let $\bfxi_t$ denote an ASEP under two-sided $(1; \rho + \varepsilon \beta)$-Bernoulli initial data).
By \Cref{hetaxi2}, there exists $c = c (\delta)>0$ such that for any $U>1$ and $\kappa \in [1, U]$,
\begin{flalign}
\label{hxizetau}
\displaystyle\max_{\substack{|X/U| \le 1 - \delta \\ |Y/U| \le 1 - \delta}} \PP \bigg[ \Big| \h_U (\llbracket X,Y\rrbracket; \bfxi) - U \displaystyle\int\limits_{X/U}^{Y/U} \Big( \displaystyle\frac{1 - z}{2} \Big) dz \Big| > \kappa U^{2/3} \bigg] < c^{-1} U e^{- c \kappa}.
\end{flalign}
Now define $\bfzeta_t$ to be an ASEP started from random initial data $\bfzeta_0$ given by
	\begin{flalign}
		\label{zeta0xis}
	\bfzeta_0 (x) = \begin{cases} \bfxi_S \big( j + \lfloor(1 - 2 \rho) S\rfloor \big) & \textrm{if }j \in \llbracket -\varepsilon S, \varepsilon S\rrbracket,
\\ 0 &\textrm{if } j\notin \llbracket -\varepsilon S, \varepsilon S\rrbracket.  \end{cases}	
	\end{flalign}
By \Cref{xizetaequal}, we may couple the ASEPs $\bfzeta_t (j)$ and $\bfxi_{S + t} \big( j + \lfloor(1 - 2 \rho) S \rfloor \big)$ so that for all $t \in [0, T]$ they coincide with high probability on $j \in \llbracket -(\varepsilon S - 4RT), \varepsilon S - 4RT\rrbracket$, namely
	\begin{flalign}
		\label{aevent}
		\PP [\mathsf{A}] \ge 1 - 4 e^{-T / 3},
	\end{flalign}
where the  event $\mathsf{A}$ is defined by
\begin{equation*}
\mathsf{A} = \Big\{ \bfzeta_t (j) = \bfxi_{S + t} \big( j + \lfloor(1 - 2 \rho)\rfloor S \big)\textrm{ for all }t\in [0,T] \textrm{ and } j \in  \llbracket -(\varepsilon S - 4RT), \varepsilon S - 4RT\rrbracket \Big\}.
\end{equation*}
By \eqref{eq:heightdiff} it then follows that for $X \in \llbracket -(\varepsilon S - 4RT), \varepsilon S - 4RT\rrbracket$,
	\begin{flalign}
		\label{hha}
		\textbf{1}_{\mathsf{A}} \h_t(X; \bfzeta) = \textbf{1}_{\mathsf{A}} \Big( \h_{S+t} \big(\big\llbracket X+\lfloor (1 - 2 \rho) S\rfloor,\lfloor (1 - 2 \rho) S\rfloor\big\rrbracket ; \bfxi\big) \Big).
	\end{flalign}

	Next, by applying \eqref{hxizetau} with $(U; X, Y)$ equal to $\big( S; X+\lfloor (1 - 2 \rho) S\rfloor, \lfloor (1 - 2 \rho) S\rfloor \big)$, and using the matching from \eqref{zeta0xis}, we see that there exists $c=c(\delta)>0$ such that for $\kappa \in [6, 6U]$,
	\begin{flalign}
		\label{bprobability}
		\PP \big[ \mathsf{B} (\kappa) \big] \ge 1 - c^{-1} S^2 e^{- c \kappa},
	\end{flalign}
where the event $\mathsf{B} (\kappa)$ is defined by
	\begin{flalign*}
		\mathsf{B} (\kappa) = \bigg\{   \displaystyle\max_{|X/S| \le \varepsilon}  \Big| \h_0(X; \bfzeta \big) - S \displaystyle\int\limits_{X/S}^0 \Big( \rho - \displaystyle\frac{z}{2} \Big) dz \Big| \le \displaystyle\frac{\kappa S^{2/3}}{6} \bigg\}.
	\end{flalign*}
In order to apply \eqref{hxizetau} we used the fact that $\big|\lfloor (1-2\rho)S\rfloor \pm \varepsilon S\big|]\leq S(1-\delta)$ as follows from the restrictions we assumed on $\varepsilon$ and $\rho$.

Now, turning to $\bfeta_0$, recall that it is $\Upsilon_{\varepsilon}^{(\rho)}$ distributed on $\llbracket -\varepsilon S, \varepsilon S\rrbracket$ and hence by \Cref{distributionconcentration} there exists $c>0$ such that for any $\kappa \in [12, 12 U]$,
	\begin{flalign}
		\label{cprobability}
		\PP \big[ \mathsf{C} (\kappa) \big] \ge 1 - 2(2S+1)^2 e^{ -\kappa^{2}/{144} }\geq 1-  c^{-1} S^2 e^{- c\kappa},
	\end{flalign}
where the event $\mathsf{C} (\kappa)$ is defined by
	\begin{flalign*}
		 \mathsf{C} (\kappa) = \bigg\{   \displaystyle\max_{|X/S| \le \varepsilon}  \Big| \h_0 (X; \bfeta) - \displaystyle\int\limits_X^0 \Big( \rho - \displaystyle\frac{z}{2S} \Big) dz \Big| \le \displaystyle\frac{\kappa S^{2/3}}{12} + 1 \bigg\}.
	\end{flalign*}
In applying \Cref{distributionconcentration} we use the fact that $S^{1/2}<S^{2/3}$ for $S>1$ and instead bound $\mathsf{C}(\kappa)$ with the term $S^{2/3}$ replaced by $S^{1/2}$.
The $1$ on the right-hand side of the inequality in $\mathsf{C}$ takes into account the potential effect of replacing the summation in \eqref{hxy2} by the above integral.
In our next deduction, however, we will use the fact that $\frac{\kappa S^{1/2}}{12}+1\leq \frac{\kappa S^{1/2}}{6}$ since we have assumed $\kappa>15$ and $S>1$.

By definition, $\bfeta_0 (x) = 0 = \bfzeta_0 (x)$ for $x \notin \llbracket-\varepsilon S, \varepsilon S\rrbracket$, thus combining \eqref{bprobability} and \eqref{cprobability} yields that there exists $c=c(\delta)>0$ such that for all $\kappa \in [12, 6S]$
\begin{flalign}\label{eqDbound}
\PP \big[ \mathsf{D} (\kappa) \big] \ge 1 - c^{-1} S^2 e^{- c\kappa},
\end{flalign}
where the event $\mathsf{D} (\kappa)$ is defined by
	\begin{flalign}
		\label{probabilityd}
		 \mathsf{D} (\kappa) =  \Big\{ \displaystyle\max_{X\in \mathbb{Z}} \Big| \h_0 (X; \bfeta) - \h_0 (X; \bfzeta) \Big| \le \displaystyle\frac{\kappa S^{2/3}}{3} \bigg\}.
	\end{flalign}
By the second part of \Cref{xizeta2}, we may couple $\bfeta_t$ and $\bfzeta_t$ such that
	\begin{flalign*}
 \textbf{1}_{\mathsf{D} (\kappa)} \displaystyle\sup_{t \ge 0} \displaystyle\max_{X \in \mathbb{Z}} \Big| \h_t (X; \bfeta) - \h_t (X; \bfzeta) \Big| \le \displaystyle\frac{\kappa S^{2/3}}{3},
	\end{flalign*}
holds almost surely for all $t>0$. Combined this with \eqref{hha}, along with the fact that $\big[- \frac{\varepsilon S}{4}, \frac{\varepsilon S}{4} \big] \subseteq [-(\varepsilon S - 4RT), \varepsilon S - 4RT]$ (as $T^{-1} S \ge \delta^{-2} \ge 4 \varepsilon^{-1} \delta^{-1} \ge 64 \varepsilon^{-1} R$) yields
\begin{flalign}\label{eqADmax}
\textbf{1}_{\mathsf{A}} \textbf{1}_{\mathsf{D} (\kappa)} \displaystyle\max_{|X/S| \le \varepsilon / 4} \Big| \h_T(X; \bfeta) - \h_{S+T} \big(\big\llbracket  X+\lfloor (1 - 2 \rho) S \rfloor,\lfloor (1 - 2 \rho) S \rfloor\big\rrbracket ; \bfxi\big) \Big| < \displaystyle\frac{\kappa S^{2/3}}{3}.
	\end{flalign}

Finally, let us define the event
\begin{align}\label{eqEevent}
		\mathsf{E} (\kappa) = \bigg\{ \displaystyle\max_{|X/T| \le \varepsilon / 4} \Big| & \h_{S + T} \big( \llbracket (1 - 2 \rho) S + X, (1 - 2 \rho) S\rrbracket ; \bfxi \big) \\
\nonumber		& \quad - T \displaystyle\int\limits_{X/S}^0 \Big( \rho + \displaystyle\frac{T}{2 (S + T)} (1 - 2 \rho - z) \Big) dz \Big| \le \displaystyle\frac{\kappa (S + T)^{2/3}}{12} \bigg\}.
\end{align}
From the $(U; X, Y) = \big( S + T, X+\lfloor (1 - 2 \rho) S \rfloor, \lfloor (1 - 2 \rho) S \rfloor\big)$ case of \eqref{hxizetau}, we have that there exists $c=c(\delta)>0$ such that for all $\kappa \in [12, 12(S+T)]$,
\begin{flalign}
\label{probabilitye}
\PP \big[ \mathsf{E} (\kappa) \big] \ge 1 - c^{-1} S^2 e^{ -c\kappa}.
\end{flalign}
In fact, when applying \eqref{hxizetau} we initially have $(S+T)$ on the right-hand side, but since $S\geq \delta^{-2}T$ by assumption, we can replace this by $S$ up to a $\delta$-dependent constant. Furthermore, in applying \eqref{hxizetau} we arrive at a slightly different form for the integral in $\mathsf{E} (\kappa)$, namely
	\begin{flalign*}
		(S + T) \displaystyle\int\limits_{(\lfloor (1 - 2 \rho) S \rfloor + X)/(S+T)}^{\lfloor (1 - 2 \rho) S \rfloor/(S+T)} \Big( \displaystyle\frac{1-w}{2}\Big) dw = T \displaystyle\int\limits_{X/T}^0 \Big( \rho + \displaystyle\frac{T}{2(S + T)} (1 - 2 \rho - z) \Big) dz + \textrm{Error}
	\end{flalign*}
where the equality is facilitated through the change of variables $z = T^{-1} (S + T)w - T^{-1} \lfloor (1 - 2 \rho) S \rfloor$) and the error (which comes from replacing $\lfloor (1 - 2 \rho) S \rfloor$ by $(1 - 2 \rho) S$ after the change of variables) is bounded in magnitude by $\frac{T}{2(S+T)}$. That error term can be absorbed, as in the case of $\mathsf{C}(\kappa)$ in \eqref{cprobability}, via the triangle inequality. This yields \eqref{probabilitye}.

Combining \eqref{eqADmax} with \eqref{eqEevent} and using Bonferroni's inequality (and the fact that under our assumptions $\kappa (S+T)^{2/3}/12 <\kappa S^{2/3}/6$) we see the first inequality below
\begin{align*}
\PP \bigg[ \displaystyle\max_{|X/S| \le \varepsilon / 4} \Big| \h_T (X; \bfeta) -& T \displaystyle\int\limits_{X/T}^0 \Big(\rho + \displaystyle\frac{T}{2 (S + T)} (1 - 2 \rho - z) \Big) dz \Big| \ge \displaystyle\frac{\kappa S^{2/3}}{2} \bigg] \\
		& \ge \PP [\mathsf{A}] + \PP \big[ \mathsf{D} (\kappa) \big] + \PP \big[ \mathsf{E} (\kappa) \big] - 2\geq 1 - c^{-1} S^2 e^{-c\kappa},
	\end{align*}
while the second (which holds for some $c=c(\delta)>0$) uses  \eqref{aevent}, \eqref{eqDbound}, and \eqref{probabilitye}.

\Cref{hetalinear} \eqref{hetaxilambdarholinear} involves a maximum over both $|X/S| \le \varepsilon / 4$ and $|Y/S| \le \varepsilon / 4$. This result follows from the above inequality by the triangle inequality and union bound.
\end{proof}

\section{Linear Trajectories of Second Class Particles and proof of \Cref{xtlimit}}
\label{Linear}

Recall from the beginning of \Cref{sec:intro} that $\ASEP_t=(\bfeta_t,\bfX_t)$ denotes ASEP started with first class particles at every site of $\mathbb{Z}_{\leq -1}$, a single second class particle started at the origin, and all other site empty.
Let $\ASEPF_s$ denote the $\sigma$-algebra generated by $\ASEP_t$ up to and including time $s$, for $s \in \mathbb{R}_{\ge 0}$. For any event $\mathsf{E}$, we will write $\PP[\mathsf{E}|\ASEP_s] := \EE[\mathbf{1}_{\mathsf{E}}|\ASEPF_s]$ for the conditional probability of $\mathsf{E}$ given $\ASEPF_s$.
In Section \ref{couple} we will prove the following.

\begin{prop}
\label{xti}
For any $S>2$ let $T=S (\log S)^{-1}$ and define the $\ASEPF_{S}$-measurable random variable $\bfrho_{S} \in \mathbb{R}$ by the relation $1 - 2 \bfrho_{S} = S^{-1} \bfX_{S}$,  the $\varepsilon$-dependent event
\begin{equation}\label{eq:hspsbdrho}
\mathsf{P}_S := \{\bfrho_{S} \in (\varepsilon, 1 - \varepsilon)\}
\end{equation}
 and the $\ASEPF_{S+T}$-measurable events
\begin{flalign*}
\mathsf{E}^{\geq}_{S} &:= \Big\{ \bfX_{S+T} - \bfX_{S} \ge (1 - 2 \bfrho_{S}) T - S^{1 -1/200} \Big\},\\
\mathsf{E}^{\leq}_{S} &:= \Big\{ \bfX_{S+T} - \bfX_{S} \le (1 - 2 \bfrho_{S}) T + S^{1 - 1/200} \Big\},
\end{flalign*}
and $\mathsf{E}_{S}:= \mathsf{E}^{\geq}_{S}\cap \mathsf{E}^{\leq}_{S}$.
Then, for any $\varepsilon \in ( 0, 1/4 )$, there exists $c = c (\varepsilon) > 0$ and a $\ASEPF_{S}$-measurable event $\mathsf{H}_{S}$ such that and all $S>2$ we have
\begin{equation}\label{eq:hspsbd}
\PP[\mathsf{P}_S\cap (\mathsf{H}_{S})^c] \leq c^{-1} e^{-c S^{1/12}}\qquad \textrm{and} \quad \PP[\mathsf{E}_{S} | \ASEPF_{S}] \ge (1 - c^{-1} S^{-1/5})\mathbf{1}_{\mathsf{H}_{S}\cap \mathsf{P}_S}.
\end{equation}
%
%
The constants $c=c(\varepsilon)$ can be chosen so as to weakly decrease as $\varepsilon$ decreases to 0.
%
%
\end{prop}

The following is a corollary of \Cref{xti}.
\begin{cor}\label{cor:almostthere}
Define
$$
U^{\inf} = \liminf_{t\to \infty} \frac{\bfX_t}{t},\quad U^{\sup} = \limsup_{t\to \infty} \frac{\bfX_t}{t}, \quad
\mathsf{L}_{\varepsilon} = \big\{ |U^{\inf}-U^{\sup}|<\varepsilon\big\}.
$$
Then, there exists $c>0$ such that for all $\varepsilon\in (0,1/4)$,
$
\PP[\mathsf{L}_{\varepsilon} ]>1-c\varepsilon.
$
\end{cor}
Before proving this, let us see how this readily implies \Cref{xtlimit}.
\begin{proof}[Proof of \Cref{xtlimit}]
Observe that for $\varepsilon'<\varepsilon$, $\mathsf{L}_{\varepsilon}\subseteq \mathsf{L}_{\varepsilon'}$. In other words, as $\varepsilon\to 0$ the events increase. Their intersection $\mathsf{L}=\cap_{\varepsilon\in (0,1/4)} \mathsf{L}_{\varepsilon}$ is equal to the event that $U^{\inf}=U^{\sup}$ which is exactly the event that $\lim_{t\to \infty} \frac{\bfX_t}{t}$ exists. By the aforementioned containment and the bound $\PP[\mathsf{L}_{\varepsilon} ]>1-c\varepsilon$ from \Cref{cor:almostthere} we see that $\PP[\mathsf{L}] = \lim_{\varepsilon\to 0} (1-c\varepsilon) = 1$, thus proving the almost limit, as desired.
\end{proof}

It remains to show how \Cref{cor:almostthere} follows from \Cref{xti}. The idea is to work with a set of times $S_m$ that grows so that $S_{m+1}= S_m + S_m/\log S_m$. Taking the first time $S_0$ large enough with probability like $1-2\varepsilon$ we have that $\bfrho_{S_0}= (1-S_0^{-1} \bfX_{S_0})/2$ lies within $(\varepsilon,1-\varepsilon)$ -- this is the event $\mathsf{P}_{S_0}$. From \Cref{xti} there exists a hydrodynamic event $\mathsf{H}_{S_0}$ which is exponentially likely on the event $\mathsf{P}_{S_0}$ such that on $\mathsf{P}_{S_0}$ and $\mathsf{H}_{S_0}$, the event $\mathsf{E}_{S}$ holds with probability like $1-c^{-1} S^{-1/5}$. On the event  $\mathsf{E}_{S}$, we can bound how much $\bfrho_{S_1}$ and $\bfrho_{S_0}$ can differ to be like $S_0^{-1/200}$. Then, we can iterate on each subsequent time $S_1$, $S_2$ and so on. Since the $S_m$ grow like $e^{\sqrt{m}}$, the total change in the $\bfrho_S$ as well as the total probabilistic error built up over each iteration can be made arbitrarily small. This shows that on the sequence of times $S_m$ we can show the claim of \Cref{cor:almostthere}. For intermediate times, we use a brutal Poissonian bound on the motion of ASEP particles to show that wandering cannot change the velocity much there either.

Before proving \Cref{cor:almostthere} we introduce the set of times involved in our multi-scale argument and some properties of functions of those times.

\begin{definition}\label{ti}
For any $S_0 \in \mathbb{R}_{\ge 2}$ define $T_m, S_m \in \mathbb{R}_{> 0}$ inductively as follows. For each $m\in\mathbb{Z}_{\ge 1}$, set $T_{m - 1} = T(S_{m-1})$ where $T(S):= S (\log S)^{-1}$ and set $S_m = S_{m - 1} + T_{m - 1}$.
We will make use of the following two properties of $T(S)$:
\begin{itemize}
\item[P$_1$] The function $S\mapsto S+T(S)$ is increasing for $S\geq 2$.
\item[P$_2$] $T(S)$  has a unique minimum for $S\geq 2$ at $S=e$ in which case $T(e) = e$  and $T(S)$ is increasing for $S>e$.
\end{itemize}
\end{definition}

The following lemma provides a lower bound on each $S_m$. It may be helpful to note that the recursion for $S_m$ is a discrete version of solving the differential equation $dS(m)/dm = S(m)/\log S(m)$ with $S(0)=S_0$, whose solution is $S(m)=\exp(\sqrt{2m+\log(S_0)^2})$.

\begin{lem}
	\label{rti}
	For each  $m\in\mathbb{Z}_{\geq 1}$, we have that $S_m \ge e^{\sqrt{m}}$. Moreover, for any real $\delta > 0$ and $\vartheta > 0$, there exists $D = D(\delta, \vartheta) > 1$ such that if $S_0>D$, then
\begin{flalign*}
  \textrm{{\bf(a)}}\, \sum_{m = 0}^{\infty} S_m^{-\vartheta} < \delta,\qquad \textrm{{\bf(b)}}\, \sum_{m = 0}^{\infty} e^{-\vartheta S_m} < \delta,\qquad  \textrm{{\bf(c)}}\, \sum_{m = 0}^{\infty} e^{-\vartheta T_m} < \delta.
\end{flalign*}
%
\end{lem}

\begin{proof}
We establish the first statement of the lemma (that $S_m \ge e^{\sqrt{m}}$) by induction on $m$. The base case $m=1$ is verified by using (P$_1$) and (P$_2$) to see that $S+T(S)$ is minimal at $S=2$ and exceeds $2+e\geq e^{\sqrt{1}}$ there. To show the induction in $m$, assume that $S_m \ge e^{\sqrt{m}}$ holds for $m = k$, for some $k\in\mathbb{Z}_{\geq 1}$. Then the induction follows from the inequalities
	\begin{flalign*}
	S_{k + 1} = S_k + T_k \ge e^{\sqrt{k}} ( 1 + k^{-1 / 2}) \ge e^{\sqrt{k + 1}}.
	\end{flalign*}
The first equality is by definition; the next inequality uses P$_1$ and the $k=m$ induction hypothesis; the final inequality follows since
$
\exp \big( (k + 1)^{1 / 2} - k^{1 / 2} \big) \le \exp \big(\frac{1}{2 k^{1 / 2}}\big) \le 1 + k^{-1 / 2}.
$
Here, the first inequality relies upon writing $(k+1)^{1/2}-k^{1/2} = k^{1/2}(1+k^{-1})^{1/2}-k^{1/2}$ and the inequalities $(1+x)^{1/2} <1+x/2$ and $x<e^x$ (both  for $x>0$); the second inequality is equivalent to $(2 k^{1/2})^{-1} \leq \log(1+k^{-1/2})$ which follows from $x/2 <\log (1+x)$ for $x\in (0,1)$.

Turning to {\bf (a)} and {\bf (b)}, observe that we now know that $S_m\geq \max(S_0,e^{\sqrt{m}})$. Thus
$$
 \sum_{m = 0}^{\infty} S_m^{-\vartheta} \leq  \sum_{m = 0}^{\infty} \min\big(S_0^{-\vartheta},e^{-\sqrt{m}\vartheta}\big)\qquad
 \sum_{m = 0}^{\infty} e^{-\vartheta S_m} \leq \sum_{m = 0}^{\infty} e^{-\vartheta\min(S_0,e^{\sqrt{m}})}.
$$
In both of these expressions it is clear that as $S_0$ goes to infinity, each summand goes to zero. Additionally, if we drop the $S_0$ term each summation is finite. Hence, by the dominated convergence theorem, each summation goes to zero as $S_0$, and thus taking $S_0$ large enough we can upper bound each sum by $\delta$ as desired. The argument for {\bf (c)} follows similarly. Since $S_0\geq 2$, combining (P$_1$) and (P$_2$) we also see that for $m\in\mathbb{Z}_{\geq 1}$, $T_m \geq S_0(\log S_0)^{-1}$. On the other hand, we also know that the function $S\mapsto T(S)$ monotonically increases as $S$ increases and thus, by the first part of the lemma which gives $S_m\geq e^{\sqrt{m}}$, we have that $T_m =T(S_m)\geq T(e^{\sqrt m}) =e^{\sqrt{m}}/\sqrt{m}$. Using $T_m \geq \max\big(S_0(\log S_0)^{-1},e^{\sqrt{m}}/\sqrt{m}\big)$ and the dominated convergence theorem yields {\bf (c)}.
\end{proof}

\begin{proof}[Proof of \Cref{cor:almostthere}]
For the duration of this proof let $\mathsf{P}^{\varepsilon}_S$ and $\mathsf{H}^{\varepsilon}_S$ denote the events $\mathsf{P}_S$, and $\mathsf{H}_S$ coming from a particular value of $\varepsilon$ (this dependence was implicit in the notation used elsewhere). For a given $S_0>2$ and $\varepsilon_0\in (0,1/4)$, define recursively for $m\in \Z_{\geq 1}$
$$
\varepsilon_m = \varepsilon_{m-1} - S_{m-1}^{-1/200}.
$$
For a given $\varepsilon\in (0,1/4)$, it follows from \Cref{rti} that there exists $D=D(\varepsilon)>0$ such that for all $S_0>D$
\begin{equation}\label{eq:threeeqs}
\sum_{m=0}^{\infty} S_m^{-1/200} < \varepsilon/4, \quad \sum_{m=0}^{\infty} c^{-1} S_m^{-1/5} < \varepsilon/2, \quad \sum_{m=0}^{\infty} c^{-1} e^{-c S_m^{1/12}} < \varepsilon/2,\quad\sum_{m=0}^{\infty} e^{-T_m} < \varepsilon/2.
\end{equation}
where $c=c(\varepsilon)$ is given by \Cref{xti}.
For $k\in \Z_{\geq 0}$ define the event
\begin{equation}
\widetilde{\mathsf{L}}^{\varepsilon_0}_{S_0}(k)=\bigcap_{m=0}^{k-1} \mathsf{P}_{S_m}^{\varepsilon_m} \cap \mathsf{H}_{S_m}^{\varepsilon_m} \cap \mathsf{E}_{S_m}
\end{equation}
with the convention that $\widetilde{\mathsf{L}}^{\varepsilon}_{S_0}(0)=\Omega$, the full sample space, and that
$\widetilde{\mathsf{L}}^{\varepsilon}_{S_0}:=\widetilde{\mathsf{L}}^{\varepsilon}_{S_0}(\infty)$ is the infinite intersection. We make two claims:

\smallskip
\noindent {\bf Claim 1:} For all $\varepsilon\in (0,1/4)$ there exists $D=D(\varepsilon)>0$ so that for all $S_0>D$,
\begin{equation}\label{eq:Lcvarep}
\PP[\widetilde{\mathsf{L}}^{\varepsilon}_{S_0}]\geq 1-4 \varepsilon.
\end{equation}

\smallskip
\noindent {\bf Claim 2:} Let $\mathsf{W}^{\varepsilon}_{S,S'} = \big\{\sup_{S\leq s<s'\leq S'}|\bfX_{s}/s-\bfX_{s'}/s'|> \varepsilon/2\}$. For all $\varepsilon\in (0,1/4)$ there exists $D>D(\varepsilon)>0$ so that for all $S_0>D$,
\begin{equation}\label{eq:Wvareps}
\sum_{m=0}^{\infty} \PP[\mathsf{W}^{\varepsilon}_{S_m,S_{m+1}}]< \varepsilon.
\end{equation}

Assuming these  claims, let us complete the proof of \Cref{cor:almostthere}. Assume that $\varepsilon=\varepsilon_0\in (0,1/4)$ is given and $D=D(\varepsilon)>0$ is suitably large so that for all $S_0>D$, \eqref{eq:threeeqs}, \eqref{eq:Lcvarep} and \eqref{eq:Wvareps} hold. This implies that
\begin{equation}\label{eq:LWcap}
\widetilde{\mathsf{L}}^{\varepsilon}_{S_0}\cap \bigcap_{m=0}^{\infty} \left(\mathsf{W}^{\varepsilon}_{S_m,S_{m+1}}\right)^c
\end{equation}
holds with probability at least $1-5\varepsilon$. Assume below that this event \eqref{eq:LWcap} holds.

On the event $\mathsf{E}_{S_m}$, we have that
\begin{equation}\label{eq:rhodiff}
\left|\frac{\bfX_{S_{m+1}}}{S_{m+1}}-\frac{\bfX_{S_{m}}}{S_{m}}\right| \leq S_m^{-1/200}.
\end{equation}
By \eqref{eq:threeeqs}, the right-hand side summed over $m\in \Z_{\geq 0}$ is bounded above by $\varepsilon/4$. Thus, on  the event in \eqref{eq:LWcap} it follows that
$$
\sup_{m,m'\in \Z_{\geq 0}} \left|\frac{\bfX_{S_{m}}}{S_{m}}-\frac{\bfX_{S_{m'}}}{S_{m'}}\right| \leq \varepsilon/2.
$$
This controls the maximal change in $\bfX_S/S$ on the set of times $S_0,S_1,\ldots$. This is complemented by the control on intermediate wiggling that is afforded to us by  the intersection of the events $\left(\mathsf{W}^{\varepsilon}_{S_m,S_{m+1}}\right)^c$. Combined, this implies that on the event in \eqref{eq:LWcap}
$$
\sup_{s,s'\geq D} \left|\frac{\bfX_{s}}{s}-\frac{\bfX_{s'}}{s'}\right| \leq \varepsilon
.
$$
This implies that on the event in \eqref{eq:LWcap},  $U^{\inf}$ and $U^{\sup}$ differ by at most $\varepsilon$. Since the probability of the event in \eqref{eq:LWcap} is at least $1-5\varepsilon$, \Cref{cor:almostthere} follows.

What remains is to prove the two claims from above.

\noindent{\bf Proof of Claim 1.}
Observe that
\begin{flalign*}
\PP[\widetilde{\mathsf{L}}^{\varepsilon_0}_{S_0}] =& \PP[\mathsf{P}^{\varepsilon_0}_{S_0}] -
\PP[\mathsf{P}^{\varepsilon_0}_{S_0}\cap(\mathsf{H}^{\varepsilon_0}_{S_0})^c ]-
\PP[\mathsf{P}^{\varepsilon_0}_{S_0}\cap \mathsf{H}^{\varepsilon_0}_{S_0}\cap (\mathsf{E}^{\varepsilon_0}_{S_0})^c ]-\sum_{k=1}^{\infty} \PP[\widetilde{\mathsf{L}}^{\varepsilon_k}_{S_k}\cap (\mathsf{P}^{\varepsilon_k}_{S_k})^{c}]\\
&
-\sum_{k=1}^{\infty} \PP[\widetilde{\mathsf{L}}^{\varepsilon_k}_{S_k}\cap \mathsf{P}^{\varepsilon_k}_{S_k}\cap (\mathsf{H}^{\varepsilon_k}_{S_k})^c]
-\sum_{k=1}^{\infty} \PP[\widetilde{\mathsf{L}}^{\varepsilon_k}_{S_k}\cap \mathsf{P}^{\varepsilon_k}_{S_k}\cap \mathsf{H}^{\varepsilon_k}_{S_k}\cap (\mathsf{E}^{\varepsilon_k}_{S_k})^c].
\end{flalign*}
Observe that
$
\PP[\mathsf{P}_{S_0}^{\varepsilon_0}]> 1-3\varepsilon_0
$
provided $S_0$ is large enough  (as follows from the weak convergence of $\bfrho$ to a $U[0,1]$ random variable via \Cref{xt1}).
Observe now that for any $k\geq 1$,
$\PP[\widetilde{\mathsf{L}}^{\varepsilon}_{S_k}\cap (\mathsf{P}^{\varepsilon_k}_{S_k})^{c}]=0$. This is because the combination of the event $\mathsf{P}^{\varepsilon_{k-1}}_{S_{k-1}}$ and $\mathsf{E}^{\varepsilon_{k-1}}_{S_{k-1}}$ implies the event $\mathsf{P}^{\varepsilon_k}_{S_k}$ (this follows from \eqref{eq:rhodiff} which shows that $|\bfrho_k-\bfrho_{k-1}|\leq S_{k-1}^{-1/200}= \varepsilon_{k-1}-\varepsilon_{k}$).
Observe that for any $k\geq 0$,
$$
\PP\big[\widetilde{\mathsf{L}}^{\varepsilon_k}_{S_k}\cap \mathsf{P}^{\varepsilon_k}_{S_k}\cap (\mathsf{H}^{\varepsilon_k}_{S_k})^c\big] \leq \PP\big[\mathsf{P}^{\varepsilon_k}_{S_k}\cap (\mathsf{H}^{\varepsilon_k}_{S_k})^c\big]\leq c^{-1} e^{-c S_k^{1/12}}
$$
where the constant $c=c(\varepsilon_0)>0$ can be chosen the same for all $k$ (as follows from the final statement in \Cref{xti}).
Similarly observe that for any $k\geq 0$,
$$
\PP\big[\widetilde{\mathsf{L}}^{\varepsilon_k}_{S_k}\cap \mathsf{P}^{\varepsilon_k}_{S_k}\cap \mathsf{H}^{\varepsilon_k}_{S_k}\cap (\mathsf{E}^{\varepsilon_k}_{S_k})^c\big]=
\EE\Big[\mathbf{1}_{\widetilde{\mathsf{L}}^{\varepsilon_k}_{S_k}}\mathbf{1}_{ \mathsf{P}^{\varepsilon_k}_{S_k}\cap \mathsf{H}^{\varepsilon_k}_{S_k}} \EE\big[\mathbf{1}_{ (\mathsf{E}^{\varepsilon_k}_{S_k})^c}|\ASEPF_{S_k}\big]\Big]\leq c^{-1} S_k^{-1/5}
$$
where, as above, the constant $c=c(\varepsilon_0)>0$ can be chosen the same for all $k$.
The first equality is evident from conditional expectations, while the second relies on the equality
$
\mathbf{1}_{ \mathsf{P}^{\varepsilon_k}_{S_k}\cap \mathsf{H}^{\varepsilon_k}_{S_k}} \EE\big[\mathbf{1}_{ (\mathsf{E}^{\varepsilon_k}_{S_k})^c}|\ASEPF_{S_k}\big] =
\mathbf{1}_{ \mathsf{P}^{\varepsilon_k}_{S_k}\cap \mathsf{H}^{\varepsilon_k}_{S_k}} \big(1 - \EE\big[\mathbf{1}_{\mathsf{E}^{\varepsilon_k}_{S_k}}|\ASEPF_{S_k}\big]\big)
$
along with the second inequality in \eqref{eq:hspsbd} and the final statement in \Cref{xti}.

Putting together the above deductions we see that
$$
\PP[\widetilde{\mathsf{L}}^{\varepsilon_0}_{S_0}] \geq 1- 3\varepsilon_0 - \sum_{k=0}^{\infty} c^{-1} e^{-c S^{1/12}}-\sum_{k=0}^{\infty} c^{-1} S_k^{-1/5}\geq 1-4\varepsilon_0
$$
by the second and third inequalities in \eqref{eq:threeeqs}. This proves Claim 1.

\smallskip
\noindent{\bf Proof of Claim 2.}
We start by noting a brutal Poisson process bound on the second class particle. Recall that this particle moves left into an unoccupied site at rate $L$, and left into a site occupied by a first class particle at rate $R$ (this is the rate at which the first class particle moves right and switches places with the second class particle). Since $R>L$ by assumption, this implies that we can lower-bound the trajectory of $\bfX_S$ by a Poisson random walk that jumps to left at rate $L+R\leq 2R$. By similar reasoning, we can upper-bound $\bfX_S$ by another Poisson random walk that jumps to the right at rate $L+R\leq 2R$. Recall that for a Poisson$(\lambda)$ random variable $\bfZ$, if $x>\lambda$ then $\PP[\bfZ>x]\leq (e\lambda/x)^x e^{-\lambda}$.

Now, observe that by the union bound and triangle inequality
$$\PP[\mathsf{W}^{\varepsilon}_{S_m,S_{m+1}}] \leq 2\PP[\widetilde{\mathsf{W}}^{\varepsilon}_{S_m,S_{m+1}}]\quad\textrm{where}\quad
\widetilde{\mathsf{W}}^{\varepsilon}_{S_m,S_{m+1}} = \left\{\sup_{s\in [S_{m},S_{m+1}]}\left|\frac{\bfX_{S_m}}{S_m}-\frac{\bfX_{s}}{s}\right|> \frac{\varepsilon}{4}\right\}.
$$
Noting that
$$
\frac{\bfX_{S_m}}{S_m}-\frac{\bfX_{s}}{s} = \frac{\bfX_{S_m}-\bfX_s}{s}+ \frac{s-S_m}{sS_m}\bfX_{S_m}
$$
we see that
$$
\widetilde{\mathsf{W}}^{\varepsilon}_{S_m,S_{m+1}} \subseteq \bigg\{\sup_{s\in [S_{m},S_{m+1}]} \bigg|\frac{\bfX_{S_m}-\bfX_s}{s}\bigg|>\frac{\varepsilon}{8}\bigg\}\cup \bigg\{\sup_{s\in [S_{m},S_{m+1}]} \bigg|\frac{s-S_m}{s S_M} \bfX_{S_m}\bigg|>\frac{\varepsilon}{8}\bigg\}.
$$
By the brutal Poisson bound above, there exists a $D=D(\varepsilon)>0$ such that for all $S_0>D$,
$$
\PP\bigg[\sup_{s\in [S_{m},S_{m+1}]} \bigg|\frac{\bfX_{S_m}-\bfX_s}{s}\bigg|>\frac{\varepsilon}{8}\bigg] \leq
\PP\bigg[\sup_{s\in [S_{m},S_{m+1}]} \bigg|\frac{\bfX_{S_m}-\bfX_s}{T_m}\bigg|>\frac{\varepsilon}{8}\log S_m \bigg] \leq
e^{-T_m}.
$$
Similarly, we see that
$$
\PP\bigg[\sup_{s\in [S_{m},S_{m+1}]} \bigg|\frac{s-S_m}{s S_m} \bfX_{S_m}\bigg|>\frac{\varepsilon}{8}\bigg]\leq
\PP\bigg[\sup_{s\in [S_{m},S_{m+1}]} \bigg|\frac{\bfX_{S_m}}{S_m}\bigg|>\frac{\varepsilon}{8} \log S_m  \bigg]\leq
e^{-S_m}.
$$
Provided $D$ is large enough, the sum over $m$ of the above upper bounds $e^{-T_m}$ and $e^{-S_m}$ are bounded above by $\varepsilon/2$ which implies Claim 2 and completes the proof of \Cref{cor:almostthere}.
\end{proof}

\section{Proving \Cref{xti}}
\label{couple}

To prove \Cref{xti} (we focus on the $\mathsf{E}^{\geq}_S$ case as the $\mathsf{E}^{\leq}_{S}$ case follows immediately from particle-hole symmetry as in \Cref{rem:secondclassduality}) we start in \Cref{b} by coupling $\ASEP$ with a slightly different multi-species ASEP $\ASEPB$ obtained from $\ASEP$ at time $S$ by adding a number of second class particles to the left of $\bfX_{S}$. Appealing to \Cref{prop:Rez}, we can control the behavior of $\bfX_{S+T}$ in terms of the behavior of the bulk of the new second class particles. That behavior can be controlled by hydrodynamic estimates. All of this, however, requires that the time $S$ density profile in $\ASEP$ is close enough to its hydrodynamic limit. This condition is encapsulated in the event $\mathsf{H}_S$ (see \eqref{eq:Hsdef} in the proof of \Cref{zti}).

\begin{figure}[t]
  \includegraphics[width=.7\linewidth]{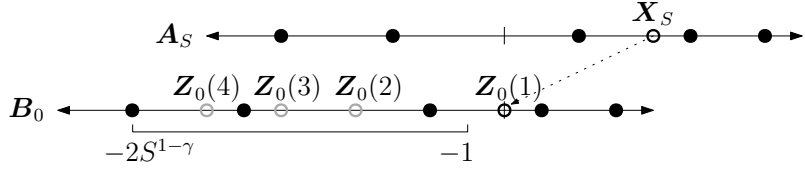}
  \caption{The coupling between $\ASEP_S$ and $\ASEPB_0$ from \Cref{b}. The second class particle (black circle) in $\ASEP_S$ moves to the origin and everything is translated relative to that, and additional second class particles (grey circles) are added with probabilities given in \eqref{probabilityb}.}
  \label{fig:ABcoupling}
\end{figure}

\begin{definition}
\label{b}
%
%
Fix $\gamma = \frac{1}{100}$ (i.e., something close enough to 0). Recall that in $\ASEP_S$ the second class particle is denoted by $\bfX_S$. Given the state of $\ASEP_S$ we define a new process $\ASEPB$ which is a multi-species ASEP with left jump rate $L$, right jump rate $R$, and the following initial data. Each site $j \in \mathbb{Z}$ is initially occupied in $\ASEPB_0$  by a first class particle if and only if $j + \bfX_S$ is occupied by a first class particle in $\ASEP_S$. Site 0 in $\ASEPB_0$ is initially occupied by a second class particle and, furthermore, for each site $j \in \llbracket - 2 S^{1 - \gamma}, - 1\rrbracket$ with $j + \bfX_S$ not occupied by a first class particle in $\ASEP_S$, $\ASEPB_0(j)$ contains a second class particle independently and with probability (see \Cref{bparticles} for an explanation of the choice of these probabilities and \Cref{rem:probs} regarding their positivity, and recall $\bfrho_S$ is defined in \Cref{xti})
	\begin{equation}
	\label{probabilityb}
	\left( S^{-\gamma} + \displaystyle\frac{j}{2S} \right) \left( 1 - \bfrho_S + \displaystyle\frac{j}{2S} \right)^{-1}.
	\end{equation}
Let
$\bfM$
equal the number of second class particles in $\ASEPB$. Denote their tagged positions at any time $t \ge 0$ by $\bfZ_t(1) > \cdots > \bfZ_t(\bfM)$, so that $\bfZ_0(1) = 0$. Set $\llbrace\bfZ_t\rrbrace = \big\{ \bfZ_t(1), \ldots , \bfZ_t(\bfM)\big\}$.

Equivalent to the above description, we let $\ASEPB_t=(\tilde\bfeta_t,\tilde\bfalpha_t)$ and assume initial data $\tilde\bfeta_0(j) = \bfeta(\bfX_S+j)$ for all $j\in \Z$, while for $\tilde\bfalpha_0$ we assume that $\tilde\bfalpha_0(0)=1$, and that for all $j \in \llbracket- 2 S^{1 - \gamma}, - 1\rrbracket $ with $\tilde\bfeta_0(j)=0$, the  $\tilde\bfalpha_0(j)$ are independent Bernoulli random variables with probability \eqref{probabilityb} of equaling $1$. For all other choices of $j$ set $\tilde\bfalpha_0(j)=0$. The Markov dynamics for $(\tilde\bfeta_t,\tilde\bfalpha_t)$ are those of first and second class particles under the basic coupling.

It will be convenient, e.g. in \Cref{LimitProcess}, for us to use $\ASEPB_t^{(1)}$ to denote the occupation variables for just the first class particles in $\ASEPB_t$ and $\ASEPB_t^{(1\cup 2)}$ to denote the occupation variables for the union of first and second class particles in $\ASEPB_t$, i.e. $\ASEPB_t^{(1)} = \tilde\bfeta_t$ and $\ASEPB_t^{(1\cup 2)} = \tilde\bfeta_t+\tilde\bfalpha_t$.

The above definition of $\ASEPB$ depends (i.e., is measurable with respect to $\ASEPF_S$) on the location $\bfX_S$ of the second class particle in $\ASEP_S$ and the associated hydrodynamic density $\bfrho_S$ defined by the relation $1-2\bfrho_S=\bfX_S/S$. We will also need notation where we define a version of $\ASEPB$ relative to a specified choice of $\bfrho_S$ and hence also $\bfX_S$. Let
\begin{equation}\label{eq:IEps}
I^{\varepsilon}_S= \left\{\rho\in (\varepsilon,1-\varepsilon): S(1-2\rho)\in\Z\right\},\qquad \textrm{and for }\rho\in I^{\varepsilon}\textrm{ let } X^{\rho}_S=S(1-2\rho).
\end{equation}
These represent the potential values of the random variables $\bfrho_S$ and $\bfX_S$ respectively. For such a $\rho\in I^{\varepsilon}_S$ and corresponding $X^{\rho}_S$, define $\ASEPB^{\rho}$ exactly as above but with $\bfrho_S$ and $\bfX_S$ replaced by the specified values $\rho$ and $X^{\rho}_S$. Similarly, let $\ASEPB^{(1),\rho}$ and $\ASEPB^{(1\cup 2),\rho}$ respectively denote the first class particle process, and union of first and second class particle processes. In this notation, $\ASEPB = \ASEPB^{\bfrho_S}$ where the variable $\rho$ is replaced by the random variable $\bfrho_S$. Recall that we are using the convention that bold symbols are random variables while their unbolded counterparts are deterministic variables.
\end{definition}

\begin{rem}
	\label{bparticles}

Let us briefly explain the choice of the probabilities in \eqref{probabilityb}. In view of the hydrodynamic limit for the ASEP with step initial data (as in \Cref{hetaxi} with $\rho=1$), the probability that a first class particle occupies a site $j \in \llbracket-\varepsilon S, \varepsilon S\rrbracket$ in $\ASEPB_0$ is approximately $\bfrho_S - \frac{j}{2S}$. Therefore, the probability that site $j$ is empty should approximately be $1 - \bfrho_S + \frac{j}{2S}$. So, \eqref{probabilityb} essentially ensures that the density of either first or second class particles in the interval $\llbracket -2S^{1 - \gamma}, -1\rrbracket$ in $\ASEPB_0$ is approximately constant and equal to $\bfrho_S + S^{- \gamma}$. In particular, the density of particles in $\ASEPB_0$ decreases linearly with slope $\frac{1}{2S}$ on $\llbracket -\varepsilon S, -2S^{1 - \gamma}\rrbracket$ to $\bfrho_S + S^{-\gamma}$ at $- 2S^{1 - \gamma}$, remains constant at $\bfrho_S + S^{-\gamma}$ on $\llbracket -2 S^{1 - \gamma}, -1\rrbracket$, discontinuously decreases to $\bfrho_S$ at site $0$, and then decreases linearly with slope $\frac{1}{2S}$ on $\llbracket 0, \varepsilon S\rrbracket$, see \Cref{fig:ASEPBdensity}.
\end{rem}

\begin{figure}[t]
  \includegraphics[width=.8\linewidth]{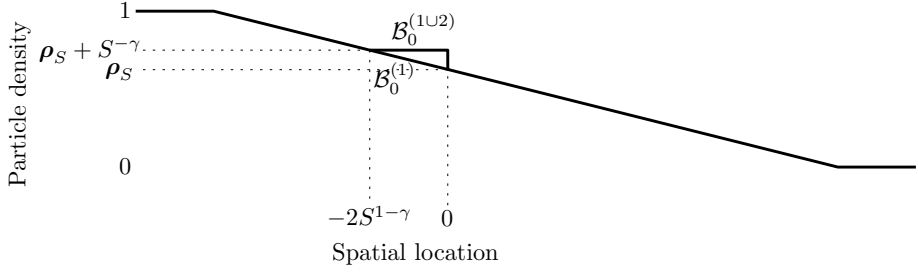}
  \caption{The average particle density versus spatial location for a typical instance of $\ASEPB_0$, as explained \Cref{bparticles}. Since there are only second class particles in $[-2S^{1-\gamma},0]$, the densities only differ therein. The upper line there corresponds to the density of the union of first and second class particles $\ASEPB_0^{(1\cup 2)}$ while the lower line is just for first class particles $\ASEPB_0^{(1)}$.}
  \label{fig:ASEPBdensity}
\end{figure}

\begin{rem}
\label{rem:probs}
Depending on the value of $\bfrho_S$ and $S$, the probabilities in  \eqref{probabilityb} may exceed $1$. However, for a given  value of $\varepsilon$ we can choose $c(\varepsilon)$ in the statement of \Cref{xti} small enough so that for $\bfrho_S\in (\varepsilon,1-\varepsilon)$, either the expressions in \eqref{probabilityb} remain bounded in $(\varepsilon/2,1-\varepsilon/2)$ for all relevant $j$, or $1-c^{-1} S^{-1/5}<0$. In the former case, the Bernoulli random variables are well-defined, while in the later case, the second claimed inequality in  \eqref{eq:hspsbd} in \Cref{xti} is trivially true (since the probability will always exceed $0$).
\end{rem}

Now observe that \eqref{eq:Rezleq}, from \Cref{prop:Rez}, implies that for any $y\in \Z$,
\begin{equation}\label{eq:XZcompare}
\PP\big[\bfX_{S+T} - \bfX_S\le y|\ASEPF_S \big] \leq \displaystyle\frac{1}{\bfM} \displaystyle\sum_{j = 1}^{\bfM} \PP^{\ASEPB_0}\big[ \bfZ_T(j) \le y \big].
\end{equation}
The left-hand side of this inequality is measurable with respect to $\ASEPF_S$ while the right-hand side is measurable with respect to the sigma algebra formed by $\ASEPF_S$ and the Bernoulli random variables used to form $\ASEPB_0$ from $\ASEP_S$. In particular, for any choice of the Bernoulli random variables, the inequality holds. We can rephrase the inequality \eqref{eq:XZcompare} in the following manner: Let $\bfK$ be uniformly distributed on $\{1,\ldots, \bfM\}$, then \eqref{eq:XZcompare} is equivalent to
\begin{equation}\label{eq:XZcompare2}
\PP\big[\bfX_{S+T} \ge \bfX_S + y|\ASEPF_S \big] \geq \PP\big[ \bfZ_T(\bfK) \ge y|\ASEPF_S \big].
\end{equation}
%
%

In light of \eqref{eq:XZcompare2}, we see that in order to establish \Cref{xti}, it suffices to control the locations of most of the second class particles in $\ASEPB$ with high probability. The following proposition achieves this aim.

\begin{prop}\label{zti}
For any  $\varepsilon \in ( 0,1/4)$, there exists $c=c(\varepsilon) > 0$ and $\ASEPF_S$-measurable events $\mathsf{H}_S$ such that for all $S>2$,
\begin{equation}\label{eq:HPbd}
\PP[\mathsf{P}_S\cap (\mathsf{H}_{S})^c] \leq c^{-1} e^{-c S^{1/12}}
\end{equation}
where $\mathsf{P}$ is defined in \eqref{eq:hspsbdrho} and
\begin{align}
\label{ztestimate}
\PP \Big[ \big| \llbrace\bfZ_T\rrbrace \cap \big[(1-2\bfrho_S)T-S^{1 - \frac{\gamma}{2}}, \infty\big) \big| \ge \bfM(1 - S^{-\frac{1}{5}}) \Big| \ASEPF_S \Big]  \ge\big(1 - c^{-1} e^{-cS^{1/12}}\big)\mathbf{1}_{\mathsf{H}_S\cap \mathsf{P}_S}.
\end{align}
The constants $c=c(\varepsilon)$ can be chosen so as to weakly decrease as $\varepsilon$ decreases to 0.
\end{prop}
This will be proved in \Cref{LimitProcess}, but first we prove  \Cref{xti} assuming it.

\begin{proof}[Proof of \Cref{xti}]
Let $\mathsf{H}_S$ and $c=c(\varepsilon)>0$ be given as in \Cref{zti}, in which case the first inequality in \eqref{eq:hspsbd} holds on account of \eqref{eq:HPbd}. We argue here that
\begin{equation}\label{eq:Egeqbd}
\PP[\mathsf{E}^{\geq}_{S} | \ASEPF_{S}] \ge (1 - c^{-1} e^{-c S^{1 / 12}})\mathbf{1}_{\mathsf{H}_{S}\cap \mathsf{P}_S}.
\end{equation}
Assuming this, we can deduce the same bound with $\mathsf{E}^{\leq}_{S}$. This is because after applying the particle-hole symmetry (\Cref{rem:secondclassduality}) to our process, the initial data remains unchanged and the events $\mathsf{E}^{\leq}_{S}$ and $\mathsf{E}^{\geq}_{S}$ swap.

To show \eqref{eq:Egeqbd}, assume that $\mathsf{H}_S\cap \mathsf{P}_S$ holds and let (recall $ \llbrace\bfZ_T\rrbrace$ defined below \eqref{probabilityb})
$$\Lambda = \llbrace\bfZ_T\rrbrace \cap \big[(1-2\bfrho_S)T-S^{1 - \frac{\gamma}{2}}, \infty\big)$$
and define the events
$$
\mathsf{F}_S = \big\{\bfZ_T(\bfK)\geq (1-2\bfrho_S)T-S^{1 - \frac{\gamma}{2}}\big\}, \quad \textrm{and}\quad
\mathsf{G}_S = \big\{|\Lambda|\geq \bfM(1-S^{-1/5})\big\}
$$
(recall that $\gamma =1/100$ and $\bfK$ is uniformly chosen on $\{1,\ldots, \bfM\}$).
From \eqref{eq:XZcompare2} it follows that
$\PP\big[\mathsf{E}^{\geq}_S| \ASEPF_S \big]\geq  \PP\big[\mathsf{F}_{S}| \ASEPF_S \big]$.
Since $\bfM=|\llbrace\bfZ_T\}\!\!\}|$, the event $\mathsf{G}_S$ says that the fraction of particles in $\llbrace\bfZ_T\rrbrace$ which lie in $[(1-2\bfrho_S)T-T^{1 - \frac{\gamma}{2}}, \infty)$ exceeds $1-S^{-1/5}$. The event $\mathsf{F}_S$ is that a randomly chosen particle in $\llbrace\bfZ_T\rrbrace$ lies in $\big[(1-2\bfrho_S)T-S^{1 - \frac{\gamma}{2}}, \infty\big)$. Thus, conditioned on $\mathsf{G}_S$, the probability of $\mathsf{F}_S$ exceeds $1-S^{-1/5}$. This implies that
$
\PP\big[\mathsf{F}_S| \ASEPF_S \big]\geq \PP(\mathsf{G}_S|\ASEPF_S) - S^{-1/5}
$
and by \Cref{zti}, $\PP(\mathsf{G}_S|\ASEPF_S)\geq 1-c^{-1} e^{-cS^{1/12}}$. Putting this all together shows that
$
\PP\big[\mathsf{E}^{\geq}_S| \ASEPF_S \big]\geq 1-c^{-1} S^{-1/5}
$
which yields the second inequality in \eqref{eq:hspsbd} as desired. The final sentence of \Cref{xti} follows from that of \Cref{zti}.
\end{proof}

\section{Proof of \Cref{zti}: reduction to a hydrodynamic limit estimate}

\label{LimitProcess}

It remains to establish \Cref{zti}. To this end, we will start by comparing the multi-class ASEP $\ASEPB$ from \Cref{b} to two versions of ASEP in \Cref{xizetaprocesses} ($\ASEPB^{(1)}$ will be compared to $\bfxi^{(1)}$ while $\ASEPB^{(1\cup 2)}$ will be compared to $\bfxi^{(1\cup 2)}$). The idea, developed in \Cref{xizetab} is that the height function for $\bfxi^{(1)}_0$ will be close (by close, we mean at most order $S^{3/4}$ apart with probability at least $1-c^{-1} e^{-c S^{1/12}}$) to that of $\ASEPB_0^{(1)}$ (the first class particles in $\ASEPB_0$), while the height function for $\bfxi^{(1\cup 2)}_0$ will be close to that of $\ASEPB_0^{(1\cup 2)}$ (the union of first and second class particles in $\ASEPB_0$). This event of height function closeness is part of the hydrodynamic event $\mathsf{H}_S$ which appears in the statement of \Cref{zti}. \Cref{xizetab} then shows that the simpler $\bfxi^{(1)}$ and $\bfxi^{(1\cup 2)}$ processes evolve over time $T=S/\log S$ to be close to the same hydrodynamic limit in the region $(-\infty , (1-2\bfrho_S)T - S^{1-\frac{\gamma}{2}})$. Since the number of second class particles is close to $S^{1-2\gamma}$ which is much larger than $S^{3/4}$, this implies that most of the second class particles in $\ASEPB^{(1\cup 2)}$ are in the complementary region $[(1-2\bfrho_S)T - S^{1-\frac{\gamma}{2}},\infty)$ which is exactly what we seek to show in \Cref{zti}.

The processes $\ASEPB^{(1)}_t$,  $\ASEPB^{(1\cup 2)}_t$, $\bfxi^{(1)}_t$ and $\bfxi^{(1\cup 2)}_t$ all depend on the random variable $\bfrho_S$ (recall from the beginning of \Cref{couple}). In order to make the comparisons mentioned above, we will instead consider $\ASEPB^{(1),\rho}_t$,  $\ASEPB^{(1\cup 2),\rho}_t$, $\bfxi^{(1),\rho}_t$ and $\bfxi^{(1\cup 2),\rho}_t$ for deterministic values of $\rho\in I^{\varepsilon}_S$ (recall from \eqref{eq:IEps}). Taking a union bound over all potential values of $\rho$ we establish that for random $\bfrho_S$, the comparison likewise holds.

\begin{definition}
	\label{xizetaprocesses}
For $\rho\in (\varepsilon,1-\varepsilon)$, let $\bfxi^{(1),\rho}_t$ and $\bfxi^{(1\cup 2),\rho}_t$ denote two versions of ASEP, each with left and right jump rates $L$ and $R$ and initial data given as follows (see also \Cref{fig:ASEPxizetadensity}). For each $j \notin \llbracket-\varepsilon S, \varepsilon S\rrbracket$, we deterministically set $\bfxi^{(1),\rho}_0 (j) = 0 = \bfxi^{(1\cup 2),\rho}_0 (j)$. To define $\bfxi^{(1),\rho}_0$ elsewhere, for each  $j \in \llbracket-\varepsilon S, \varepsilon S\rrbracket$, we define $\bfxi^{(1),\rho}_0 (j)$ according to independent Bernoulli random variables with probabilities
\begin{equation}\label{eq:xirho}
\PP \big[ \bfxi^{(1),\rho}_0 (j) = 1 \big] = \rho - \frac{j}{2S}, \qquad \PP \big[ \bfxi^{(1),\rho}_0 (j) = 0 \big] = 1 - \rho + \frac{j}{2S}.
\end{equation}
In the language of \Cref{distributedinitial}, this initial data is $\Upsilon^{(\rho)}_{\varepsilon}$-distributed on the interval $\llbracket-\varepsilon S, \varepsilon S\rrbracket$.
We define $\bfxi^{(1\cup2),\rho}_0 (j)$ for $j \in \llbracket-\varepsilon S, \varepsilon S\rrbracket$ so for each  $j \in \llbracket-\varepsilon S, \varepsilon S\rrbracket \setminus \llbracket -2 S^{1 - \gamma}, - 1\rrbracket$,
$$
\PP \big[ \bfxi^{(1\cup2),\rho}_0 (j) = 1 \big] = \rho - \frac{j}{2S}; \qquad \PP \big[ \bfxi^{(1\cup2),\rho}_0 (j) = 0 \big] = 1 - \rho + \frac{j}{2S},
$$
while for each  $j \in \llbracket -2 S^{1 - \gamma}, - 1\rrbracket$,
$$
	\PP \big[ \bfxi^{(1\cup2),\rho}_0 (j) = 1 \big] = \rho + S^{-\gamma} \qquad \PP \big[ \bfxi^{(1\cup2),\rho}_0 (j) = 0 \big] =  1 - \rho - S^{-\gamma}.
$$
Again, these choices are mutually independent over all $j$. Moreover, we assume that all of these Bernoulli random variables are chosen independent of the state of $\ASEPB_0$.

Finally, set $\bfxi^{(1)}_t=\bfxi^{(1),\bfrho_S}_t$ and $\bfxi^{(1\cup2)}_t=\bfxi^{(1\cup2),\bfrho_S}_t$, i.e., the processes just defined above but with $\rho$ replaced by $\bfrho_S$ determined by the location of the second class particle in $\ASEP_S$.
\end{definition}

\begin{figure}[t]
  \includegraphics[width=.8\linewidth]{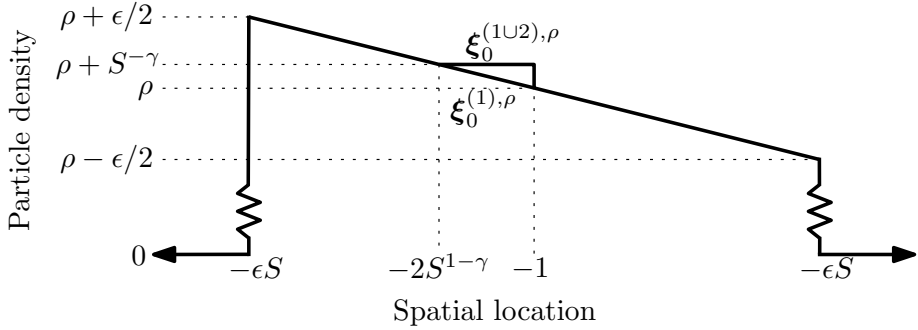}
  \caption{Particles in $\bfxi^{(1),\rho}_0$ and $\bfxi^{(1\cup2),\rho}_0$ (see \Cref{xizetaprocesses}) are initially present according to independent Bernoulli random variables with probabilities give by the plot shown here. The probabilities coincide for $\bfxi^{(1),\rho}_0$ and $\bfxi^{(1\cup2),\rho}_0$, except in the window $[-2S^{1-\gamma},-1]$ where the $\bfxi^{(1\cup2),\rho}_0$ probability remains flat and the $\bfxi^{(1),\rho}_0$ probability decreases linearly.}
  \label{fig:ASEPxizetadensity}
\end{figure}

Under these choices, we have the following lemma, which essentially states that $\bfxi^{(1)}_0$ initially approximates $\ASEPB_0^{(1)}$ and $\bfxi^{(1\cup2)}_0$ initially approximates $\ASEPB_0^{(1\cup 2)}$ (recall \Cref{b}).

\begin{prop}\label{xizetab}
For all $\varepsilon\in (0,1/4)$, there exists $c=c(\varepsilon)>0$ such that for
\begin{equation*}
\mathsf{D}_S(\ASEPB,\bfxi) = \big\{ \displaystyle\max_{|j| \le \varepsilon S} \big| \h_0 \big(j; \ASEPB \big) - \h_0 (j; \bfxi) \big| > S^{\frac{3}{ 4}}\big\},\qquad
\mathsf{M}_S= \big\{\big|\bfM - S^{1-2\gamma}\big| > S^{\frac{3}{ 4}} \big\}
\end{equation*}
and $\mathsf{P}_S$ as in \eqref{eq:hspsbdrho}, the following holds for any $S>2$:
\begin{align}
\label{eq:hbxi}
\PP \big[\mathsf{D}_S(\ASEPB^{(1)},\bfxi^{(1)})\cap \mathsf{P}_S \big] &< c^{-1} e^{- c S^{1/12}}, \\
\label{eq:hbzeta}
\PP \big[\mathsf{D}_S(\ASEPB^{(1\cup 2)},\bfxi^{(1\cup 2)})\cap \mathsf{P}_S \big] &< c^{-1} e^{- c S^{1/12}},\\
\label{eq:Mbdd}
\PP \big[\mathsf{M}_S\cap \mathsf{P}_S\big] &< c^{-1} e^{- c S^{1/12}}.
\end{align}
The constants $c=c(\varepsilon)$ can be chosen so as to weakly decrease as $\varepsilon$ decreases to 0.
\end{prop}

\begin{rem}
Note that in \Cref{xizetab}, $S$ comes into the definition of $\ASEPB_0$ since it determines that time at which we observe and modify the state of $\ASEP$;  $S$ comes into the definition of $\bfxi^{(1)}$ and $\bfxi^{(1,2)}$ in determining the parameters of the Bernoulli occupation variables; and $S$ comes into the definition of $\bfrho_S$ since $(1-2\bfrho_S)S = \bfX_S$. Also, note that for any $S_0>2$, by taking $C$ large enough and $c$ small enough, we can make the bounds in \Cref{xizetab} trivial for $S<S_0$ (i.e., make the right-hand side exceed $1$). We will use this in proving this result. Also, note that our proof of \eqref{eq:hbxi} and \eqref{eq:hbzeta} applies for  $S^{3/4}$ replaced by any power of $S$ exceeding $2/3$. We choose $3/4$ as it is sufficient for our purposes.
\end{rem}

\begin{proof}
Equation \eqref{eq:hbxi} follows readily from the triangle inequality and a union bound by combining \Cref{hetaxi} (which controls the deviations of the height function for $\ASEPB_0^{(1)}$ around its hydrodynamic limit) and \Cref{distributionconcentration} (which controls the deviation of the height function for $ \bfxi^{(1)}$ around its hydrodynamic limit). The proof of \eqref{eq:hbzeta} is  more involved since we need to track the effect of the additional particles added to go from $\ASEPB_0^{(1)}$ to $\ASEPB_0^{(1\cup 2)}$. We give the details below.

Recall $I^{\varepsilon}_S$ from \eqref{eq:IEps} and observe that the event on the left-hand side of \eqref{eq:hbzeta} satisfies
$$
\mathsf{D}_S(\ASEPB^{(1\cup 2)},\bfxi^{(1\cup 2)})\cap\mathsf{P}_S\subseteq \bigcup_{\rho\in I^{\varepsilon}_S} \mathsf{D}_S(\ASEPB^{(1\cup 2),\rho},\bfxi^{(1\cup 2),\rho})
$$
Since $|I^{\varepsilon}_S|$ is of order $S$, to establish  \eqref{eq:hbzeta} it suffices to show that there exists $c=c(\epsilon)>0$ such that for all  $\rho\in I^{\varepsilon}_S$ and $S>2$
\begin{equation}\label{eq:boundD1}
\PP\big[\mathsf{D}_S(\ASEPB^{(1\cup 2),\rho},\bfxi^{(1\cup 2),\rho})\big] \leq c^{-1} e^{- c S^{1/12}}.
\end{equation}
Observe that for any choice of function $P^{(1\cup 2),\rho}(j)$, we have
$$
\mathsf{D}_S(\ASEPB^{(1\cup 2),\rho},\bfxi^{(1\cup 2),\rho})\subseteq \mathsf{D}_S(\ASEPB^{(1\cup 2),\rho})\cap \mathsf{D}_S(\bfxi^{(1\cup 2),\rho})
$$
where ($\mathsf{D}_S(\bfxi^{(1\cup 2),\rho})$ is likewise defined with $\bfxi^{(1\cup 2),\rho}$ replacing $\ASEPB^{(1\cup 2),\rho}$)
$$
 \mathsf{D}_S(\ASEPB^{(1\cup 2),\rho})=  \big\{ \displaystyle\max_{|j| \le \varepsilon S} \big| \h_0 (j; \ASEPB^{(1\cup 2),\rho}) -P^{(1\cup 2),\rho}(j) \big| >S^{\frac{3}{4}}/2\big\}.
 $$
Thus, to prove \eqref{eq:boundD1} it suffices to find $P^{(1\cup 2),\rho}(j)$ such that there exists $c=c(\epsilon)>0$ so that for all  $\rho\in I^{\varepsilon}_S$ and $S>2$
\begin{equation}\label{eq:boundD1diff}
\PP\big[\mathsf{D}_S^{(1\cup 2),\rho;\beta}\big] \leq c^{-1} e^{- c S^{1/12}}\qquad \textrm{and}\qquad
\PP\big[\mathsf{D}_S^{(1\cup 2),\rho;\xi}\big] \leq c^{-1} e^{- c S^{1/12}}.
\end{equation}
We make the natural choice (defining $P^{(1\cup 2),\rho}([a,b])=P^{(1\cup 2),\rho}(a)-P^{(1\cup 2),\rho}(b)$ for $a,b\in \Z$)
$$
P^{(1\cup 2),\rho}(j) = \EE\big[ \h_0 (j; \bfxi^{(1\cup 2),\rho})\big]
$$
from which the second inequality in \eqref{eq:boundD1diff} follows immediately from applying Hoeffding's inequality (in the spirit of \Cref{distributionconcentration}). This gives a stronger bound with $S^{3/4}$ replace by $S^{1/2}$, though we will not need this here.

It remains to demonstrate the first bound in \eqref{eq:boundD1diff}. This follows from showing that there exists $c=c(\epsilon)>0$ such that for all $\rho\in I^{\varepsilon}_S$ and $S>2$
\begin{align}\label{eq:threeregions}
\nonumber &\PP\bigg[ \displaystyle\max_{j\in \llbracket 0,\varepsilon S\rrbracket} \Big| \h_0 \big(j; \ASEPB^{(1\cup 2),\rho} \big) -P^{(1\cup 2),\rho}(j) \Big| >\frac{ S^{3 / 4}}{6}\bigg] \leq c^{-1} e^{- c S^{1/12}},\\
&\PP\bigg[ \displaystyle\max_{j\in \llbracket -2S^{1-\gamma},-1\rrbracket} \Big| \h_0 \big(j; \ASEPB^{(1\cup 2),\rho} \big) -P^{(1\cup 2),\rho}(j) \Big| >\frac{ S^{3 / 4}}{6} \bigg] \leq c^{-1} e^{- c S^{1/12}},\\
\nonumber &\PP\bigg[ \displaystyle\max_{j\in \llbracket -\varepsilon S,-2S^{1-\gamma}\rrbracket} \Big| \h_0 \big(\llbracket j,-2S^{1-\gamma}\rrbracket; \ASEPB^{(1\cup 2),\rho} \big) -P^{(1\cup 2),\rho}(\llbracket j,-2S^{1-\gamma}\rrbracket) \Big| >\frac{ S^{3 / 4}}{6} \bigg] \leq c^{-1} e^{- c S^{1/12}}.
\end{align}
where in the final inequality we recall the notation from \eqref{eq:heightdiff}.
The first and third inequalities above are immediate from \eqref{eq:hbxi}: For $j\in \llbracket 0,\varepsilon S\rrbracket$ we have $\h_0 \big(j; \ASEPB^{(1\cup 2),\rho} \big)=\h_0 \big(j; \ASEPB^{(1),\rho} \big)$ and for $j\in \llbracket -\varepsilon S,-2S^{1-\gamma}\rrbracket$ we have $\h_0 \big(\llbracket j,-2S^{1-\gamma}\rrbracket; \ASEPB^{(1\cup 2),\rho} \big) =  \h_0 \big(\llbracket j,-2S^{1-\gamma}\rrbracket; \ASEPB^{(1),\rho} \big)$.


Thus, we are left to show the middle inequality in \eqref{eq:threeregions}. To do this we will split the interval $\llbracket -2S^{1-\gamma},-1\rrbracket$ into pieces of size $S^{2/3}$. On each of these we will control the number of first class particles in $\ASEPB^{(1\cup 2),\rho}$ to order $S^{1/3}$ by using the final part of \Cref{hetaxi} (as we are dealing with step initial data), and then control the number of second class particles by bounds on sums of Bernoulli random variables. This will yield an upper and lower bound with error of order $S^{1/3}$ on the number of first and second class particles in $\ASEPB^{(1\cup 2),\rho}$ within each interval. Summing over order $S^{1/3}$ such intervals introduces an error of order $S^{2/3}$ which is still much smaller than the $S^{3/4}$ allowed error.

Define $K_S = \lfloor 2S^{1/3-\gamma}\rfloor$ and intervals $I_k = \llbracket -(k+1) S^{2/3}, -kS^{2/3}\rrbracket$ for $k\in \llbracket 0,K_S-1\rrbracket$ and $I_{k_S} = \llbracket -2S^{1-\gamma},-K_S S^{2/3}\rrbracket$. Let $j_0,\ldots, j_{K_S}$ denote the endpoints of these intervals, i.e., $I_k = [j_{k+1},j_{k}]$ and notice that the union of these intervals covers $\llbracket -2S^{1-\gamma},-1\rrbracket$. Since $\h_0 \big(j; \ASEPB^{(1\cup 2),\rho} \big)$ and $P^{(1\cup 2),\rho}(j)$ are both 1-Lipschitz functions and since $S^{2/3}\ll S^{3/4}$ it suffices to show the following claim: there exist a constant $c>0$ such that for all $\rho\in I^{\varepsilon}_S$, $k\in \llbracket 0,K_S\rrbracket$ and $S>2$
\begin{equation}\label{eq:hojk}
\PP\bigg[\Big| \h_0 \big(j_k; \ASEPB^{(1\cup 2),\rho} \big) - P^{(1\cup 2),\rho}(j_k)\Big| > \frac{S^{3 / 4}}{8}\bigg] \leq c^{-1} e^{- c S^{1/12}},
\end{equation}
This implies the middle equation in \eqref{eq:threeregions} since the most that $\h_0 \big(j; \ASEPB^{(1\cup 2),\rho} \big) - P^{(1\cup 2),\rho}(j)$ can change over $j\in I_k$ is by $2|I_k| = 2S^{2/3}$. For large $S$, this is much smaller than $S^{3/4}/24$ (while for small $S$, we can just choose $c$ small enough so that the right-hand side of the middle equation in \eqref{eq:threeregions} exceeds 1, and hence the relation there trivially holds).

For each $k\in \llbracket 0,K_S\rrbracket$ define $\mathrm{First}_k = \h_0 \big(I_k; \ASEPB^{(1),\rho} \big)$ and the event
$$
\mathsf{F}_k(\kappa) := \left\{  \rho S^{2/3} +\frac{2k+1}{4} S^{1/3} -\kappa S^{1/3} \leq \mathrm{First}_k \leq  \rho S^{2/3} +\frac{2k+1}{4} S^{1/3} +\kappa S^{1/3}\right\}
$$
that the number of first class particles in $I_k$ is within $\kappa S^{1/3}$ of the expected number according to the hydrodynamic limit.
Noting that the term $ \rho S^{2/3} +\frac{2k+1}{4} S^{1/3}$ agrees with the hydrodynamic limit profile for step initial data, we see that by the final part of \Cref{hetaxi} there exists $c,\kappa_0>0$ such that for all $k\in \llbracket 0,K_S\rrbracket$ and $\kappa\in [\kappa_0,S^{2/3}/2]$,
$
\PP\big[\mathsf{F}_k(\kappa)\big]\leq c^{-1} e^{-c \kappa}.
$
On the event $\mathsf{F}_k(\kappa)$, we can bound the number of empty sites $\mathrm{Empty}_k:=S^{2/3}-\mathrm{First}_k$ for $\ASEPB^{(1),\rho}$ at time zero in the interval $I_k$ by
$$
(1-\rho) S^{2/3} - \frac{2k+1}{4}S^{1/3} - \kappa S^{1/3}\leq \mathrm{Empty}_k\leq   (1-\rho) S^{2/3} - \frac{2k+1}{4}S^{1/3} + \kappa S^{1/3}.
$$
As explained in \Cref{b}, in order to construct $\ASEPB^{(1\cup 2),\rho}$ from $ \ASEPB^{(1),\rho}$ on the interval $I_k$, we replace a hole at location $j$ by a second class particle (independently over all $j\in I_k$) with  the probability in \eqref{probabilityb}. Let us denote this probability by $Q(j)$. Observe that $Q(j)$ increases as $j$ decreases, and thus we can lower bound the total number of second class particles on $I_k$ by replacing $Q(j)$ by $Q(-kS^{2/3})$ for each $j\in I_k$, and likewise upper bound the number by using $Q(-(k+1)S^{2/3})$. This shows that given $\mathrm{Empty}_k$, the expected number second class particles that will be added in the interval $I_k$ will be bounded between $\mathrm{Empty}_k Q(-kS^{2/3})$ and $\mathrm{Empty}_{k}Q(-(k+1)S^{2/3})$. Call $\mathrm{Second}_k$ the number of second class particles added in the interval $I_k$ and define the event
$$
\mathsf{S}_k(\kappa) := \left\{ \mathrm{Empty}_k \cdot Q(-kS^{2/3}) -\kappa S^{1/3}\leq  \mathrm{Second}_k \leq \mathrm{Empty}_{k}\cdot Q(-(k+1)S^{2/3}) + \kappa S^{1/3}\right\}.
$$
By Hoeffding's inequality there exists $c,\kappa_0>0$ such that for all $k\in \llbracket 0,K_S\rrbracket$ and $\kappa\geq \kappa_0$,
$$
\PP\big[\mathsf{S}_k(\kappa)\big] \leq c^{-1} e^{-c \kappa}.
$$
On the event that both $\mathsf{F}_k(\kappa)$ and $\mathsf{S}_k(\kappa)$ hold, it follows that
$$
 (\rho + S^{-\gamma}) S^{2/3} - 4\kappa S^{1/3}\leq \mathrm{First}_k +\mathrm{Second}_k \leq (\rho + S^{-\gamma}) S^{2/3} + 4\kappa S^{1/3}
$$
where we have expanded the terms $Q(-kS^{2/3})$ and $Q(-(k+1)S^{2/3})$ and absorbed errors into the $4\kappa S^{1/3}$ term.
Recalling that $P^{(1\cup 2),\rho}(I_k) =(\rho + S^{-\gamma}) S^{2/3}$ and $\h_0 \big(I_k; \ASEPB^{(1\cup 2),\rho} \big)= \mathrm{First}_k +\mathrm{Second}_k$, and using the bounds above on $\PP\big[\mathsf{F}_k(\kappa)\big]$ and $\PP\big[\mathsf{S}_k(\kappa)\big]$, we conclude that there exists $c,\kappa_0>0$ such that for all $k\in \llbracket 0,K_S\rrbracket$ and $\kappa\in [\kappa_0/4,S^{2/3}/2]$,
\begin{equation}
\PP\bigg[\Big| \h_0 \big(I_k; \ASEPB^{(1\cup 2),\rho} \big) - P^{(1\cup 2),\rho}(I_k)\Big| >\kappa S^{1/3}\bigg] \leq c^{-1} e^{- c \kappa}.
\end{equation}
Taking $\kappa =S^{1/12}/8$ and a union bound over all  $k\in \llbracket 0,K_S\rrbracket$ leads to \eqref{eq:hojk}, as desired.

The inequality, \eqref{eq:Mbdd}, follows from what we have shown in \eqref{eq:threeregions} above upon noting that
$$\bfM = \h_0(\llbracket -2S^{1-\gamma},-1\rrbracket ;  \ASEPB^{(1\cup 2)}) -  \h_0(\llbracket -2S^{1-\gamma},-1\rrbracket ;  \ASEPB^{(1)}).$$
Notice that the centering of $\bfM$ by $S^{1-2\gamma}$ is consistent with the hydrodynamic limit, namely that the area of the triangle bounded between the two profiles in Figure \ref{fig:ASEPxizetadensity}.
\end{proof}

Having established \Cref{xizetab} we now know that the initial condition for the height functions of $\ASEPB^{(1)}$ and  $\bfxi^{(1)}$ as well as for $\ASEPB^{(1\cup 2)}$ and  $\bfxi^{(1\cup 2)}$ are, respectively, close to order $S^{3/4}$. The next result, \Cref{zetaestimate}, will show that the product form initial height profiles for $\bfxi^{(1)}$ and  $\bfxi^{(1\cup 2)}$ evolved over a time interval $T$ will be close to order at least $T^{3/4}$ to their hydrodynamic limits (at least when focusing to the left of the characteristic velocity $1-2\bfrho_S$). \Cref{zti} follow then follow by combining \Cref{zetaestimate} with \Cref{xizetab} and the monotonicity afforded to us by \Cref{xizeta2}.

\begin{prop}

\label{zetaestimate}

For any $\varepsilon \in ( 0, 1/2 )$, there exists $c = c (\varepsilon) > 0$ such that the following holds for any $S>2$ (recall $T=S(\log S)^{-1}$). Define the interval and function
\begin{flalign*}
\mathcal{J}_{S,T,\rho}\!= \!\Big[\! -\frac{\varepsilon S}{4}, (1 - 2 \rho) T - S^{1 - \frac{\gamma}{2}} \Big],\quad\!
\mathcal{H}_{S,T,\rho}(X,Y)\!=\! \bigg(\! \rho + \displaystyle\frac{T (1 - 2 \rho)}{2 (S + T)} \bigg) (Y - X)\! + \displaystyle\frac{Y^2 - X^2}{4 (S + T)},
\end{flalign*}
as well as the maximal deviation of the height function and hydrodynamic limit function
\begin{flalign*}
\mathrm{Diff}^{\pm}_{S,T}(\xi,\rho)&= \max_{X, Y \in \mathcal{J}_{S,T,\rho}} \pm \big( \h_T (\llbracket X,Y\rrbracket ; \xi) -\mathcal{H}_{S,T,\rho}(X,Y)\big),\\
\mathrm{Diff}_{S,T}(\xi,\rho) &= \max_{X, Y \in \mathcal{J}_{S,T,\rho}}  \big| \h_T (\llbracket X,Y\rrbracket ; \xi) -\mathcal{H}_{S,T,\rho}(X,Y)\big| =\max\big(\mathrm{Diff}^{+}_{S,T}(\xi,\rho),\mathrm{Diff}^{-}_{S,T}(\xi,\rho)\big)
\end{flalign*}
Then we have that
\begin{flalign}
\mathbb{P} \Big[\big\{\mathrm{Diff}_{S,T}(\bfxi^{(1)},\bfrho) \ge S^{3/4}\big\} \bigcap \big\{ \bfrho_S\in (\varepsilon,1-\varepsilon)\big\} \Big] & < c^{-1} e^{-c S^{1/12}},\label{xizetahb}\\
\mathbb{P} \Big[ \big\{\mathrm{Diff}_{S,T}(\bfxi^{(1\cup 2)},\bfrho) \ge S^{3/4} \big\} \bigcap \big\{ \bfrho_S\in (\varepsilon,1-\varepsilon)\big\} \Big]& < c^{-1} e^{-c S^{1/12}}.\label{xizetahb2}
\end{flalign}
The constants $c=c(\varepsilon)$ can be chosen so as to weakly decrease as $\varepsilon$ decreases to 0.
\end{prop}
\begin{proof}[Proof of \Cref{zetaestimate}]
As in the proof of \Cref{xizetab}, we will demonstrate that there exists   $c = c (\varepsilon) > 0$ such that the following holds for all $S>2$ and all $\rho\in I^{\varepsilon}_S$ (recall \eqref{eq:IEps}):
\begin{flalign}
\PP\Big[\mathrm{Diff}_{S,T}(\bfxi^{(1),\rho},\rho) \ge S^{3/4}\Big]&\leq   c^{-1} e^{-c S^{1/12}},\label{xizetahbfixedrho}\\
\PP\Big[\mathrm{Diff}_{S,T}(\bfxi^{(1\cup 2),\rho},\rho) \ge S^{3/4}\Big]&\leq   c^{-1} e^{-c S^{1/12}}.\label{xizetahb2fixedrho}
\end{flalign}
Having shown this, the results in the statement of \Cref{zetaestimate} follow by a union bound (absorbing the resulting linear prefactor of $S$ into the exponent $ c^{-1} e^{-c S^{1/12}}$).

By \Cref{xizetaprocesses}, the initial data for $\bfxi^{(1),\rho}$ is $\Upsilon_{\varepsilon}^{(\rho)}$-distributed (recall \eqref{functionlinear}) on $[-\varepsilon S, \varepsilon S]$. Thus, \eqref{xizetahbfixedrho} follows from the first statement of \Cref{hetalinear} (with $\kappa=S^{1/12}$ there), together with the fact that
\begin{flalign*}
T \displaystyle\int\limits_{X/T}^{Y/T} \Big( \rho + \displaystyle\frac{T}{2 (S + T)} (1 - 2 \rho - z) \Big) dz = \mathcal{H}_{S,T,\rho}(X,Y).
\end{flalign*}

To establish \eqref{xizetahb2fixedrho}, first observe by \Cref{xizeta1} that $\bfxi^{(1),\rho}$ and $\bfxi^{(1\cup 2),\rho}$ can be coupled so that $\h_t (\llbracket X,Y\rrbracket ; \bfxi^{(1\cup 2),\rho}) \ge \h_t (\llbracket X,Y\rrbracket ;\bfxi^{(1),\rho})$, for each $t \ge 0$, whenever $X \le Y$.
By this and \eqref{xizetahbfixedrho}, there exists $c = c(\varepsilon) > 0$ such that for all $S>2$ and all $\rho\in I^{\varepsilon}_S$ (recall \eqref{eq:IEps}):
$$
\PP\Big[\mathrm{Diff}^{-}_{S,T}(\bfxi^{(1\cup 2),\rho},\rho) \ge S^{3/4}\Big]\leq \PP\Big[\mathrm{Diff}^{-}_{S,T}(\bfxi^{(1),\rho},\rho) \ge S^{3/4}\Big]\leq   c^{-1} e^{-c S^{1/12}}.
$$
So, it suffices to establish the complementary bound
\begin{flalign}
\label{xyestimatexy}
\PP\Big[\mathrm{Diff}^{+}_{S,T}(\bfxi^{(1\cup 2),\rho},\rho) \ge S^{3/4}\Big]\leq   c^{-1} e^{-c S^{1/12}}.
\end{flalign}

To establish \eqref{xyestimatexy} observe that $\bfxi^{(1\cup 2),\rho}_0$  is $\Phi_{\varepsilon; \beta}^{(\rho)}$-distributed (as in \Cref{linearconstant}) with $\beta = \varepsilon^{-1} S^{-\gamma}$.
Thus, applying the second part of  \Cref{hetalinear} yields
\begin{flalign}
\nonumber\PP\Bigg[ \displaystyle\max_{\substack{|X/S| \le \varepsilon / 4 \\ |Y/S| \le \varepsilon / 4}} \bigg| \h_T(\llbracket X,Y\rrbracket ; \bfxi^{(1\cup 2),\rho}) - T \displaystyle\int\limits_{X/T}^{Y/T} \max \Big\{ \rho + \displaystyle\frac{(1 - 2 \rho - z) T}{2 (S + T)},& \rho + S^{-\gamma} \Big\} dz \bigg| > S^{3/4} \Bigg]\\
& < c^{-1} e^{-c S^{1/12}}.\label{eq:htXYzetam}
\end{flalign}
Recall that we have assumed $X, Y \le (1 - 2 \rho) T - S^{1 - \frac{\gamma}{2}}$. For large enough $S$, we have that
$(1 - 2 \rho) T - S^{1 - \frac{\gamma}{2}} \le (1 - 2 \rho) T - 2 S^{-\gamma} (S + T)$. In that case,
\begin{flalign*}
 \displaystyle\int\limits_{X/T}^{Y/T} \max \Big\{ \rho + \displaystyle\frac{(1 - 2 \rho - z) T}{2 (S + T)}, \rho + S^{-\gamma} \Big\} dz & =  \displaystyle\int\limits_{X/T}^{Y/T} \Big( \rho + \displaystyle\frac{(1 - 2 \rho - z) T}{2 (S + T)} \Big) dz =\frac{\mathcal{H}_{S,T,\rho}(X,Y)}{T}.
\end{flalign*}
Combining this with \eqref{eq:htXYzetam} yields \eqref{xyestimatexy}, as desired.
\end{proof}

\begin{proof}[Proof of \Cref{zti}]
We will start by defining the $\ASEPF_S$-measurable event $\mathsf{H}_S$ (recall that $\mathsf{E}^c$ is the complement of an event $\mathsf{E}$):
\begin{equation}\label{eq:Hsdef}
\mathsf{H}_S= \mathsf{D}_S(\ASEPB^{(1)},\bfxi^{(1)})^c\cap \mathsf{D}_S(\ASEPB^{(1\cup 2)},\bfxi^{(1\cup 2)})^c \cap \mathsf{M}_S^{c}
\end{equation}
where these events (all of which also depend on $S$ but whose dependence is not explicit in the notation) are defined in \Cref{xizetab}. Recalling the notation $\mathsf{P}_S$ from \eqref{eq:hspsbdrho}
observe that by the union bound and then \eqref{eq:hbxi}, \eqref{eq:hbzeta} and \eqref{eq:Mbdd} we have that there exists $c=c(\varepsilon)>0$ such that for all $S>2$
\begin{align*}
\PP[\mathsf{P}_S\cap (\mathsf{H}_{S})^c]
&\leq \PP \big[\mathsf{P}_S \cap \mathsf{D}_S(\ASEPB^{(1)},\bfxi^{(1)})\big]+
\PP \big[\mathsf{P}_S \cap \mathsf{D}_S(\ASEPB^{(1\cup 2)},\bfxi^{(1\cup 2)})\big]+
\PP \big[\mathsf{P}_S \cap \mathsf{M}_S\big]\\
&\leq c^{-1} e^{- c S^{1/12}}.
\end{align*}
%
%
%
%
This shows \eqref{eq:HPbd}.
Thus, to prove \Cref{zti} it now suffices to show that for the choice of $\mathsf{H}_S$ in \eqref{eq:Hsdef}, \eqref{ztestimate} holds, namely there exists $c=c(\varepsilon)>0$ such that for all $S>2$
$$
\PP \Big[ \Big| \llbrace\bfZ_T\rrbrace \cap \big[(1-2\bfrho_S)T-S^{1 - \frac{\gamma}{2}}, \infty\big) \Big| \ge \bfM(1 - c^{-1} S^{-\frac{1}{5}}) \Big| \ASEPF_S \Big]  \ge\big(1 - c^{-1} e^{-c S^{1/12}}\big)\mathbf{1}_{ \mathsf{H}_S\cap\mathsf{P}_S}.
$$
In other words, to prove the above bound we must show that there exists $c=c(\varepsilon)>0$ such that for any $S>2$, assuming the event $\mathsf{P}_S\cap \mathsf{H}_S$ holds, it follows that with probability at least $1 - c^{-1} e^{-c S^{1/12}}$ the number of second class particles in the interval $ \big[(1-2\bfrho_S)T-S^{1 - \frac{\gamma}{2}}, \infty\big) $ is at least $\bfM(1 - c^{-1} S^{-\frac{1}{5}})$. Observe that on the event $\mathsf{H}_S\cap\mathsf{P}_S$, we have that $\big|\bfM - S^{1-2\gamma}\big| \leq S^{\frac{3}{4}}$ holds and that
\begin{align}
\PP\Big[ \Big| \llbrace\bfZ_T\rrbrace \cap \Big(-\infty, - \frac{\varepsilon S}{4}\Big]\Big|= 0\Big] &\geq 1- c^{-1} e^{-c S^{1/12}},\label{eq:llbraczt1}\\
\PP\Big[ \Big| \llbrace\bfZ_T\rrbrace \cap \Big(- \frac{\varepsilon S}{4}, (1-2\bfrho_S)T-S^{1 - \frac{\gamma}{2}}\Big]\Big|< 4S^{\frac{3}{4}}\Big] &\geq 1- c^{-1} e^{-c S^{1/12}}.\label{eq:llbraczt2}
\end{align}
The first of these inequalities follows immediately from  \Cref{xizetaequal} (and does not depend on the occurrence of $\mathsf{H}_S$). This is because  $\ASEPB^{(1)}$ and $ \ASEPB^{(1\cup 2)}$ are the same at time 0 on the interval $(-\infty, -2S^{1-\gamma})$ and hence remain the same on the smaller interval $(-\infty, -2S^{1-\gamma}-4RT)$ at time $T=S/\log S$ with probability at least $1-4e^{-T/3}$. We can find $c=c(\varepsilon)>0$ such that for all $S>2$ either $(-\infty, -2S^{1-\gamma}-4RT)\subset (-\infty, - \frac{\varepsilon S}{4}]$ and $1-4e^{-T/3}\geq 1-c^{-1} e^{-cS^{1/12}}$, or $1-c^{-1} e^{-cS^{1/12}}<0$. In the first case (which occurs for large enough $S$) \eqref{eq:llbraczt1} follows, and in the second case (for small $S$) \eqref{eq:llbraczt1} follows trivially as the right-hand side is negative.

The second inequality, \eqref{eq:llbraczt2}, relies on \Cref{zetaestimate}. Observe that by the triangle inequality, on the event that
\begin{equation}\label{eq:Diffevs}
\big\{\mathrm{Diff}_{S,T}(\bfxi^{(1)},\bfrho) < S^{3/4}\big\} \cap \big\{\mathrm{Diff}_{S,T}(\bfxi^{(1\cup 2)},\bfrho) < S^{3/4}\big\}
\end{equation}
holds in addition to $ \mathsf{H}_S\cap\mathsf{P}_S$, it follows that
\begin{equation}\label{eq:htcompars}
\h_T\Big(\Big\llbracket - \frac{\varepsilon S}{4}, (1-2\bfrho_S)T-S^{1 - \frac{\gamma}{2}}\Big\rrbracket; \ASEPB^{(1\cup 2)}\Big)-
\h_T\Big(\Big\llbracket - \frac{\varepsilon S}{4}, (1-2\bfrho_S)T-S^{1 - \frac{\gamma}{2}}\Big\rrbracket; \ASEPB^{(1)}\Big) \leq 4S^{3/4}.
\end{equation}
Here we used the monotonicity from \Cref{xizeta2} to show that the $S^{3/4}$ closeness of $\ASEPB^{(1)}$ and $\bfxi^{(1)}$, and of $\ASEPB^{(1\cup 2)}$ and $\bfxi^{(1\cup 2)}$, at time $0$ (which holds on $\mathsf{D}_S(\ASEPB^{(1)},\bfxi^{(1)})^c\cap \mathsf{D}_S(\ASEPB^{(1\cup 2)},\bfxi^{(1\cup 2)})^c$) persists for all time. Then we used the fact that on the event in \eqref{eq:Diffevs} both $\bfxi^{(1)}$ and $\bfxi^{(1\cup 2)}$ have height functions that are within $S^{3/4}$ of the same hydrodynamic limit function $\mathcal{H}_{S,T,\rho}(X,Y)$.
By \Cref{xizeta1} and equation \eqref{eq:heightdiff}, the left-hand side of \eqref{eq:htcompars} is  the event
$$
\Big| \llbrace\bfZ_T\rrbrace \cap \Big(- \frac{\varepsilon S}{4}, (1-2\bfrho_S)T-S^{1 - \frac{\gamma}{2}}\Big]\Big|< 4S^{\frac{3}{4}}
$$
whose probability we wish to control in \eqref{eq:llbraczt2}. By \Cref{zetaestimate} the probability of the event in \eqref{eq:Diffevs} (that, in conjunction with $\mathsf{H}_S \cap \mathsf{P}_S$, imply \eqref{eq:htcompars}) is at least $1- c^{-1} e^{-c S^{1/12}}$ for some $c=c(\varepsilon)>0$. This establishes \eqref{eq:llbraczt2}.

We can now show \eqref{ztestimate} holds. By \eqref{eq:llbraczt1} and \eqref{eq:llbraczt2}, on the event  $\mathsf{H}_S \cap \mathsf{P}_S$,  we have that
$$
\Big\{\Big| \llbrace\bfZ_T\rrbrace \cap \big[(1-2\bfrho_S)T-S^{1 - \frac{\gamma}{2}}, \infty\big) \Big| \ge \bfM  - 4S^{3/4}\Big\}
$$
holds with probability at least $1-2c^{-1} e^{-c S^{1/12}}$. On  $\mathsf{H}_S \cap \mathsf{P}_S$ we also have  $\bfM>S^{1-2\gamma} - S^{3/4}$ which implies that there exists $S_0>0$ such that
$$
\bfM - 4S^{3/4} = \bfM \Big( 1- \frac{4S^{3/4}}{\bfM}\Big) \geq \bfM\Big(1- \frac{4S^{3/4}}{S^{1-2\gamma}-S^{3/4}}\Big) \geq \bfM(1-S^{-1/5}).
$$
This implies \eqref{ztestimate}, provided $S>S_0$ (for smaller $S$ it follows by taking $c$ sufficiently small).

All that remains to completes the proof of  \Cref{zti} is to show that the constants $c=c(\varepsilon)$ in that statement can be chosen so as to weakly decrease as $\varepsilon$ decreases to 0. However, this is easily seen to be the case due to the fact that all results upon which we relied in this proof have a similar qualification on the constants.
\end{proof}
\appendix
\section{Rezakhanlou's good coupling}\label{sec:Rez}
We recall here a coupling which is presented in Section 4.1 of \cite{MSL}. It is proved there, though not stated as a quotable result, hence we also include a proof along the lines of \cite{MSL}. We will stick with the notation used in that paper to make the comparison there simpler. This notation is a bit different than what we use in the main body of this paper, hence we also explain how match to \Cref{prop:Rez}.

Let $(\eta_t,x_t)$ denote the occupation variables for first class particles ($\eta_t$) along with the location  of a single second class particle ($x_t$). Let $p(1)=q$ and $p(-1)=p$ (and $p(i)=0$ for all other $i$) and assume $p\geq q$. The state space for $(\eta_t,x_t)$ is $\{(\eta,x)\in \{0,1\}^{\Z}\times \Z: \eta(x)=0\}$ and the generator is specified by its action on local functions $f$ as
\begin{align*}
\mathcal{A}f (\eta,x) :=& \sum_{u,v\in \Z\setminus x} p(v-u)\eta(u) \big(1-\eta(v)\big) \big[f(\eta^{u,v},x)-f(\eta,x)\big]\\
&+ (p-q) \big(1-\eta(x-1)\big)\big[f(\eta,x-1)-f(\eta,x)\big]\\
&+ (p-q) \eta(x+1)\big[f(\eta^{x+1,x},x+1)-f(\eta,x)\big]\\
&+ q \big[f(\eta^{x,x+1},x+1)+f(\eta^{x,x-1},x-1)-2f(\eta,x)\big].
\end{align*}
Here $\eta^{u,v}(w)$ is equal to $\eta(v)$ if $w=u$, $\eta(u)$ if $w=v$ and $\eta(w)$ otherwise. Let $\PP^{\eta,x}$ denote the probability measure for this Markov process from initial data $(\eta,x)$.

It is easy to check that $\mathcal{A}$ does, indeed, encode the desired jump rates. The first term involves jumps which are separate from the second class particle. When there is a second class particle at $x$ and no (first class) particle at $x-1$, the second line gives a jump rate $p-q$ for the second class particle to move to $x-1$, and the fourth line gives a jump rate $q$, thus a total of rate $p$. If there is a second class particle at $x$ and no particle at $x+1$, then only the fourth line contributes a jump rate of $q$. Thus, the second class particle behaves a expected. If there is a particle at $x-1$ and second class particle at $x$, then the two switch at rate $q$ from the fourth line, and if there is a particle at $x+1$ and a second class particle at $x$, then the two switch with rate $p-q$ from the third line and rate $q$ from the fourth line, hence rate $p$. This matches the dynamics one expects for first/second particle pairs. This type of case-by-case verification of couplings can be implemented for all of the other generators that we define below, though we will not go through it there.

Consider two states $\eta_0\geq \zeta$ in $\{0,1\}^{\Z}$ and let $\alpha$ be defined via $\zeta_0 = \eta+\alpha$, with $\alpha\in \{0,1\}^{\Z}$ as well. These represent all of the second class particles. By the basic coupling (and attractivity of ASEP) we can define the joint evolution $(\eta_t,\alpha_t)$ of first and second class particles started from $\eta_0=\eta$ and $\alpha_0=\alpha$. We will not record the generator, though note that it is given in \cite[(4.4)]{MSL}.
Let $\PP^{\eta,\alpha}$ denote the probability measure for this Markov process from initial data $(\eta,\alpha)$.

For $x\in \Z$, $\eta\in \{0,1\}^{\Z}$ and $N\in \Z_{\geq 1}$, let $A^{\geq}(x,\eta,N)$ equal the set of $\alpha\in \{0,1\}^\Z$ such that $\sum_{j\in \Z} \alpha_j=N$, $\eta+\alpha\in \{0,1\}^Z$, $\alpha(x)=1$ and $\alpha(w)= 1$ only if $w\geq x$ (this is not an if and only if). In words, this means that we start with $N$ second class particles relative to the first class particles at $\eta$, with the left-most one at $x$. Associate to such an $\alpha$, a set $\{z_t(1),\ldots, z_t(N)\}$ of locations for the second class particles.
%
%

\begin{prop}\label{prop:Rezappendix}
For any $x,y\in \Z$, $x_0\in \Z$ and $\eta\in \{0,1\}^{\Z}$ with $\eta_0(x_0)=0$, and for any $N\in \Z_{\geq 1}$ and $\alpha_0\in A^{\geq}(x_0,\eta_0,N)$,
\begin{equation}\label{eq:Rez}
\PP^{\eta_0,x_0}(x_t\geq y) \leq \frac{1}{N} \sum_{j=1}^{N} \PP^{\eta,\alpha}(z_t(j)\geq y).
\end{equation}
\end{prop}
\begin{proof}
To prove this we will introduce a process $\mathbf{z}_t=\big(z_t(1),\ldots, z_t(N)\big)$ which is comprised of the locations in $\alpha_t$, but for which the order of the labels can change. We will then show that for any particular label $j\in \{1,\ldots, N\}$, there exists a coupling of $(\eta_t,\mathbf{z}_t,x_t)$ such that $x_t\leq z_t(j)$ for all $t$. Finally, since the uniform distribution on orders for $\mathbf{z}_t$ is preserved, marginally, for all $t$, we will be able to conclude that \eqref{eq:Rez} holds.
We need a bit of notation. Let $\alpha(u;\mathbf{z}):=\alpha(u)= \mathbf{1}_{u = z_j\textrm{ for some }j\in \{1,\ldots, N\}}$ denote the indicator that there is a second class particle at position $u$, given location vector $\mathbf{z}$; let $\zeta(u;\mathbf{z}) = \eta(u) + \alpha(u;\mathbf{z})$ similarly denote the indicator for either a first or second class particle at position $u$.

We now define a coupling of $\eta_t$ and the second class particle label process $\mathbf{z}_t$. Forgetting about the labels, this reduces to the usual (basic) coupling of first and second class particles. It should be noted that the labels in $\mathbf{z}_t$ do not stay ordered. The state space for $(\eta_t,\mathbf{z}_t)$ is evident, and the generator is given by its actions on local functions $f$ as
\begin{align*}
\mathcal{B}f (\eta,\mathbf{z}) :=& (p-q) \sum_{u\in \Z} \eta(u) (1-\zeta(u-1,\mathbf{z})) \left[f(\eta^{u,u-1},\mathbf{z})-f(\eta,\mathbf{z})\right]\\
&+ (p-q)\sum_{j=1}^{N} (1-\zeta(z_j-1;\mathbf{z}))\left[f(\eta,\mathbf{z}^{z_j,z_j-1})-f(\eta,\mathbf{z})\right]\\
&+ (p-q) \sum_{j=1}^{N} \eta(z_j+1)\left[f(\eta^{z_j+1,z_j},\mathbf{z}^{z_j,z_j+1})-f(\eta,\mathbf{z})\right]\\
&+ q \sum_{u\in \Z}\left[f(\eta^{u,u+1},\mathbf{z}^{u,u+1})-f(\eta,\mathbf{z})\right].
\end{align*}
Here $\mathbf{z}^{u,v}$ denotes the configuration resulting from exchanging the content of sites $u$ and $v$ in $\mathbf{z}$. In particular, if $z_i=u$ and $z_j=v$, then $\tilde{\mathbf{z}}=\mathbf{z}^{u,v}$ has $\tilde{z}_i=v$ and $\tilde{z}_j=u$ (and all other values unchanged). It is worth noting that the fourth line in $\mathcal{B}$ results in such swapping of labels between second class particles.

Now, we will demonstrate that for any given $j\in \{1,\ldots, N\}$ it is possible to construct a coupling $(\eta_t,\mathbf{z}_t,x_t)$ such that if  $x_0\leq z_0(j)$, then $x_t\leq z_t(j)$, for all $t$. We have already defined couplings of $(\eta_t,x_t)$ and $(\eta_t,\mathbf{z}_t)$.
So, to see that our desired triple coupling exists, we simply need to show that if, for some $t$, $x_t = z_t(j)=u$, then the jump rates for $x_t$ and $z_t(j)$ are ordered so $x_t$ jumps left with higher rate than $z_t(j)$ and jumps right with lower rate than $z_t(j)$. This is shown by inspection of the generators: The left jump rate for $x_t$ is $(p-q)(1-\eta_t(u-1))+q$ while  for $z_t(j)$ it is $(p-q)\zeta_t(u-1) +q$. Since $\eta_t(u-1)=0$ implies $\zeta_t(u-1)=0$, the rates a ordered as desired; the right jump rate for $x_t$ and $z_t(j)$ are both $(p-q)\eta_t(u+1)+q$. This proves that the desired coupling exists. Let us denote the coupled probability measure started in state $\eta,\mathbf{z},x$ by $\PP^{\eta,\mathbf{z},x}$.

Now we can conclude with the proof of \eqref{eq:Rez}. For a given $\alpha$, let $\mathbf{z}_0$ denote a uniformly random ordering on the elements of $\alpha$. The dynamics on $\mathbf{z}_t$ preserve the uniform ordering in the sense that for any fixed $t$, the marginal distribution of the order of labels in $\mathbf{z}_t$ remains uniform. For any $\alpha\in A^{\geq}(x,\eta,N)$, fix $j=1$ and use the coupling from the previous paragraph to see that
\begin{equation}
\PP^{\eta,x}(x_t\geq y) = \PP^{\eta,\mathbf{z}, x}(x_t\geq y) \leq \PP^{\eta,\mathbf{z}, x}(z_t(1)\geq y) =\PP^{\eta,\mathbf{z}}(z_t(1)\geq y) = \frac{1}{N} \sum_{j=1}^{N} \PP^{\eta,\alpha}\big(z_t(j)\geq y\big).
\end{equation}
The first equality is immediate from the coupling since including the $\mathbf{z}_t$ process has no baring on the event $x_t\geq y$. The second inequality is because under the $j=1$ version of the coupling of $(\eta_t,\mathbf{z}_t,x_t)$, we have that $z_t(1)\geq x_t$ provided $z_0(1)\geq x_0$. That inequality, however, is implied by the assumption that  $\alpha\in A^{\geq}(x,\eta,N)$ (which means that $x_0=x$ is the left-most particle in $\alpha$ and hence in $\mathbf{z}$). The third equality is again immediate from the coupling since now $x_t$ has no baring on the event $z_t(1)\geq y$. The final equality is because we have assumed that the order in $\mathbf{z}$ is uniformly chosen and this property holds for all $t$. Thus, we must average over the events $z_t(j)\geq y$ as stated.
\end{proof}

\begin{proof}[Proof of \Cref{prop:Rez}]
In order to match \Cref{prop:Rezappendix} with  \Cref{prop:Rez} we take $R=p$ and $L=q$ and then reverse space.
\end{proof}

\section{Proof of moderate deviation results and proof of \Cref{hetaxi}}\label{sec:modDevproof}

To prove  \Cref{hetaxi}, we will use the following proposition that provides upper and lower tail bounds on $\h_T (X; \boldsymbol{\eta}_T)$. These bounds are summarized in Figure \ref{fig:Tailbounds} and its caption.

\begin{prop}\label{prop:combined}
For any $\varepsilon > 0$, there exists $c = c(\varepsilon) > 0$ such that the following holds. Let $\rho \in [\varepsilon, 1]$ and $\boldsymbol{\eta}$ be ASEP under $(\rho; 0)$-Bernoulli initial data. For any $T > 1$, $s\geq 0$ letting  $Y_0:=(1 - 2 \rho) T + T^{2/3} $ we have (recall from \eqref{eq:heightdiff} that $\h_T([X,Y];\bfeta):=\h_T(X;\bfeta)-\h_T(Y;\bfeta)$)
\begin{flalign}
&\PP \bigg[  \h_T (X; \bfeta) \ge \displaystyle\frac{(T-X)^2}{4T} + s T^{\frac{1}{3}} \bigg] \le c^{-1} e^{-cs} \quad \textrm{for }X \in \big\llbracket\! -(1 - \varepsilon)T, (1 - \varepsilon) T \big\rrbracket,\label{0hetaxi2}\\
&\PP \bigg[ \h_T(X; \bfeta) \le \displaystyle\frac{(T-X)^2}{4T} - s T^{\frac{1}{3}} \bigg] \le  c^{-1} (e^{-cs} + e^{-cT}),  \quad \textrm{for }X\in  \big\llbracket Y_0, (1 - \varepsilon) T \big\rrbracket,\label{h1}\\
&\PP \big[ \h_T([X,Y_0]; \bfeta)  \ge \rho (Y_0 - X) + s T^{\frac{1}{2}} \big] \le 2 e^{-\frac{s^2}{4}},  \quad \textrm{for }X\in \big\llbracket\! -T, Y_0 \big\rrbracket \label{hxyeta1} \\
&\PP \big[ \h_T([X,Y_0];\bfeta) \le \rho (Y_0 - X) -s T^{\frac{2}{3}} \big] \le c^{-1} T^{\frac{1}{3}} (e^{-cs} + e^{-cT})\quad \textrm{for }X\in \big\llbracket \!-T, Y_0 \big\rrbracket.\label{hxyeta2}
\end{flalign}
The constants $c=c(\varepsilon)$ can be chosen so as to weakly decrease as $\varepsilon$ decreases to 0.
\end{prop}

\begin{figure}
	\begin{center}
	\includegraphics[width=4.5in]{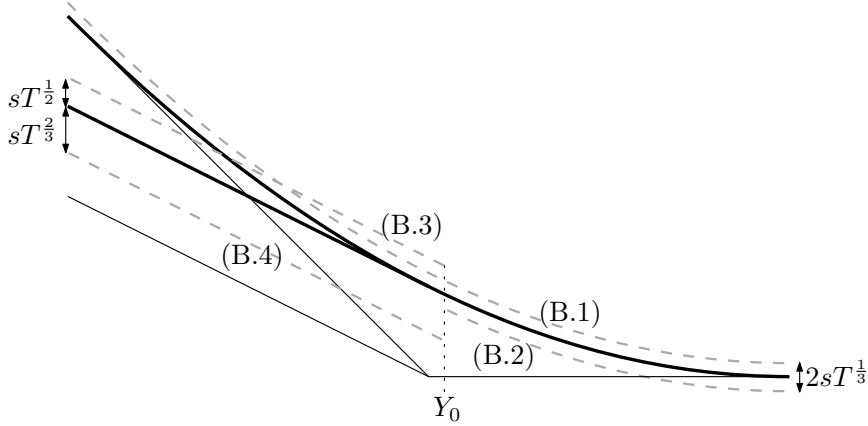}
	\end{center}
	\caption{The four bounds in \Cref{prop:combined}. The thin solid lines represent the initial hydrodynamic profile for $(\rho; 0)$-Bernoulli initial data with $\rho=1/2$ and step initial data, and the thick solid lines present the evolved hydrodynamic profiles. The four dashed lines are labeled with the corresponding bounds in \Cref{prop:combined} and the scale of the error bounds are noted as well.}
	\label{fig:Tailbounds}
\end{figure}

We will first provide a proof of \eqref{hxyeta1} based on a simple coupling argument and concentration bound for sums of i.i.d. Bernoulli random variables. Then we will prove \eqref{hxyeta2}, assuming \eqref{0hetaxi2} and \eqref{h1}. We then prove \eqref{0hetaxi2}, relying upon a remarkable identity from \cite{AEPDPP} that relates ASEP to the discrete Laguerre ensemble. Finally, we prove \eqref{h1} through asymptotics of a Fredholm determinant formula coming from \cite{BCS,PTAEP}.

A few remarks about the proposition are in order. While \eqref{0hetaxi2} gives an upper bound on $\h_T (X; \bfeta)$ for all $X \in \big[- (1-\varepsilon) T, (1 - \varepsilon) T \big]$, it is only useful for us (though the decay is likely not as sharp as possible) for $X\in \big((1 - 2 \rho) T,(1 - \varepsilon) T\big)$. This is because the hydrodynamic limit center changes for velocities less that $1-2\rho$. Equation \eqref{h1} gives an effective lower bound on $\h_T (X; \bfeta)$ for $X\in \big[Y_0,(1 - \varepsilon) T\big]$. The reason for the $T^{2/3}$ offset in how we define $Y_0$ comes from the proof of this bound where it simplifies the analysis and choice of contours in the Fredholm determinant used there.  The Gaussian $T^{1/2}$ order upper bound in  \eqref{hxyeta1} should be close to tight, while the $T^{2/3}$ order lower bound in \eqref{hxyeta2} is not tight. We expect the actual lower bound should involve $T^{1/2}$ and a Gaussian tail as in \eqref{hxyeta1}. Such a bound in place of  \eqref{hxyeta2} may be possible from the Fredholm determinant, though we do not pursue it (see \cite{PTAEP} for an example of such an analysis).

Assuming \Cref{prop:combined}, we can establish \Cref{hetaxi}.

\begin{proof}[Proof of \Cref{hetaxi}]
As in \Cref{prop:combined}, let $Y_0 = (1 - 2 \rho) T + T^{2/3}$.  Observe that by combining \eqref{hxyeta1} with \eqref{hxyeta2} (for the first bound below) and  \eqref{0hetaxi2} with \eqref{h1} (for the second bound below) we have that for any
\begin{flalign*}
\PP \Big[ \big| \h_T([X,Y_0]; \bfeta) - \rho (Y_0 - X) \big| \ge s T^{\frac{2}{3}} \Big]  \le c^{-1} T e^{-cs}, \qquad & \text{if $-(1-\varepsilon) T \le X \le Y_0$}; \\
\PP \Bigg[ \bigg| \h_T(X; \bfeta) - \displaystyle\frac{(T - X)^2}{4T} \bigg| \ge sT^{\frac{2}{3}} \Bigg]\le c^{-1} T e^{-cs}, \qquad & \text{if $Y_0 \le X \le (1 - \varepsilon) T$}.
\end{flalign*}
This holds for any $\varepsilon > 0$ (with $c=c(\varepsilon)$, $\rho \in [\varepsilon, 1]$, $T>1$ and $s\in [0,T]$). We restrict $s\in [0,T]$ (as opposed to $s\geq 0$) in order to bound $c^{-1} T^{\frac{1}{3}} (e^{-cs} + e^{-cT})\leq  T e^{-cs}$ in \eqref{hxyeta2}.
Now observe that from the explicit form of $\Upsilon^{(\rho; 0)}$ given by \Cref{lambdarhofunctions}, we have that
\begin{flalign*}
T \displaystyle\int\limits_{X/T}^{Y/T} \Upsilon^{(\rho; 0)} (z) dz = \begin{cases}
 \displaystyle\tfrac{1}{4T} \big( (T - X)^2 - (T - Y)^2 \big), \qquad & \text{if $(1 - 2\rho) T \le X \le Y \le T$}, \\
(Y-X) \rho, \qquad & \text{if $X \le Y \le (1 - 2 \rho) T$.}\end{cases}			
\end{flalign*}
Now \eqref{hetaxilambda0} readily follows by combining this and the previous display.

The claim about step initial data follows immediately from \eqref{0hetaxi2} with \eqref{h1} after observing that for any fixed $\varepsilon$, there is a $T_0>0$ such that for all $T>T_0$, $Y_0/T <-(1-\varepsilon)$. Thus, by choosing a small enough $c$, we can ensure that \eqref{hetaxilambda0} holds for all $T>1$ with $sT^{2/3}$ replaced by $sT^{1/3}$, as desired.
\end{proof}

\subsection*{ Proof of \Cref{hxyeta1}}
We compare the ASEP $\bfeta$ to a $\rho$-stationary ASEP $\bfxi$. Since we can clearly couple $\bfeta_0$ and $\bfxi_0$ so that $\bfeta_0 (j) \le \bfxi_0 (j)$ for all $j \in \mathbb{Z}$, the attractivity in \Cref{xizeta1} yields a coupling between $\bfeta$ and $\bfxi$ such that $\bfeta_t (j) \le \bfxi_t (j)$ for all $t \ge 0$ and $j \in \mathbb{Z}$. It follows from \eqref{eq:heightdiff} that
\begin{flalign*}
\PP \big[ \h_T([X,Y_0]; \bfeta)  \ge \rho (Y_0 - X) + s T^{\frac{1}{2}} \big]
&=\PP \Bigg[ \displaystyle\sum_{j = X+1}^{Y_0 } \bfeta_T (j) \ge \rho (Y_0 - X) + s T^{\frac{1}{2}} \Bigg]\\
&\le \PP \Bigg[ \displaystyle\sum_{j = X+1}^{Y_0 } \bfxi_T (j) \ge \rho (Y_0 - X) + s T^{\frac{1}{2}} \Bigg] \le 2 e^{-\frac{s^2}{4}},
		\end{flalign*}
where to deduce the last bound we used the concentration estimate \eqref{hx1}, together with the fact that $|X - Y_0| \le 2T$. This establishes \eqref{hxyeta1}.
		
\subsection*{ Proof of \Cref{hxyeta2} assuming \Cref{0hetaxi2} and \Cref{h1}}
We claim that for any $A,X\in \Z$ such that $A\geq 0$ and $X \le Y_0$, and for any $M\in \R$
\begin{equation}
\label{x0y01}
\PP \big[ \h_T([X,X + A]; \bfeta) \le  M \big]\leq  \PP \big[ \h_T([Y_0,Y_0 + A]; \bfeta) \le M \big].
\end{equation}
This follows from a simple coupling argument. Let $\bfzeta$ denote an ASEP whose initial data $\bfzeta_0$ is obtained by setting $\bfzeta_0(j) = \bfeta_0(j)$ for $j\leq X-Y_0$ and $\bfzeta_0(j)=0$ otherwise. Notice that  $\bfzeta_0$ is equal in distribution to $\bfeta_0$, shifted to left by a distance of magnitude $|X-Y_0|$, and that it is also coupled in such a way that $\bfeta_0 (j) \ge \bfzeta_0 (j)$ for all $j \in \Z$. By the attractivity of ASEP from \Cref{xizeta1} we may couple $\bfeta$ and $\bfzeta$ so that $\bfeta_t (j) \ge \bfzeta_t (j)$ for all $t \ge 0$ and $j \in \Z$. It follows then from \eqref{eq:heightdiff}  and the above considerations that
\begin{flalign*}
\PP \big[ \h_T([X,X + A; \bfeta) \le  M\big] &=	\PP \Bigg[ \displaystyle\sum_{j = X+1}^{X + A} \bfeta_T (j) \le  M \Bigg] \le \PP \Bigg[ \displaystyle\sum_{j = X+1}^{X + A} \bfzeta_T (j) \le M \Bigg]\\
&= \PP \Bigg[ \displaystyle\sum_{j = Y_0+1}^{Y_0 + A} \bfeta_T (j) \le  M \Bigg] = \PP \big[ \h_T([Y_0,Y_0 + A; \bfeta) \le M \big],
\end{flalign*}
where the inequality uses the attractive coupling and the penultimate equality holds since $\big( \bfzeta_T (j) \big)_{j\in \Z}$ and $\big( \bfeta_T (j + Y_0-X) \big)_{j\in \Z}$ have the same law.

Observe now that to show \eqref{hxyeta2} we must show that the height difference $ \h_T([X,Y_0];\bfeta)$ compensated by the linear hydrodynamic profile $\rho (Y_0 - X)$ is unlikely to dip more than $-s T^{\frac{2}{3}}$. The inequality in \eqref{x0y01} shows that we can control the height changes to the left of $Y_0$ by those at $Y_0$. The bounds in \eqref{0hetaxi2} and \eqref{h1} are effective in controlling the height changes around $Y_0$. However, they involve parabolic hydrodynamic terms $\frac{(T-X)^2}{4T}$ whereas in \eqref{hxyeta2} we are dealing with linear hydrodynamic terms $\rho(Y_0-X)$. However, on short enough spatial intervals, the parabolic term is approximately linear. In particular, on the spatial scale $T^{\frac{2}{3}}$, the parabolic effect is of order $T^{\frac{1}{3}}$ which is of the order of fluctuations. Thus, to establish the desired control in \eqref{hxyeta2} we use \eqref{x0y01} repeatedly on spatial intervals of order $T^{\frac{2}{3}}$. Each application introduces a fluctuation error of order  $T^{\frac{1}{3}}$. By a union bound over order  $T^{\frac{1}{3}}$ such spatial intervals, we arrive at the order  $T^{\frac{2}{3}}$ fluctuation error bound in \eqref{hxyeta2} (this union bound also explains the $T^{\frac{1}{3}}$  factor on the right-hand side in  \eqref{hxyeta2}). The rest of this proof provides the details to the argument sketched above.

Let us assume that $A \in \llbracket 0, T^{\frac{2}{3}}\rrbracket$ and $X\in \Z$ such that $X \le Y_0$. Letting $M=\rho A-sT^{\frac{1}{3}}$ in \eqref{x0y01} we have that
\begin{equation}
\label{x0y01prime}
\PP \big[ \h_T([X,X + A]; \bfeta) \le  \rho A-sT^{\frac{1}{3}} \big]\leq  \PP \big[ \h_T([Y_0,Y_0 + A]; \bfeta) \le \rho A-sT^{\frac{1}{3}} \big].
\end{equation}
We now claim that there exists $c=c(\varepsilon)>0$ such that for any $T>1$ and $s\geq 0$,
\begin{equation}\label{eq:hyy0a}
\PP \big[ \h_T([Y_0,Y_0 + A]; \bfeta) \le \rho A-sT^{\frac{1}{3}} \big] \leq c^{-1} (e^{-cs} + e^{-cT}).
\end{equation}
To prove this, we note that for $s\geq 1$,
\begin{flalign*}
&\Big\{\h_T([Y_0,Y_0 + A]; \bfeta) \le \rho A-sT^{\frac{1}{3}}\Big\} \subset \\
&\bigg\{\h_T(Y_0; \bfeta) \le \frac{(T-Y_0)^2}{4T}-\frac{sT^{\frac{1}{3}}}{8}\bigg\}\cup
 \bigg\{\h_T(Y_0+A; \bfeta) \ge \frac{(T-Y_0-A)^2}{4T}+\frac{sT^{\frac{1}{3}}}{8}\bigg\}
\end{flalign*}
as follows immediately from the inequality (also for $s\geq 1$) that
\begin{flalign*}
\displaystyle\frac{(T - Y_0)^2}{4T}  - \displaystyle\frac{(T - Y_0 - A)^2}{4T}\ge \rho A - \frac{3sT^{\frac{1}{3}}}{4},
		 \end{flalign*}
Thus, by the union bound along with the bounds in \eqref{0hetaxi2} and \eqref{h1}, we arrive at \eqref{eq:hyy0a} provided $s\geq 1$. For $s\in [0,1]$, the result follows  by choosing $c$ sufficiently close to zero.
	
Let us now apply \eqref{eq:hyy0a} to conclude the desired bound in \eqref{hxyeta2}. Observe that for $X\in \llbracket -T,Y_0\rrbracket$, the interval $[X,Y_0]$ can be covered by at most $K=2T^{1/3}$ intervals (this is an overestimate but suffices) each of length $A \in \llbracket 0, T^{\frac{2}{3}}\rrbracket$. Call the endpoints of these intervals $X=X_0<X_1<\cdots <X_K=Y_0$. Then we have
\begin{flalign*}
\PP \big[ \h_T([X,Y_0];\bfeta) \le \rho (Y_0 - X) -s T^{\frac{2}{3}}\big]
&\leq \sum_{j=1}^{K}  \PP \bigg[ \h_T([X_{j-1},X_j]; \bfeta)\leq \rho A-\frac{sT^{\frac{1}{3}}}{2}\bigg]\\
&\leq K  c^{-1} (e^{-cs/2} + e^{-cT}).
\end{flalign*}
The first inequality follows from the union bound while the second from combining \eqref{x0y01prime} (with $X=X_{j-1}$ and $X+A=X_{j}$) with \eqref{eq:hyy0a}. Clearly, this implies \eqref{hxyeta2} as desired.
	
\subsection{Proof of \Cref{0hetaxi2} }
\label{Estimate1s}

The main result needed in this proof of \Cref{0hetaxi0} (about step initial data ASEP) from which \eqref{0hetaxi2} follows via the montonicity result \Cref{xizeta2}.

\begin{prop}
\label{0hetaxi0}
For any $\varepsilon > 0$, there exists $c = c(\varepsilon) > 0$ such that the following holds. Let $\bfxi$ be ASEP under step initial data. Then, for any $T > 1$, $s\geq 0$  and $X \in \big\llbracket\! -(1 - \varepsilon)T, (1 - \varepsilon) T \big\rrbracket$,
\begin{flalign}
\label{hetaxilambda00}
\PP \Bigg[ \bigg| \h_T(X; \bfxi) - \displaystyle\frac{(T-X)^2}{4T} \bigg| \ge s T^{1/3} \Bigg] \le c^{-1} e^{-cs}.
\end{flalign}
\end{prop}
	
\begin{proof}[Proof of \Cref{0hetaxi2}]
		Letting $\boldsymbol{\xi} = \big( \xi_t (x) \big)$ denote an ASEP under step initial data, we have $\bfeta_0 (x) \le \bfxi_0 (x)$ for each $x \in \mathbb{Z}$. Thus, $\h_0 (x; \boldsymbol{\eta}) \le \h_0 (x; \boldsymbol{\xi})$, and so the monotonicity result \Cref{xizeta2} yields a coupling between $(\boldsymbol{\eta}; \boldsymbol{\xi})$ such that $\h_T(x; \bfeta) \le \h_T(x; \bfxi)$ for each $T \ge 0$ and $x \in \mathbb{Z}$. Hence the proposition follows from the fact that
		\begin{flalign*}
			\PP \Bigg[  \h_T(X; \bfeta) - \displaystyle\frac{(T-X)^2}{4T}  \ge s T^{1/3} \Bigg] \le \PP \Bigg[  \h_T(X; \bfxi) - \displaystyle\frac{(T-X)^2}{4T} \ge s T^{1/3} \Bigg] \le c^{-1} e^{-cs}.
		\end{flalign*}
where in the last inequality we applied \Cref{0hetaxi0}.
	\end{proof}

The proof of \Cref{0hetaxi0} relies on an identity \cite[Theorem 10.2]{AEPDPP} (cited below as \Cref{qexpectation12}) which relates a $q$-Laplace transform for the step initial data ASEP height function to a multiplicative statistic for the determinantal point process called the discrete Laguerre ensemble. From this identity and existing asymptotics regarding this ensemble, we are able to prove our tail bound. We remark that our tail bound is suboptimal and we do not fully take advantage of the decay afforded to us by the identity. However, the exponential decay we prove is sufficient for our purposes. We also note that this style of result -- using a $q$-Laplace transform identity with a determinantal point process in order to prove tail bounds -- goes back to work of \cite{CG} which uses a similar identity relating the KPZ equation and Airy point process \cite{BG}.

To prepare for the proof of \Cref{0hetaxi0}, we recall a few  results. The first is the fact that the statement in \Cref{0hetaxi0} holds in the case of TASEP, when $L=0$. This result is implicitly due to \cite{CTPS} (in terms of an estimate on the Fredholm determinant that determines this tail probability), though appears explicitly as a probabilistic tail estimate (formulated in terms of exponential last passage percolation) as \cite[Theorem 13.2]{LPDL} and \cite[Proposition 4.1 and Proposition 4.2]{ASFTLPM}.
\begin{lem}
\label{l0estimate}
For any $\varepsilon > 0$, there exists $c = c(\varepsilon) > 0$ such that the following holds. Let $\bfxi$ be TASEP ($L=0$ temporarily) under step initial data. Then, for any $T > 1$, $s\geq 0$  and $X \in \big\llbracket\! -(1 - \varepsilon)T, (1 - \varepsilon) T \big\rrbracket$,
\begin{flalign*}
\PP \Bigg[ \bigg| \h_T(X; \bfxi) - \displaystyle\frac{(T-X)^2}{4T} \bigg| \ge s T^{1/3} \Bigg] \le c^{-1} e^{-cs}.
\end{flalign*}
\end{lem}
	
To establish \Cref{0hetaxi0}, we will make use of a determinantal point process, introduced in \cite{AEPDPP}, called the discrete Laguerre ensemble. We begin by recalling its definition. In what follows, a \emph{configuration} on $\mathbb{Z}_{\ge 0}$ is a subset $\mathfrak{Z} \subseteq \mathbb{Z}_{\ge 0}$ of nonnegative integers; we let $\Conf (\mathbb{Z}_{\ge 0})$ denote the set of all configurations on $\mathbb{Z}_{\ge 0}$. Given a function $K : \mathbb{Z}_{\ge 0} \times \mathbb{Z}_{\ge 0} \rightarrow \mathbb{R}$, a \emph{determinantal point process} on $\Conf (\mathbb{Z}_{\ge 0})$ with \emph{correlation kernel $K$} is a probability measure $\PP$ on $\Conf (\mathbb{Z}_{\ge 0})$ satisfying the following property. Letting $\mathfrak{Z} \in \Conf (\mathbb{Z}_{\ge 0})$ denote a random configuration sampled under $\PP$, we have
$
\PP \big[ x_1, x_2, \ldots , x_k \in \mathfrak{Z} \big] = \det [ K (x_i, x_j)]_{1 \le i, j \le k}
$
for any distinct $x_1, x_2, \ldots , x_k \in \mathbb{Z}_{\ge 0}$. 	Generic $K$ will not define a probability measure.
			
\begin{definition}

\label{definitionl}
Fix $\beta\in \R_{>0}$. \emph{Laguerre polynomials} are the orthogonal polynomials on $[0, \infty)$ under the weight measure $t^{\beta-1}e^{-t}dt$. The degree $n$ polynomial in this ensemble, with leading coefficient $n!^{-1}$, is denoted by$L_n^{(\beta)} (x)$. The \emph{discrete Laguerre kernel} is defined by
\begin{align*}
K_{\mathrm{DLaguerre} (r;\beta)}(x,y) = \bigg(\frac{x! y!}{\Gamma(x+\beta)\Gamma(y+\beta)}\bigg)^{1/2} \int^{\infty}_{r} L^{(\beta)}_{x}(t)L^{(\beta)}_{y}(t)t^{\beta-1}e^{-t}dt,
\end{align*}
for any $x, y \in \mathbb{Z}_{\ge 0}$. The \emph{discrete Laguerre ensemble} $\text{DLaguerre} (r; \beta)$ is the determinantal point process on $\Conf (\mathbb{Z}_{\ge 0})$ with correlation kernel $K_{\mathrm{DLaguerre} (r; \beta)}$.

\end{definition}

The following lemmas indicate our use of the discrete Laguerre ensemble. The first shows that its smallest element match the TASEP height function in distribution (which is accessible by \Cref{l0estimate}); the second explains its relation to the step initial data ASEP.

\begin{lem}[{\cite[Corollary 10.3 and Theorem 3.7]{AEPDPP}}]

\label{processparticle}

Adopt the notation of \Cref{0hetaxi0}, and assume that $L = 0$. Let $\mathfrak{Z} \in \Conf (\mathbb{Z}_{\ge 0})$ denote a sample of the discrete Laguerre ensemble $\mathrm{DLaguerre}^+ (T, x + 1)$. Then, $\h_T(x; \bfxi)$ has the same law as $\min \mathfrak{Z}$.

\end{lem}
	
	Below, for any $q, a \in \mathbb{C}$ and any integer $k \ge 0$ (possibly infinite, in which case we assume that $|q| < 1$), the \emph{$q$-Pochhammer symbol} is defined by $(a; q)_k = \prod_{j = 0}^k (1 - a q^j)$.

\begin{lem}[{\cite[Theorem 10.2]{AEPDPP}}]
\label{qexpectation12}
Fix any time $T > 0$, an integer spatial location $x \ge 0$ and let $q = \frac{L}{R} \in (0, 1)$. Let $\bfxi$ denote ASEP, with left jump rate $L$ and right jump rate $R$, under step initial data and let $\mathfrak{Z} \subset \mathbb{Z}_{\ge 0}$ denote a sample of $\mathrm{DLaguerre}^+ \big( (1 - q) T, x + 1 \big)$. Then, for any $\zeta \in \mathbb{C}\setminus \{-q^{\Z_{\leq 0}}\}$ we have
\begin{flalign}
	\label{zh}
	\EE \Bigg[ \displaystyle\frac{1}{\big(- \zeta q^{\h_T(x; \bfxi)}; q \big)_{\infty}} \Bigg] = \EE\Bigg[ \displaystyle\prod_{z \in \mathfrak{Z}} \displaystyle\frac{1}{1 + \zeta q^z} \Bigg],
\end{flalign}
where the expectation on the left side is with respect to the ASEP $\boldsymbol{\xi}$, and the expectation on the right side is with respect to the discrete Laguerre configuration $\mathfrak{Z}$.
\end{lem}

In order to make use of this lemma, we need to be able to translate between the above $q$-Laplace transform type expectations and statements about probabilities.

\begin{lem}
\label{aqestimate}
Let $\mathbf{A}$ be a real-valued random variable, $q \in [0, 1)$ and $b\in \R$. Then,
\begin{flalign}
&\PP [\mathbf{A} \le 0] \le 2 \cdot \Big( 1 - \EE \big[ (-q^{\mathbf{A}}; q)_{\infty}^{-1} \big] \Big),\label{za1}\\
&\EE \big[ (-q^{\mathbf{A}}; q)_{\infty}^{-1} \big] \ge e^{q^b / (q-1)} \cdot \PP [\mathbf{A} \ge b], \label{za2}\\
&\EE\big[ (1 + q^{\mathbf{A}})^{-1} \big] \le \PP [\mathbf{A} > -b] + q^b \cdot \PP [\mathbf{A} \le -b].\label{za3}
\end{flalign}

\end{lem}

\begin{proof}
Observe that for any $a \in \R$ we have
\begin{flalign*}
& (-q^a; q)_{\infty}^{-1} \le (1 + q^a)^{-1} \le 1 - \displaystyle\frac{\textbf{1}_{a \le 0}}{2},\\
&(-q^a; q)_{\infty}^{-1} \ge (-q^b; q)_{\infty}^{-1} \cdot \textbf{1}_{a \ge b} \ge e^{q^b / (q-1)} \cdot \textbf{1}_{a \ge b},\\
& (1 + q^a)^{-1} \le \textbf{1}_{a > -b} + q^b \cdot \textbf{1}_{a \le -b}.
\end{flalign*}
Setting $a=\mathbf{A}$ and taking expectations yields the lemma.
\end{proof}

Now we can establish \Cref{0hetaxi0}.

\begin{proof}[Proof of \Cref{0hetaxi0}]
By the particle-hole symmetry \Cref{etaeta}, it suffices to address the case $X \ge 0$. Let $m = \frac{1}{4} \big( \frac{T-X}{T}\big)^2$, and set $\zeta = q^{-mT - sT^{1/3}}$. As in the statement of \Cref{qexpectation12}, let $\mathfrak{Z} \in \Conf (\mathbb{Z}_{\ge 0})$ denote a sample of the discrete Laguerre ensemble $\mathrm{DLaguerre} \big( (1 - q) T; X + 1 \big)$. Then, by \Cref{qexpectation12} we have that
\begin{flalign}
	\label{zxi}
	\EE\Bigg[  \big(- q^{\h_T(x; \bfxi) - mT - sT^{1/3}} ; q \big)_{\infty}^{-1} \Bigg] = \EE\Bigg[ \displaystyle\prod_{z \in \mathfrak{Z}} \big( 1 +  q^{z - mT - sT^{1/3}} \big)^{-1} \Bigg].
\end{flalign}
Let $\mathfrak{Z}_0$ denote the minimal element in $\mathfrak{Z}$. We claim that there exists $c= c(\varepsilon) > 0$ such that
\begin{flalign*}
	\EE \Bigg[ \displaystyle\prod_{z \in \mathfrak{Z}} \big( 1 + q^{z - mT - sT^{1/3}} \big)^{-1} \Bigg] & \le \EE\Big[ \big(1 + q^{\mathfrak{Z}_0 - mT - sT^{1/3}} \big)^{-1} \Big] \\
	& \le \PP \bigg[ \mathfrak{Z}_0 \ge mT + \displaystyle\frac{sT^{1/3}}{2} \bigg] + q^{sT^{1/3} / 2} \le c^{-1} e^{-c s}.
\end{flalign*}
The first inequality is immediate (dropping the other terms in the product only increases the value), the second utilizes \eqref{za3} (with
$\mathbf{A} =  \mathfrak{Z}_0 - mT - s T^{1/3}$ and $b=\frac{s}{2} T^{1/3}$), and the third follows by combining \Cref{processparticle} with \Cref{l0estimate}.

Combining the above inequality with \eqref{zxi} and \eqref{za2} (with $\mathbf{A} = \h_T(x; \bfxi) - mT - sT^{1/3}$ and $b=0$), we find that
\begin{flalign*}
	\PP \big[ \h_T(x; \bfxi) \ge mT + sT^{1/3} \big] & \le e^{1/(1-q)} \cdot \EE \Bigg[ \big( -q^{ \mathfrak{Z}_0 - mT - sT^{1/3}}; q \big)_{\infty}^{-1} \Bigg] \\
	& \le e^{1/(1-q)} \cdot \EE \Bigg[ \displaystyle\prod_{z \in \mathfrak{Z}} \big( 1 + q^{z - mT - sT^{1/3}} \big)^{-1} \Bigg] \le e^{1/(1-q)} c^{-1} e^{-c s}.
\end{flalign*}
Modifying the value of $c$, this yields the upper bound on $\h_T(x; \bfxi)$ in \Cref{0hetaxi0}.

It remains to prove the lower bound.
Observe that there exists $c = c(\varepsilon) > 0$ such that
\begin{flalign*}
\EE\Bigg[ \displaystyle\prod_{z \in \mathfrak{Z}} \big( 1 + q^{z - mT + sT^{1/3}} \big)^{-1} \Bigg] & \ge \EE \Big[ \big(- q^{ \mathfrak{Z}_0 - mT + sT^{1/3}}; q \big)_{\infty}^{-1} \Big] \\
	& \ge \exp \bigg( \displaystyle\frac{q^{sT^{1/3} / 2}}{q-1} \bigg) \cdot \PP \bigg[  \mathfrak{Z}_0 \ge mT - \displaystyle\frac{sT^{1/3}}{2} \bigg] \\
	& \ge \big( 1 - C_2 e^{-c_2 s T^{1/3}} \big) \cdot  (1 - C_2 e^{-c_2 s}) \ge 1 - 2 c^{-1} e^{-c s}.
\end{flalign*}
The first inequality is immediate (inserting all missing terms in the product above the minimal term only decreases its value), the second utilizes \eqref{za2} (with $\mathbf{A} =  \mathfrak{Z}_0 - mT + s T^{1/3}$ and $b=\frac{s}{2} T^{1/3}$), and the third follows by combining  \Cref{processparticle} and \Cref{l0estimate}.

Combining the above inequality with \eqref{zxi} (with $-sT^{1/3}$ replaced there by $sT^{1/3}$) and \eqref{za1} (with $\mathbf{A} = \h_T(x; \bfxi) - mT + sT^{1/3}$), we find that
 \begin{flalign*}
 	\PP \big[ \h_T(x; \bfxi) \le mT - sT^{1/3} \big] & \le 2 \cdot \bigg( 1 - \EE \Big[ \big( q^{\h_T(x; \bfxi) - mT + sT^{1/3}}; q \big)_{\infty}^{-1} \Big]\bigg) \\
 	& = 2 \cdot \Bigg( 1 - \EE \bigg[ \displaystyle\prod_{z \in \mathfrak{Z}} \big( 1 + q^{z - mT + sT^{1/3}} \big)^{-1} \bigg] \Bigg) \le 4 c^{-1} e^{-cs}.
  \end{flalign*}
Modifying the value of $c$, this yields the lower bound on $\h_T(x; \bfxi)$ in \Cref{0hetaxi0}.
\end{proof}

	\subsection{Proof of \Cref{h1}}
	
	\label{Estimate2s}
	
	In this section we establish \eqref{h1}, which is based a Fredholm determinant identity in the ASEP under $(\rho; 0)$-Bernoulli initial data due to \cite[Theorem 5.3]{BCS}. This formula also appears in \cite[Proposition 5.1]{PTAEP} where it is extended to a more general class of initial data. We will utilize a number of estimates used in \cite{PTAEP} to perform asymptotics on this formula (\cite[Apendix D]{BCS} sketch some asymptotics from their formula, though without going into details). Let us first recall the definition of a Fredholm determinant series.

\begin{definition}
\label{definitiondeterminant}
Fix a contour $\mathcal{C} \subset \mathbb{C}$ in the complex plane, and let $K: \mathcal{C} \times \mathcal{C} \rightarrow \mathbb{C}$ be a meromorphic function with no poles on $\mathcal{C} \times \mathcal{C}$. We define the \emph{Fredholm determinant}
\begin{flalign}
\label{determinantsum}
\det \big( \Id + K \big)_{L^2 (\mathcal{C})} = 1 + \displaystyle\sum_{k = 1}^{\infty} \displaystyle\frac{1}{ (2 \pi \mathrm{i} )^k k!} \displaystyle\int_{\mathcal{C}} \cdots \displaystyle\int_{\mathcal{C}} \det \big[ K(x_i, x_j) \big]_{i, j = 1}^k \displaystyle\prod_{j = 1}^k d x_j.
\end{flalign}
\end{definition}

	We next require the following identity for the $q$-Laplace transform (essentially the left side of \eqref{zh}) of ASEP with $(\rho; 0)$-Bernoulli initial data. In what follows, we recall that a contour $\gamma \subset \mathbb{C}$ is called \emph{star-shaped} (with respect to the origin) if, for each real number $a \in [-\pi, \pi]$, there exists exactly one complex number $z_a \in \gamma$ such that $z_a/|z_a| = e^{\mathrm{i} a}$.
	
	\begin{prop}[{\cite[Theorem 5.3]{BCS},\cite[Proposition 5.1]{PTAEP}}]
		
		\label{asymptoticheightvertex}
		
		Fix $\rho \in (0, 1)$, $x \in \mathbb{Z}$, and $p \in \mathbb{R}$. Denote $q = \frac{L}{R} \in (0, 1)$, and set $\beta = \frac{\rho}{1 - \rho}$. Let $\Gamma \subset \mathbb{C}$ be a positively oriented, star-shaped contour  enclosing $0$, but leaving outside $-q$ and $q \beta$. Further let $\mathcal{C} \subset \mathbb{C}$ be a positively oriented, star-shaped contour contained inside $q^{-1} \Gamma$, that encloses $0$, $-q$, and $\Gamma$, but that leaves outside $q \beta$. For ASEP $\bfeta$ with $(\rho; 0)$-Bernoulli initial data, we have
		\begin{flalign}
		\label{expectationk}
			\EE\Big[ \big( -q^{\h_T(X; \bfeta) + p}; q \big)_{\infty}^{-1} \Big] = \det \big( \Id + K^{(p)} \big)_{L^2 (\mathcal{C})},
		\end{flalign}
where
		\begin{flalign}
			\label{kvw1}
			K^{(p)} (w, w') & = \displaystyle\frac{1}{2 \mathrm{i} \log q} \displaystyle\sum_{j = -\infty}^{\infty} \displaystyle\oint_{\Gamma} \displaystyle\frac{g (w; X, T)}{g (v; X, T)} \displaystyle\frac{v^{p - 1} w^{-p}}{\sin \big( \frac{\pi}{\log q} ( \log v - \log w + 2 \pi  \mathrm{i} j ) \big)} \displaystyle\frac{dv}{ w' - v },
		\end{flalign}
and
		\begin{flalign*}
		g (z; X, T) = (z + q)^{X - 1} \exp \left( \displaystyle\frac{q T}{z + q} \right) \displaystyle\frac{1}{(q^{-1} \beta^{-1} z; q)_{\infty}}. 		
		\end{flalign*}
	
	\end{prop}
	
The following result captures the decay of the right side of \eqref{expectationk} as $p$ grows with a suitable centering and scaling. Its proof closely follows \cite[Section 6]{PTAEP} and is provided in \Cref{RightKernel} below.
It is here that our choice that  $X\geq  Y_0= (1-2\rho)T+T^{2/3}$ is used. By making this assumption, the choices for the contours $\mathcal{C}$ and $\Gamma$ are simplified. It is possible, as was done \cite[Section 8]{PTAEP}, to address the case where $X<Y_0$, though it involves more complicated contours and since we have other ways to control that case (namely, \eqref{hxyeta1} and \eqref{hxyeta2}) we forgo that additional complexity.

\begin{prop}
\label{estimatedeterminant}
For any $\varepsilon \in ( 0,1/4 )$, there exists $c = c(\varepsilon) > 0$ such that the following holds. Adopt the notation of \Cref{asymptoticheightvertex}, and set $X = \nu T + 1$. Assume that
	\begin{flalign*}
	\rho \in [\varepsilon, 1], \qquad \nu \in [-(1-\varepsilon), 1 - \varepsilon], \qquad p = p(T, \nu, s) = s f_{\nu} T^{1/3} - m_{\nu} T, \qquad s \ge 0,
	\end{flalign*}
where
	\begin{flalign}
	\label{metafeta}
		m_{\nu} = \bigg( \displaystyle\frac{1 - \nu}{2} \bigg)^2, \qquad f_{\nu} = \bigg( \displaystyle\frac{1 - \nu^2}{4} \bigg)^{2/3}.
	\end{flalign}
If $\nu \ge 1 - 2 \rho + T^{-1/3}$, then
		\begin{flalign}
		\label{eta1estimate}
		\Big| \det \big( \Id + K^{(p)} \big)_{L^2 (\mathcal{C})} - 1 \Big| \le c^{-1} (e^{-cs} + e^{-cT}).
		\end{flalign}
	\end{prop}
	
	Given \Cref{estimatedeterminant}, we can quickly establish \eqref{h1}.

	\begin{proof}[Proof of \Cref{h1}]
	 By \eqref{expectationk} and \eqref{eta1estimate} (and replacing $s$ with $f_{\nu}^{-1} s$ in the latter), we find that there exists $c= c (\varepsilon) > 0$ such that
	\begin{flalign}
	\label{qh}
		\EE \Big[ \big( -q^{\h_T(X; \bfeta) - m_{\nu} T + s T^{1/3}}; q \big)_{\infty}^{-1} \Big] \ge 1 -c^{-1} (e^{-c s} + e^{-c T}).
	\end{flalign}
This, together with \eqref{za1} (applied with $\mathbf{A} = \h_T(X; \bfeta) - m_{\nu} T + sT^{1/3}$), yields \eqref{h1}.
	\end{proof}

	\section{Fredholm determinant estimates}
	
	\label{RightKernel}
	
	In this section we establish \Cref{estimatedeterminant}; we assume throughout that $\nu > 1 - 2 \rho + T^{-1/3}$ and that $\rho \in [-(1-\varepsilon), 1 - \varepsilon]$. We closely follow \cite[Section 6]{PTAEP}, which asymptotically analyzed the Fredholm determinant $\det \big( \Id + K^{(p)} \big)_{L^2 (\mathcal{C})}$ but did not control its decay in $s$. In Section \ref{ckgammakcontours} we recall from \cite[Section 6.1]{PTAEP} a useful choice of contours $\mathcal{C}$ and $\Gamma$, and we then prove \Cref{estimatedeterminant} in Section \ref{VertexRight}.

	\subsection{Choosing the Contours $\mathcal{C}$ and $\Gamma$}
	
	\label{ckgammakcontours}
	
	In this section we recall from \cite[Section 6.2.2]{PTAEP} a choice of contours $\mathcal{C}$ and $\Gamma$ useful for analyzing $\det \big( \Id + K^{(p)} \big)_{L^2 (\mathcal{C})}$. To that end, we first rewrite the kernel (dropping the $p$ in our notation below) $K (w, w') = K^{(p)} (w, w')$ from \eqref{kvw1} as
	\begin{flalign}
		\begin{aligned}
			\label{vptright}
			K (w, w') = \displaystyle\frac{1}{2 \mathrm{i} \log q} \displaystyle\sum_{j \in \mathbb{Z}} & \displaystyle\oint_{\Gamma} \displaystyle\frac{ \exp \left( T \big( G (w) - G (v) \big) \right) }{\sin \big( \pi (\log q)^{-1} (2 \pi \mathrm{i} j + \log v - \log w) \big) } \displaystyle\frac{(q^{-1} \beta^{-1} v; q)_{\infty}}{(q^{-1} \beta^{-1} w; q)_{\infty}} \\
			& \qquad \times \left( \displaystyle\frac{v}{w} \right)^{s f_{\nu} T^{1 / 3}} \displaystyle\frac{dv}{v (w' - v)} ,
		\end{aligned}
	\end{flalign}
where $G (z)$ is given by
	\begin{flalign*}
		G (z) = \displaystyle\frac{q}{z + q} + \nu \log (z + q) + m_{\nu} \log z,
	\end{flalign*}
and $m_{\nu}$ and $f_{\nu}$ are given by \eqref{metafeta}. Observe that
	\begin{flalign*}
			\qquad \quad G' (z) = \left( \displaystyle\frac{\nu + 1}{2} \right)^2 \displaystyle\frac{(z - \psi)^2 }{z (z + q)^2},  \qquad \text{with} \qquad \psi =  \displaystyle\frac{q (1 - \nu)}{1 + \nu}.
	\end{flalign*}
Thus, $\psi$ is a critical point of $G$, and
	\begin{flalign*}
		G'' (\psi) = 0; \qquad \displaystyle\frac{G''' (\psi)}{2} = \displaystyle\frac{(\nu + 1)^5}{16 q^3 (1 - \nu)} = \left( \displaystyle\frac{f_{\nu}}{\psi } \right)^3.
	\end{flalign*}
From a Taylor expansion, this implies that	
	\begin{flalign}
		\begin{aligned}
			\label{gzpsi}
			G (z) - G (\psi) & = \frac{1}{3} \left( \displaystyle\frac{f_{\nu} (z - \psi)}{\psi} \right)^3 + R \left( \displaystyle\frac{f_{\nu;} (z - \psi)}{\psi} \right)
		\end{aligned}
	\end{flalign}
where, uniformly in $\nu \in [\varepsilon - 1, 1 - \varepsilon]$, as $|z - \psi| \rightarrow 0$,
\begin{flalign}
		\label{r13}
		R \left( \displaystyle\frac{f_{\nu} (z - \psi)}{\psi} \right) = G (z) - G (\psi) - \displaystyle\frac{1}{3} \left( \displaystyle\frac{f_{\nu} (z - \psi)}{\psi} \right)^3 =\mathcal{O} \big( |z - \psi|^4 \big).
	\end{flalign}

	Next we recall from \cite{PTAEP} a choice contours $\mathcal{C}$ and $\Gamma$ satisfying the conditions of Proposition \ref{asymptoticheightvertex} such that $\Real \big( G (w) - G (v) \big) < 0$ for $w \in \mathcal{C}$ and $v \in \Gamma$ both away from $\psi$. To explain these contours, it will be useful to recall properties of the level lines of  $\Real G (z) = G(\psi)$.

	\begin{prop}[{\cite[Proposition 6.7]{PTAEP}}]
		
		\label{linesgv}
		
		There exist three simple, closed curves, $\mathcal{L}_1 $, $\mathcal{L}_2 $, and $\mathcal{L}_3$, that all pass through $\psi$ and satisfy the following properties.
		
		\begin{enumerate}[leftmargin=*]
			
			\item{ \label{vzpsij} For any $z \in \mathbb{C} \setminus \{ q \}$, we have $\Real G(z) = G(\psi)$ if and only if $z \in (\mathcal{L}_1 \cup \mathcal{L}_2 \cup \mathcal{L}_3) \setminus \{ q \}$. }
			
			\item{ \label{anglev} The level lines $\mathcal{L}_1$, $\mathcal{L}_2$, and $\mathcal{L}_3$ are all star-shaped. }
			
			\item{ \label{containmentv} We have that $\mathcal{L}_1 \cap \mathcal{L}_2 = \mathcal{L}_2 \cap \mathcal{L}_3 = \mathcal{L}_1 \cap \mathcal{L}_3 = \{ \psi \}$. Furthermore, $\mathcal{L}_1 \setminus \{ \psi \}$ is contained in the interior of $\mathcal{L}_2$, and $\mathcal{L}_2 \setminus \{ \psi \}$ is contained in the interior of $\mathcal{L}_3$. }
			
			\item{ \label{qkappaq0v}   The level line $\mathcal{L}_1$ encloses $0$ but not $-q$, the level line $\mathcal{L}_2$ encloses $0$, and the level line $\mathcal{L}_3$ encloses $0$ and $-q$. Furthermore, $-q$ lies on $\mathcal{L}_2$.}
			
			\item{ \label{psianglesv} The level line $\mathcal{L}_1$ meets the positive real axis (at $\psi$) at angles $5 \pi / 6$ and $- 5 \pi /6$, the level line $\mathcal{L}_2$ meets the positive real axis (at $\psi$) at angles $\pi / 2$ and $- \pi / 2$, and the level line $\mathcal{L}_3$ meets the positive real axis (at $\psi$) at angles $\pi / 6$ and $- \pi /6$. }
			
			\item{ \label{positiverealv} For all $z$ enclosed by $\mathcal{L}_2$ but outside of $\mathcal{L}_1$, we have that $\Real \big( G(z) - G(\psi) \big) > 0$. }
			
			\item{ \label{negativerealv} For all $z$ enclosed by $\mathcal{L}_3$ but outside of $\mathcal{L}_2$, we have that $\Real \big( G(z) - G (\psi) \big) < 0$. }
			
		\end{enumerate}
		
	\end{prop}
	
	\begin{rem}
	
	\label{etacontinuous}
	
	By the continuity of $G$ and $\psi$ in $\nu$, the level lines $\mathcal{L}_1$, $\mathcal{L}_2$, and $\mathcal{L}_3$ vary uniformly continuously in $\nu \in [-(1-\varepsilon), 1 - \varepsilon]$.
	 	
	\end{rem}

	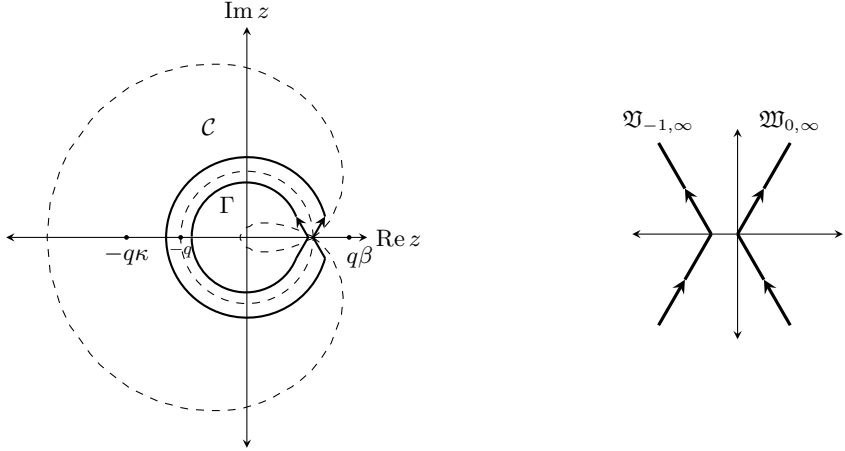
\begin{figure}
		\begin{minipage}{0.4\linewidth}
			\centering
			\begin{tikzpicture}[
				>=stealth,
				auto,
				style={
					scale = .8
				}
				]					
				\draw[<->, black	] (0, -3.5) -- (0, 3.5) node[black, above = 0] {$\Imaginary z$};
				\draw[<->, black] (-4, 0) -- (2, 0) node[black, right = 0] {$\Real z$};
				
				\draw[->,black, thick] (1.1, 0) -- (1.3, .34641);
				\draw[-,black, thick] (1.1, 0) -- (1.3, -.34641);
				\draw[black, thick] (1.3, .34641) arc (15.4:180:1.34536) node [black, above = 42, right = 10 	] {$\mathcal{C}$};
				\draw[black, thick] (1.3, -.34641) arc (-15.4:-180:1.34536);
				
				\draw[->, black, thick] (1.02, 0) -- (.82, .34641);
				\draw[-, black, thick] (1.02, 0) -- (.82, -.34641);
				\draw[black, thick] (.82, .34641) arc (21.67:180:.9)  node [black, above = 12, right = 7 	] {$\Gamma$};
				\draw[black, thick] (.82, -.34641) arc (-21.67:-180:.9);
												
				\path[draw, dashed] (1.1, 0) -- (1.08, .02) -- (1.06, .03) -- (1.03, .04) -- (1, .05) -- (.98, .06) -- (.95, .08) -- (.93, .09) -- (.89, .1) -- (.86, .12) -- (.83, .13) -- (.8, .14) -- (.75, .16) -- (.71, .17) -- (.65, .19) -- (.61, .2) -- (.56, .21) -- (.5, .22) -- (.47, .23) -- (.43, .24) -- (.38, .24) -- (.35, .24) -- (.3, .24) -- (.27, .24) -- (.22, .24) -- (.18, .23) -- (.14, .22) -- (.08, .21) -- (.06, .19) -- (.04, .18) -- (0, .17) -- (-.02, .15) -- (-.04, .13) -- (-.06, .12) -- (-.07, .1) -- (-.08, .09) -- (-.08, .08) -- (-.09, .07) -- (-.1, .04) -- (-.11, .03) -- (-.11, 0);
				
				\path[draw, dashed] (1.1, 0) -- (1.08, -.02) -- (1.06, -.03) -- (1.03, -.04) -- (1, -.05) -- (.98, -.06) -- (.95, -.08) -- (.93, -.09) -- (.89, -.1) -- (.86, -.12) -- (.83, -.13) -- (.8, -.14) -- (.75, -.16) -- (.71, -.17) -- (.65, -.19) -- (.61, -.2) -- (.56, -.21) -- (.5, -.22) -- (.47, -.23) -- (.43, -.24) -- (.38, -.24) -- (.35, -.24) -- (.3, -.24) -- (.27, -.24) -- (.22, -.24) -- (.18, -.23) -- (.14, -.22) -- (.08, -.21) -- (.06, -.19) -- (.04, -.18) -- (0, -.17) -- (-.02, -.15) -- (-.04, -.13) -- (-.06, -.12) -- (-.07, -.1) -- (-.08, -.09) -- (-.08, -.08) -- (-.09, -.07) -- (-.1, -.04) -- (-.11, -.03) -- (-.11, 0);
				
				\draw[black, dashed] (0, 0) circle [radius=1.1];
				
				\path[draw, dashed] (1.1, 0) -- (1.23, .1) -- (1.4, .3) -- (1.57, .71) -- (1.6, 1) -- (1.57, 1.3) -- (1.4, 1.76) -- (1.23, 2.02) -- (1.07, 2.22) -- (.9, 2.37) -- (.74, 2.49) -- (.57, 2.58) -- (.4, 2.68) -- (.23, 2.75) -- (.07, 2.8) -- (-.1, 2.85) -- (-.27, 2.87) -- (-.43, 2.88) -- (-.6, 2.88) -- (-.77, 2.87) -- (-.94, 2.86) -- (-1.1, 2.83) -- (-1.27, 2.78) -- (-1.44, 2.74) -- (-1.6, 2.67) -- (-1.77, 2.59) -- (-1.94, 2.51) -- (-2.1, 2.39) -- (-2.27, 2.26) -- (-2.44, 2.12) -- (-2.6, 1.94) -- (-2.77, 1.7) -- (-2.94, 1.47) -- (-3.1, 1.13) -- (-3.27, .57) -- (-3.32, 0);
				
				\path[draw, dashed] (1.1, 0) -- (1.23, -.1) -- (1.4, -.3) -- (1.57, -.71) -- (1.6, -1) -- (1.57, -1.3) -- (1.4, -1.76) -- (1.23, -2.02) -- (1.07, -2.22) -- (.9, -2.37) -- (.74, -2.49) -- (.57, -2.58) -- (.4, -2.68) -- (.23, -2.75) -- (.07, -2.8) -- (-.1, -2.85) -- (-.27, -2.87) -- (-.43, -2.88) -- (-.6, -2.88) -- (-.77, -2.87) -- (-.94, -2.86) -- (-1.1, -2.83) -- (-1.27, -2.78) -- (-1.44, -2.74) -- (-1.6, -2.67) -- (-1.77, -2.59) -- (-1.94, -2.51) -- (-2.1, -2.39) -- (-2.27, -2.26) -- (-2.44, -2.12) -- (-2.6, -1.94) -- (-2.77, -1.7) -- (-2.94, -1.47) -- (-3.1, -1.13) -- (-3.27, -.57) -- (-3.32, 0);
				
				\filldraw[fill=black, draw=black] (-1.1, 0) circle [radius=.03] node [black, below = 0, scale = .75] {$-q$};
				
				\filldraw[fill=black, draw=black] (-2, 0) circle [radius=.03] node [black, below = 0] {$-q \kappa$};
				
				\filldraw[fill=black, draw=black] (1.7, 0) circle [radius=.03] node [black, right = 4, below = 0] {$q \beta$};			
			\end{tikzpicture}		
		\end{minipage}
		\qquad
		\begin{minipage}{0.4\linewidth}		
			\centering 		
			\begin{tikzpicture}[
				>=stealth,
				scale=.35
				]
				lem
				\draw[<->] (-4, 5) -- (4, 5);
				\draw[<->] (0, 1) -- (0, 9);
				
				\draw[->,black,very thick] (2, 1.535) -- (1, 3.268);
				\draw[-,black,very thick] (1, 3.268) -- (0, 5);
				\draw[-,black,very thick]  (1, 6.732) -- (2, 8.465) node [black, above = 0] {$\mathfrak{W}_{0, \infty}$};
				\draw[->,black,very thick] (0, 5) -- (1, 6.732);
				
				\draw[->,black,very thick] (-3, 1.535) -- (-2, 3.268);
				\draw[-,black,very thick] (-1, 5) -- (-2, 3.268);
				\draw[-,black,very thick] (-2, 6.732) -- (-3, 8.465) node [black, above = 0] {$\mathfrak{V}_{-1, \infty}$};
				\draw[->,black,very thick] (-1, 5) -- (-2, 6.732);
			\end{tikzpicture}
		\end{minipage}
		
		\caption{\label{l1l2l3} To the left, the three level lines $\mathcal{L}_1$, $\mathcal{L}_2$, and $\mathcal{L}_3$ are depicted as dashed curves; the contours $\Gamma$ and $\mathcal{C}$ are depicted as solid curves and are labeled. To the right are the two contours $\mathfrak{W}_{0, \infty}$ and $\mathfrak{V}_{-1, \infty}$. }
	\end{figure}

	Now, let us explain how to select the contours $\mathcal{C}$ and $\Gamma$. They will be the unions of two contours, a ``small piecewise linear part'' near $\psi$, and a ``large curved part'' that closely follows the level line $\mathcal{L}_2$. The former are given by the following definition.

	\begin{definition}[{\cite[Definition 6.2]{PTAEP}}]
		
		\label{crgammar}
		
		For $r \in \mathbb{R}$ and $\varpi > 0$ (possibly infinite), let $\mathfrak{W}_{r, \varpi}$ denote the piecewise linear curve in the complex plane that connects $r + \varpi e^{- \pi \mathrm{i} / 3} $ to $r$ to $r + \varpi e^{\pi \mathrm{i} / 3}$. Similarly, let $\mathfrak{V}_{r, \varpi}$ denote the piecewise linear curve in the complex plane that connects $r + \varpi e^{- 2 \pi \mathrm{i} / 3}$ to $r$ to $r + \varpi^{2 \pi \mathrm{i} / 3}$. See \Cref{l1l2l3} for depictions.
		
	\end{definition}
Definition \ref{linearvertexright} and Definition \ref{curvedvertexright} define the piecewise linear and curved parts of the contours $\mathcal{C}$ and $\Gamma$, respectively, and Definition \ref{vertexrightcontours} defines the contours $\mathcal{C}$ and $\Gamma$.
	
	\begin{definition}[{\cite[Definition 6.3]{PTAEP}}]
		
		\label{linearvertexright}
		
		Let $\mathcal{C}^{(1)} = \mathfrak{W}_{\psi, \varpi}$ and $\Gamma^{(1)} = \mathfrak{V}_{\psi - \psi f_{\nu}^{-1} T^{-1 / 3}, \varpi}$, where $\varpi$ is chosen to be sufficiently small (independently of $T$ and $\nu$) so that:
		
		\begin{itemize}[leftmargin=*]
			\item{The two conjugate endpoints of $\mathcal{C}^{(1)}$ lie strictly between $\mathcal{L}_2$ and $\mathcal{L}_3$, so that their distance from $\mathcal{L}_2$ and $\mathcal{L}_3$ is bounded away from $0$, independently of $T \ge 1$ and $\nu \in [\varepsilon - 1, 1 - \varepsilon]$.}
			
			\item{The two conjugate endpoints of $\Gamma^{(1)}$ are strictly between $\mathcal{L}_1$ and $\mathcal{L}_2$, so that their distance from $\mathcal{L}_1$ and $\mathcal{L}_2$ is bounded away from $0$, independently of $T \ge 1$ and $\nu \in [\varepsilon - 1, 1 - \varepsilon]$.}
			
			\item{We have $\big| R\big( \psi^{-1} f_{\nu} (z - \psi) \big) \big| < |f_{\nu} (z - \psi) / 2 \psi|^3$, for all $z \in \mathcal{C}^{(1)} \cup \Gamma^{(1)}$, for $R$ in \eqref{r13}.}
			
			\item{We have that $|v / w| \in (q^{1 / 2}, 1)$ for all $v \in \Gamma^{(1)}$ and $w \in \mathcal{C}^{(1)}$.}
			
		\end{itemize}
Such a $\varpi>0$ is guaranteed to exist by part \ref{psianglesv} of Proposition \ref{linesgv} and \eqref{gzpsi}.
		
	\end{definition}

	\begin{definition}[{\cite[Definition 6.4]{PTAEP}}]
		
		\label{curvedvertexright}
		
		Let $\mathcal{C}^{(2)}$ denote a positively oriented contour from the top endpoint $\psi + \varpi e^{\pi \mathrm{i} / 3}$ of $\mathcal{C}^{(1)}$ to the bottom endpoint $\psi + \varpi e^{- \pi \mathrm{i} / 3}$ of $\mathcal{C}^{(1)}$, and let $\Gamma^{(2)}$ denote a positively oriented contour from the top endpoint $\psi - \psi f_{\nu}^{-1} T^{-1 / 3} + \varpi e^{2 \pi \mathrm{i} / 3}$ of $\Gamma^{(1)}$ to the bottom endpoint $\psi + \psi f_{\nu}^{-1} T^{-1 / 3} + \varpi e^{-2 \pi \mathrm{i} / 3}$ of $\Gamma^{(1)}$, satisfying:		
		\begin{itemize}[leftmargin=*]
			
			\item{The contour $\mathcal{C}^{(2)}$ remains strictly between $\mathcal{L}_2$ and $\mathcal{L}_3$, so that the distance from $\mathcal{C}^{(2)}$ to $\mathcal{L}_2$ and $\mathcal{L}_3$ remains bounded away from $0$, independently of $T \ge 1$ and $\nu \in [\varepsilon - 1, 1 - \varepsilon]$.}
			
			\item{The contour $\Gamma^{(2)}$ remains strictly between $\mathcal{L}_1$ and $\mathcal{L}_2$, so that the distance from $\mathcal{C}^{(2)}$ to $\mathcal{L}_1$ and $\mathcal{L}_2$ remains bounded away from $0$, independently of $T \ge 1$ and $\nu \in [\varepsilon - 1, 1 - \varepsilon]$.}
			
			\item{ The contour $\mathcal{C}^{(1)} \cup \mathcal{C}^{(2)}$ is star-shaped.}
			
			\item{The contour $\Gamma^{(1)} \cup \Gamma^{(2)}$ is star-shaped and does not contain $-q \kappa$.}
			
			\item{The contours $\mathcal{C}^{(2)}$ and $\Gamma^{(2)}$ are both sufficiently close to $\mathcal{L}_2$ so that the interior of $\Gamma^{(1)} \cup \Gamma^{(2)}$ encloses the image of $\mathcal{C}^{(1)} \cup \mathcal{C}^{(2)}$ under multiplication by $q$. }
			
		\end{itemize}
Such contours $\mathcal{C}^{(2)}$ and $\Gamma^{(2)}$ are guaranteed to exist by	 part \ref{anglev} and \ref{qkappaq0v} of Proposition \ref{linesgv}.
		
	\end{definition}
	
	\begin{definition}[{\cite[Definition 6.5]{PTAEP}}]
		
		\label{vertexrightcontours}
		
		Set $\mathcal{C} = \mathcal{C}^{(1)} \cup \mathcal{C}^{(2)}$ and $\Gamma = \Gamma^{(1)} \cup \Gamma^{(2)}$. Examples of the contours $\mathcal{C}$ and $\Gamma$ are depicted in Figure \ref{l1l2l3}.
		
	\end{definition}
	
	 The following lemma states that $\mathcal{C}$ and $\Gamma$ satisfy their required conditions, and that $\Real \big( G(w) - G(v) \big) < 0$ for each $w \in \mathcal{C}$ and $v \in \Gamma^{(2)}$.

	\begin{lem}[{\cite[Definition 6.6 and Lemma 6.13]{PTAEP}}]
		
		\label{rightcagammaa}
		
		The contour $\Gamma$ is positively oriented and star-shaped; it encloses $0$, but leaves outside $-q$ and $q \beta$. Furthermore, $\mathcal{C}$ is a positively oriented, star-shaped contour that is contained inside $q^{-1} \Gamma$; that encloses $0$, $-q$ and $\Gamma$; but that leaves outside $q \beta$. Moreover, there exists a positive real number $c = c (\varepsilon) > 0$, such that
		\begin{flalign*}
			\max \left\{ \sup_{\substack{w \in \mathcal{C} \\ v \in \Gamma^{(2)}}} \Real \big( G(w) - G(v) \big), \sup_{\substack{w \in \mathcal{C}^{(2)} \\ v \in \Gamma}} \Real \big( G(w) - G(v) \big) \right\} < - c.
		\end{flalign*}
	\end{lem}

	The uniformity in $\nu$ of the constant $c$ in \Cref{rightcagammaa}, was not explicitly stated in \cite{PTAEP}, but it follows from \Cref{etacontinuous}.

	\subsection{Proof of \Cref{estimatedeterminant}}
	
	\label{VertexRight}
We start by analyzing the contribution to the right side of \eqref{vptright} when $w \in \mathcal{C}^{(1)}$ and $v \in \Gamma^{(1)}$, that is, when both $w$ and $v$ are near $\psi$. To that end, define $\widetilde{K} (w, w')$ by the same formula as used to define $K(w,w')$ in \eqref{vptright}, but with the contour $\Gamma$ replaced by $\Gamma^{(1)}$.
%
Now let us change variables, a procedure that will in effect ``zoom into'' the region around $\psi$. Denote $\sigma = \psi f_{\nu}^{-1} T^{-1 / 3}$, and set
	\begin{flalign}
		\label{wwvvertexright}
		w = \psi + \sigma \widehat{w}, \qquad w' = \psi + \sigma \widehat{w}', \qquad v = \psi + \sigma \widehat{v}, \qquad \widehat{K} (\widehat{w}, \widehat{w}') = \sigma \widetilde{K} (w, w').
	\end{flalign}
Also, for any contour $\mathcal{D}$, set $\widehat{\mathcal{D}} = \sigma^{-1} \big( \mathcal{D} - \psi \big)$, where $\sigma^{-1} (\mathcal{D} - \psi)$ denotes all numbers of the form $\sigma^{-1} (z - \psi)$ with $z \in \mathcal{D}$. In particular, from Definition \ref{crgammar} and Definition \ref{linearvertexright}, we find that $\mathcal{C}^{(1)} = \mathfrak{W}_{0, \varpi / \sigma}$ and $\Gamma^{(1)} = \mathfrak{V}_{-1, \varpi / \sigma}$. The following lemma provides a bound on $\widehat{K}$; its proof is similar to that of \cite[Lemma 6.11]{PTAEP}.
	
	\begin{lem}
		
		\label{rightvgamma1tildev}
		
		There exists $c=c (\varepsilon) > 0$ such that for each $\widehat{w} \in \widehat{\mathcal{C}}^{(1)}$ and $\widehat{w}' \in \widehat{\mathcal{C}}$
		\begin{flalign*}
			\big| \widehat{K} (\widehat{w}, \widehat{w}') \big| \le \displaystyle\frac{c^{-1}}{1 + |\widehat{w}'|} \exp \big( - c |\widehat{w}|^3 - cs).
		\end{flalign*}

	\end{lem}
	
	To establish this lemma, we first rewrite the kernel $\widehat{K}$. By \eqref{wwvvertexright} and the fact that $\sigma = \psi f_{\nu}^{-1} T^{-1 / 3}$, we deduce that
	\begin{flalign}
		\label{vright111}
		\widehat{K} (\widehat{w}, \widehat{w}') & = \displaystyle\frac{1}{2 \pi \mathrm{i}}\displaystyle\int_{\widehat{\Gamma}^{(1)}} I \big( \widehat{w}, \widehat{w}'; \widehat{v} \big) d \widehat{v},
	\end{flalign}
where
	\begin{flalign}
		\begin{aligned}
			\label{vright112}
			I \big( \widehat{w}, \widehat{w}'; \widehat{v} \big) & = \displaystyle\frac{1}{\big( 1 + \psi^{-1} \sigma \widehat{v} \big) \big( \widehat{w}' - \widehat{v} \big)  } \exp \left( \displaystyle\frac{\widehat{w}^3 - \widehat{v}^3}{3} + T \big( R (T^{-1 / 3} \widehat{w}) - R (T^{-1 / 3} \widehat{v}) \big)  \right) \\
			& \qquad \times \left( \displaystyle\frac{1 + \psi^{-1} \sigma \widehat{v}}{1 + \psi^{-1} \sigma \widehat{w}} \right)^{\psi \sigma^{-1} s } \displaystyle\frac{\big(q^{-1} \beta^{-1} ( \psi + \sigma \widehat{v}); q \big)_{\infty}}{\big( q^{-1} \beta^{-1} (\psi + \sigma \widehat{w} ); q \big)_{\infty}} \\
			& \qquad \times \displaystyle\frac{\pi \psi^{-1} \sigma}{\log q} \displaystyle\sum_{j = -\infty}^{\infty}  \displaystyle\frac{1}{\sin \Big( \frac{\pi}{\log q} \big( 2 \pi \mathrm{i} j + \log (1 + \psi^{-1} \sigma \widehat{v} )  - \log (1 + \psi^{-1} \sigma \widehat{w} ) \big) \Big)}.
		\end{aligned}
	\end{flalign}
	
	Lemma \ref{rightvgamma1tildev} will follow from an estimate on $I$, given by the following lemma.
	
	\begin{lem}
		
		\label{uniformlimit13integrand}
		
There exists $c=c(\varepsilon)>0$ such that
	\begin{flalign}
			\label{uniform13integrand}
			\Big| I \big( \widehat{w}, \widehat{w}'; \widehat{v} \big) \Big| \le \displaystyle\frac{c^{-1}}{1 + |\widehat{w}'|} \exp \Big( cs \Real  \widehat{v} - c \big( |\widehat{w}|^3 + |\widehat{v}|^3 \big) \Big),
		\end{flalign}
for all $\widehat{w} \in \mathfrak{W}_{0, \varpi / \sigma}$, $\widehat{w}' \in \widehat{\mathcal{C}}$, and $\widehat{v} \in \mathfrak{V}_{-1, \varpi / \sigma}$.
		
	\end{lem}
	\begin{proof}[Proof of Lemma \ref{rightvgamma1tildev}]
This follows from \Cref{uniformlimit13integrand}, \eqref{vright111}, and the fact that $\Real \widehat{v} \le -1$  for each $\widehat{v} \in \mathfrak{V}_{-1, \varpi / \sigma}$.
	\end{proof}
	
	
	\begin{proof}[Proof of \Cref{uniformlimit13integrand}]
	
	Observe that there exist $c= c(\varepsilon) > 0$  such that the six inequalities
		\begin{flalign}
			\begin{aligned}
				\label{termsintegrand13}
				&   \left| \displaystyle\frac{1}{1 + \psi^{-1} \sigma \widehat{v}} \right| < c^{-1}, \qquad \displaystyle\frac{1}{\widehat{v} - \widehat{w}'} < \displaystyle\frac{c^{-1}}{1 + |\widehat{w}'|}, \qquad \left| \displaystyle\frac{1 + \psi^{-1} \sigma \widehat{v}}{1 + \psi^{-1} \sigma \widehat{w}} \right|^{\psi \sigma^{-1} s} \le c^{-1} \exp \big( c s \Imaginary \widehat{v} \big),  \\
				& \qquad \qquad \displaystyle\frac{\pi \psi^{-1} \sigma}{\big| \log q \big|} \displaystyle\sum_{j \ne 0} \left| \displaystyle\frac{1}{\sin \Big( \frac{\pi}{\log q} \big( 2 \pi \mathrm{i} j + \log (1 + \psi^{-1} \sigma \widehat{v} )  - \log (1 + \psi^{-1} \sigma \widehat{w} ) \big) \Big)} \right| \le c^{-1} T^{-1 / 3},\\
				& \qquad \left| \left( \displaystyle\frac{\pi \psi^{-1} \sigma}{\log q} \right) \displaystyle\frac{1}{\sin \Big( \frac{\pi}{\log q} \big( \log (1 + \psi^{-1} \sigma \widehat{v} )  - \log (1 + \psi^{-1} \sigma \widehat{w} ) \big) \Big)} \right| < c^{-1},\\
				& \qquad \qquad \left| \displaystyle\frac{\big(q^{-1} \beta^{-1} ( \psi + \sigma \widehat{v}  ); q \big)_{\infty}}{\big( q^{-1} \beta^{-1} (\psi + \sigma \widehat{w} ); q \big)_{\infty}} \right| <c^{-1} \exp \Big(c^{-1} \big( |\widehat{w}| + |\widehat{v}| \big) \Big),
			\end{aligned}
		\end{flalign}
all hold for each  $\widehat{w} \in \mathfrak{W}_{0, \varpi / \sigma}$, $\widehat{w}' \in \widehat{\mathcal{C}}$, and $\widehat{v} \in \mathfrak{V}_{-1, \varpi / \sigma}$.
		
Indeed, the first inequality holds since $v = \psi (1 + \psi^{-1} \sigma \widehat{v})$ is bounded away from $0$ for $v \in \Gamma$. The second inequality holds since $\big| \widehat{w}' - \widehat{v} \big| \ge c_1 (\Real \widehat{v} + 1)$ for $\widehat{w}' \in \widehat{\mathcal{C}}$ and $\widehat{v} \in \mathfrak{V}_{-1, \varpi / \sigma}$. The third inequality holds since $1 + \psi^{-1} \sigma \widehat{w}$ is bounded away from $0$, and since $\Real \widehat{v} < \Real \widehat{w}$, for $\widehat{w} \in \widehat{\mathcal{C}}$. The fourth inequality holds since $\sigma = \mathcal{O} \big( T^{-1 / 3} \big)$, since $\sin \big( \frac{\pi}{\log q} \big( 2 \pi \mathrm{i} j + \log (1 + \psi^{-1} \sigma \widehat{v} )  - \log (1 + \psi^{-1} \sigma \widehat{w} ) \big) \big)$ increases exponentially in $|j|$, and since that term is also bounded away from $0$ (the latter statement is true since $j$ is nonzero). The fifth inequality follows from a Taylor expansion, the fact that $v / w$ is always bounded away from any integral power of $q$, and the fact that $\big| \widehat{v} - \widehat{w} \big|^{-1} < c^{-1}$ for sufficiently small $c> 0$. The sixth inequality is true since its left side grows polynomially in $|\widehat{w}|$ and $|\widehat{v}|$, while its right side grows exponentially in these two quantities. To see this, recall that  $|\widehat{v}|$ and $|\widehat{w}|$ are at most $\varpi / \sigma$ (by the assumptions of this lemma), which in particular indicates that both the numerator and denominator on the left-hand side of the sixth inequality are bounded above (since the arguments of those Pochhammer symbols are bounded above). The concern is then that the denominator might be very small. This can only happen if $\psi$ is close to $q\beta$. Observe that $\psi < q\beta (1 - c'\sigma)$, for some $c' > 0$ (by the definition of $\psi$ and the facts that $\nu \ge 1 - 2\rho - T^{-1/3}$ and $\sigma \sim T^{-1/3}$). Thus it suffices to bound the ratio of
$$
\frac{|(1 - q^{-1}\beta^{-1} (\psi + \sigma \widehat{v}))|}{|(1 - q^{-1}\beta^{-1} (\psi + \sigma\widehat{w}))|},
$$
as the remaining parts are bounded above and below by constants. Using the previously mentioned bound $\psi < q\beta (1 - c'\sigma)$, it follows that this ratio is bounded by a polynomial in $\widehat{v}$ and $\widehat{w}$ (actually, linearly in $\widehat{v}$), as claimed.
		
The estimates \eqref{termsintegrand13} address all terms on the right side of \eqref{vright112}, except for the exponential term. To analyze this term, first recall $\big| R(z) \big| < |z|^3 / 8$, for all $z \in \widehat{\mathcal{C}}^{(1)} \cup \widehat{\Gamma}^{(1)}$; this was stipulated as the third part of Definition \ref{linearvertexright}. Thus, decreasing $c=c(\varepsilon)>0$ if necessary,
\begin{flalign}
\begin{aligned}
\label{vright10}
\big| e^{ \frac{\widehat{w}^3}{3} - \frac{\widehat{v}^3}{3} + T ( R (T^{- 1 / 3} \widehat{w}) - R (T^{- 1 / 3} \widehat{v}) ) }\big| & = e^{\frac{\widehat{w}^3}{3} - \frac{\widehat{v}^3}{3} + \frac{| \widehat{v} |^3}{8} + \frac{| \widehat{w} |^3}{8}} < c^{-1} e^{ - \frac{1}{5} ( | \widehat{w} |^3 + | \widehat{v} |^3 )} .
\end{aligned}
\end{flalign}
In \eqref{vright10}, the last estimate follows from the fact that $\widehat{w}^3 - \widehat{v}^3 < 0$, for sufficiently large $\widehat{w} \in \widehat{\mathcal{C}}^{(1)}$ and $\widehat{v} \in \widehat{\Gamma}^{(1)}$, and that it decreases cubically in $|\widehat{w}|$ and $|\widehat{v}|$ as they tend to $\infty$.
		
		Now, the estimate \eqref{uniform13integrand} follows from the definition \eqref{vright112} of $I$, the six estimates \eqref{termsintegrand13}, and the exponential estimate \eqref{vright10}.
	\end{proof}

	We next analyze the integral \eqref{vptright} defining $K (w, w')$ when either $w$ or $v$ is not close to $\psi$, that is, when either $w \in \mathcal{C}^{(2)}$ or $v \in \Gamma^{(2)}$. In this case, we will see that the integral decays exponentially in $T$. Define
	\begin{flalign*}
		\overline{K} \big( \widehat{w}, \widehat{w}' \big) = \sigma K (w, w') = \displaystyle\frac{1}{2 \pi \mathrm{i}} \displaystyle\int_{\widehat{\Gamma}} I \big( \widehat{w}, \widehat{w}'; \widehat{v} \big) d \widehat{v},
	\end{flalign*}  	
for each $\widehat{w}, \widehat{w}' \in \widehat{\mathcal{C}}$. From the change of variables \eqref{wwvvertexright}, we have that
	\begin{flalign}
		\label{rightvdeterminantkernelbar13}
		\det \big( \Id + K \big)_{L^2 (\mathcal{C})} = \det \big( \Id + \overline{K} \big)_{L^2 (\widehat{\mathcal{C}})}.
	\end{flalign}
	
	The following lemma indicates that $\big| \overline{K} - \widehat{K} \big|$ decays exponentially on the domain of $\widehat{K}$ and that $|\overline{K}|$ decays exponentially elsewhere. Although the uniformity in \Cref{kclosekbarc2small} of the dependence of $c$ and $C$ on $\nu$ was not stated directly in \cite{PTAEP}, it follows from the uniformity of the constant $c$ from \Cref{rightcagammaa} in $\nu$.
	
	\begin{cor}[{\cite[Corollary 6.14]{PTAEP}}]
		
		\label{kclosekbarc2small}
		
		There exist $c = c(\varepsilon) > 0$ so that
		\begin{flalign}
			\label{kclosekbar}
			\Big| \overline{K} \big( \widehat{w}, \widehat{w}' \big) - \widehat{K} \big( \widehat{w}, \widehat{w}' \big) \Big| < c^{-1} \exp \Big( - c \big( T + |\widehat{w}|^3 \big) \Big),
		\end{flalign}
for all $\widehat{w} \in \widehat{\mathcal{C}}^{(1)}$ and $\widehat{w}' \in \widehat{\mathcal{C}} \cup \mathfrak{W}_{0, \infty}$, and such that
		\begin{flalign}
			\label{kbarsmallc2}
			\Big| \overline{K} \big( \widehat{w}, \widehat{w}' \big) \Big| < c^{-1} \exp \Big( - c \big( T + |\widehat{w}|^3 \big) \Big),
		\end{flalign}
for all $\widehat{w} \in \widehat{\mathcal{C}}^{(2)}$ and $\widehat{w}' \in \widehat{\mathcal{C}} \cup \mathfrak{W}_{0, \infty}$.
	\end{cor}
	
	To show \Cref{estimatedeterminant}, we will make use of the following lemma, which is the $K_1 = 0$ case of \cite[Lemma A.4]{PTAEP} that approximates a Fredholm determinant with a small kernel.

	\begin{lem}[{\cite[Lemma A.4]{PTAEP}}]
		
		\label{determinantclosekernels}
		
		Adopting the notation of Definition \ref{definitiondeterminant}, we have
		\begin{flalign*}
			\Big| \det \big( \Id + K \big)_{L^2 (\mathcal{C})} - 1 \Big| \le  \displaystyle\sum_{k = 1}^{\infty} \displaystyle\frac{2^k k^{k / 2}}{(k - 1)!} \displaystyle\int_{\mathcal{C}} \cdots & \displaystyle\int_{\mathcal{C}} \displaystyle\prod_{i = 1}^k \Bigg| \displaystyle\frac{1}{k} \displaystyle\sum_{j = 1}^k \big| K (x_i, x_j) \big|^2  \Bigg|^{1 / 2}  \displaystyle\prod_{i = 1}^k dx_i.
		\end{flalign*}

	\end{lem}
	
	 Now we can use Lemma \ref{rightvgamma1tildev} and Corollary \ref{kclosekbarc2small} to establish \Cref{estimatedeterminant}.
	
	\begin{proof}[Proof of \Cref{estimatedeterminant}]
		
		By \Cref{rightvgamma1tildev} and \Cref{kclosekbarc2small},
		\begin{flalign}
		\label{k1k}
		\Big| \overline{K} \big( \widehat{w}, \widehat{w}' \big) \Big| \le c^{-1} \exp \Big( -c \big| \widehat{w} \big|^3 - c s \Big) +c^{-1} \exp \Big( -c \big| \widehat{w} \big|^3 - c T \Big)
		\end{flalign}
for some $c = c (\varepsilon) > 0$.
Thus, allowing constants to change between lines we see that 
		\begin{flalign*}
		\Big| \det \big( \Id + K^{(p)} \big)_{L^2 (\mathcal{C})} -1 \Big| & = \Big| \det \big( \Id + \overline{K} \big)_{L^2 (\widehat{\mathcal{C}})} - 1 \Big| \\
		& \le \displaystyle\sum_{k = 1}^{\infty} \displaystyle\frac{2^k k^{k / 2}}{(k - 1)!} \displaystyle\int_{\widehat{\mathcal{C}}} \cdots  \displaystyle\int_{\widehat{\mathcal{C}}} \displaystyle\prod_{i = 1}^k \Bigg| \displaystyle\frac{1}{k} \displaystyle\sum_{j = 1}^k \big| \overline{K} (x_i, x_j) \big|^2  \Bigg|^{1 / 2}  \displaystyle\prod_{i = 1}^k dx_i \\
		& \le \displaystyle\sum_{k = 1}^{\infty} \displaystyle\frac{2^k k^{k/2}}{(k-1)!} \Bigg( c^{-1} (e^{-c s} + e^{-c T}) \displaystyle\int_{\widehat{\mathcal{C}}} \exp \big( -c_1 |\widehat{w}|^3 \big) d \widehat{w} \Bigg)^k \\
		& \le \displaystyle\sum_{k = 1}^{\infty} \displaystyle\frac{16^k}{k^{k/2}} (c^{-1})^k (e^{-c s} + e^{-c T})^k \le c^{-1} (e^{-c s} + e^{-c T}),
		\end{flalign*}
from which we deduce \eqref{eta1estimate}. Here, to deduce the first statement we used \eqref{rightvdeterminantkernelbar13}, to deduce the second we used \Cref{determinantclosekernels}, to deduce the third we used \eqref{k1k}, and to deduce the fourth we used the facts that
$\int_{\widehat{\mathcal{C}}} \exp \big( -c_1 |\widehat{w}|^3 \big) d \widehat{w}$ is bounded above by a constant and that $(k-1)! \ge 2^{-k} k! \ge 8^{-k} k^k$.
	\end{proof}

\end{document}